\documentclass[11pt]{article}
\usepackage{mymacros}
\usepackage{comment}
\usepackage{mathtools}
\usepackage{color}
\usepackage{xcolor}

\usepackage{soul}

\usepackage{enumitem}

\newcommand{\EE}{\bm{E}}  
\newcommand{\PP}{\bm{P}}
\newcommand{\LL}{\bm{L}}
\newcommand{\GFF}{\mathrm{GFF}}
\newcommand{\Lv}{\mathrm{Lv}}
\newcommand{\ShG}{\mathrm{ShG}}
\newcommand{\extepsI}{I_\epsilon} 
\newcommand{\opnorm}[2]{\sup_{x\in \Omega_{#2}}\| #1(x,\cdot) \|_{L^1(\Omega_{#2})}}

\newcommand{\hdphs}{{\eta}} 

\newcommand{\tPhisLv}{\tilde \Phi_s^{\Lv_\epsilon}}
\newcommand{\sqrtpi}{} 
\newcommand{\ZDM}{{\rm Z}}
\newcommand{\scaling}{\sqrt{8\pi}}

\newcommand{\bxlen}[1]{\tilde L_{#1}}

\newcommand{\bx}[2]{ B_{#1}^{#2}}

\newcommand{\Gs}[1]{G^{#1}}
\newcommand{\bGs}[1]{\bar G^{#1}}
\newcommand{\fom}{\frac{2}{\sqrt{2\pi}}} 

\newcommand{\expmax}{A}
\newcommand{\rem}{\cR}

\newcommand{\GFFtc}{\Phi^0}
\newcommand{\GFFtd}{\Phi^\epsilon}
\newcommand{\Leb}{\mathrm{Leb}}

\date{}

\title{The Liouville model in the $L^1$ phase: coupling and extreme values}

\author{Michael Hofstetter \footnote{Department of Mathematics,
Weizmann Institute.  E-mail: {\tt michael.hofstetter@weizmann.ac.il}}  \and Ofer Zeitouni  \footnote{Department of Mathematics,
Weizmann Institute.  E-mail: {\tt ofer.zeitouni@weizmann.ac.il}} }

\begin{document}

\maketitle

\begin{abstract}
We establish a strong coupling between the Liouville model and the Gaussian free field on the two dimensional torus in the $L^1$ phase $\beta \in (0, 8\pi)$,
such that the difference of the two fields is a H\"older continuous function.
The coupling originates from a Polchinski renormalisation group approach, which was previously used to prove analogous results for other Euclidean field theories in dimension two.
Our main observations for the Liouville model are that the Polchinski flow has a definite sign
and can be controlled well thanks to an FKG argument.

The coupling allows to relate extreme values of the Liouville model and the Gaussian free field, and as an application we show that the global maximum of the Liouville field converges in distribution to a randomly shifted Gumbel distribtion.

\end{abstract}




\section{Introduction}
\subsection{Model and main results}

In this work we study the continuum Liouville model in the subcritical regime on the unit torus $\Omega=\T^2$ in dimension $d=2$.
This is a probability measure on the space of distributions $S'(\Omega)$ with density formally given by
\begin{equation}
\label{eq:lv-formal-density}
\nu^\Lv(d\phi) \propto 
\exp \Big [
- \lambda \int_{\Omega} \exp\big(\sqrt{\beta}\phi_x \big) dx
\Big]
d\nu_m^\GFF(\phi) ,
\end{equation}
where
$\beta\in (0, 8\pi)$, $\lambda >0$ and $\nu_m^\GFF$ is the continuum massive Gaussian free field on $\Omega$ with mass $m\geq 1$, i.e., the centred Gaussian field on $\T^2$ with covariance $(-\Delta + m^2)^{-1}$.

Since the Gaussian free field is a.s.\ not a regular function, but takes values in $S'(\Omega)$,
the density \eqref{eq:lv-formal-density} is ill-defined.
The standard way to give it a meaning is to introduce a suitable space regularisation $\Omega$, reinterprete the nonlinearity by a Wick-ordered version and then construct $\nu^\Lv$ as a weak limit by removing the regularisation.

In this work we use the following space discretisation.
For $\epsilon >0$ such that $1/\epsilon = 2^n$, where $n\in \N$, let $\Omega_\epsilon= \epsilon \Z^2 \cap \Omega$ be the discretised unit torus of mesh size $\epsilon$ and let $X_\epsilon= \R^{\Omega_\epsilon}$ be the set of real valued functions on $\Omega_\epsilon$.
The restriction to dyadic discretisations is inessential,
but simplifies the presentation of, and the arguments in, our work at various places. 
It is also possible to choose the discretisation such that $1/\epsilon$ is an integer.
For a function $\varphi \in X_\epsilon$ we define the discrete integral
\begin{equation}
\int_{\Omega_\epsilon} \varphi_x dx \equiv \epsilon^2 \sum_{x\in \Omega_\epsilon} \varphi_x.
\end{equation}
Then the density of the discretised Liouville measure is given by
\begin{equation}
\label{eq:lv-density}
\nu^{\Lv_\epsilon}(d\phi)
\propto
\exp{\Big[-\lambda \int_{\Omega_\epsilon} \wick{\exp(\sqrt{\beta} \phi_x)}_\epsilon dx \Big]} d\nu_m^{\GFF_\epsilon} (\phi),
\end{equation}
where $\nu_m^{\GFF_\epsilon}$ is the density of the discrete Gaussian free field on $\Omega_\epsilon$ with mass $m\geq 1$, i.e.\
the centred Gaussian field on $\Omega_\epsilon$ with covariance
\begin{equation}
\label{eq:cov-GFF}
c_{\infty}^\epsilon = (-\Delta^{\epsilon} + m^2)^{-1},
\end{equation}
with $\Delta^\epsilon f (x) = \epsilon^{-2}\sum_{y\sim x} (f(y)-f(x))$ being the lattice Laplacian acting on functions $f \in X_\epsilon$,
and the Wick ordering is
\begin{equation}
\label{eq:exp-wick-ordering}
\wick{\exp(\sqrt{\beta} \phi_x)}_\epsilon = \epsilon^{\beta/4\pi} \exp(\sqrt{\beta}\phi_x).
\end{equation}
In our setting the variance of the Gaussian free field is
\begin{equation}
\var(\phi_x) = \frac{1}{2 \pi} \log \frac{1}{\epsilon} + O_m(1), \qquad \phi \sim \nu_m^{\GFF_\epsilon},
\end{equation}
where $O_m(1)$ depends on $m$ and remains bounded as $\epsilon\to 0$.

Thus, under the measure $\nu_m^{\GFF_\epsilon}$, the integral of \eqref{eq:exp-wick-ordering} over $\Omega_\epsilon$ gives up to a multiplicative constant the $\epsilon$-regularisation of the total mass of the Gaussian multiplicative chaos of the Gaussian free field,
which is known to have a non-trivial limit as $\epsilon \to 0$ in the subcritical regime $\beta\in (0,8\pi)$.
For a proof of this result we refer to \cite{MR829798}, \cite{MR3652040}, \cite{MR3475456} and \cite{MR3274356}.

Our first result gives the convergence of the regularised measure $\nu^{\Lv_\epsilon}$ to a limiting measure $\nu^{\Lv_0}$ as $\epsilon \to 0$,
as well as a coupling to the Gaussian free field up to an additive H\"older continuous random field.
Here, we implicitly view the measure $\nu^{\Lv_\epsilon}$ as measures on $\R^\Omega$ via the isometric embedding,
which is defined above Theorem \ref{thm:lv-coupling}.
To express the H\"older regularity of the difference of the Gaussian free field we set
\begin{equation}
\label{eq:parameter-hoelder-phase}
\hdphs = 2\frac{\sqrt{\beta}}{\sqrt{2\pi}} - \frac{\beta}{4\pi} .
\end{equation}
Note that $\hdphs$ is increasing in $\beta$ on $[0,8\pi]$ with $\hdphs(0) = 0$ and $\hdphs (8\pi) = 2$,
and thus, $4\pi\hdphs$ can be seen as a reparametrisation of the full subcritical regime $(0,8\pi)$.
Moreover, we denote by $\partial$ the spatial gradient on $\Omega$ acting on functions $f\colon \Omega\to \R$.

\begin{theorem}
\label{thm:coupling-intro}
Let $\beta\in(0,8\pi)$ and $\lambda > 0$.
Then the $\epsilon$-regularised Liouville field $\Phi^{\Lv_\epsilon} \sim \nu^{\Lv_\epsilon}$
converges weakly to a
limiting field $\Phi^{\Lv}$ in $H^{-\kappa}(\T^2)$ for any $\kappa>0$ in probability.
This continuum Liouville field $\Phi^{\Lv}$ on $\T^2$ can be coupled to a continuum Gaussian free field $\Phi^{\GFF}$ on $\T^2$
such that   $\Phi^{\Delta} \coloneqq \Phi^{\Lv}-\Phi^{\GFF} \leq 0$ and satisfies the following H\"older continuity estimates:
for any $\delta>0$, there are positive and a.s.\ finite random variables $M_{\beta, \delta}$ such that
\begin{align}
\label{eq:intro-phi-delta-bounds}
\begin{split}
\max_{x\in\T^2} |\Phi^{\Delta}(x)|
+ \max_{x\in\T^2} |\partial \Phi^{\Delta} (x) | 
+
\max_{x,y\in \T^2} \frac{| \partial \Phi^\Delta(x)- \partial\Phi^\Delta(y)|}{|x-y|^{1-\hdphs -\delta}}
&\leq \lambda M_{\beta, \delta},
\qquad (0<\hdphs < 1),
%
\\
\max_{x\in \T^2} |\Phi^{\Delta}(x)|+ \max_{x,y\in \T^2}
\frac{|\Phi^\Delta(x)-\Phi^\Delta(y)|}{|x-y|^{2-\hdphs-\delta}}
&\leq \lambda M_{\beta,\delta},
\qquad (1 \leq \hdphs <2).
\end{split}
\end{align}
\end{theorem}

This statement is a consequence of the results in Section \ref{sec:coupling-lattice}.
The bounds on $\Phi^\Delta$ follow from Theomre \ref{thm:lv-coupling-bounds-lattice},
which in addition includes the analogous bounds on the lattice regularisations of $\Phi^\Delta$,
and the convergence is shown in Theorem \ref{thm:lv-phi-delta-t-limit}.


\begin{remark}
The presence of $\delta>0$ in the H\"older continuity estimates is not optimal.
Ultimately, we obtain the bounds on $\Phi^\Delta$ from a control on the maximum of the Gaussian free field along all scales up to a small error in the first order.
This result is stated in Theorem \ref{thm:lv-good-event}.
We refer to Remark \ref{rem:optimal-estimate-on-max} for more details.
Using the conjectured stronger estimate on the maximum of the Gaussian free field as input,
it is possible to obtain H\"older continuity estimates for $\partial \Phi^\Delta$ for $\hdphs<1$ and $\Phi^\Delta$ for $1\leq \hdphs<2$ as in Theorem \ref{thm:coupling-intro} for $\rho = 0$ and with an a.s.\ finite random variable $M_\beta$, which does not depend on $\rho$.
In addition, one would have a multiplicative correction of the form $1+(\log\frac{1}{|x-y|}\big)$ in the case $\hdphs = 1$.
\end{remark}

\begin{remark}
The proof of the H\"older continuity estimates in \eqref{eq:intro-phi-delta-bounds} indicates that $\Phi_0^\Delta$ has the least regularity
where the field $\Phi^\GFF$ is close to the its maximum.
%
On the other hand, one can prove that $\Phi_0^\Delta$ is much more regular in regions where the field $\Phi_0^\GFF$ is not extremal.
%
\end{remark}

The coupling of the continuum fields follows from the analogous coupling on the level of discretisations together with the convergence as $\epsilon \to 0$,
which we state in Section \ref{sec:coupling-lattice}; see Theorem \ref{thm:lv-coupling} and Theorem \ref{thm:lv-phi-delta-t-limit} for the precise statements.
This furhter allows to relate the extreme values of the discrete Gaussian free field and the discretised Liouville field.
Note that for every $\epsilon>0$
these fields are collections of finitely many a.s.\ finite random variables and thus, unlike for the continuum fields,
their global maxima are well-defined random variables,
which diverge as $\epsilon \to 0$.
%
To state the result on the global maximum of the Liouville field, we write
\begin{equation}
m_\epsilon = \frac{1}{\sqrt{2\pi}} \big( 2\log \frac{1}{\epsilon}  - \frac{3}{4} \log \log \frac{1}{\epsilon} \big)
\end{equation}
for the order of the maximum of the Gaussian free field, i.e., we have
\begin{equation}
\E \max_{\Omega_\epsilon} \Phi^{\GFF_\epsilon} = m_\epsilon + O(1), \qquad \Phi^{\GFF_\epsilon} \sim \nu_m^{\GFF_\epsilon} .
\end{equation}
where $O(1)$ denotes a function which remains bounded as $\epsilon \to 0$.
For a proof of this result we refer to \cite{MR2846636}.

\begin{theorem}
\label{thm:lv-max-convergence-to-gumbel}
The centred maximum of the $\epsilon$-regularised Liouville field $\Phi^{\Lv_\epsilon} \sim \nu^{\Lv_\epsilon}$ converges in dis\-tri\-bu\-tion  as  $\epsilon \to 0$
to a randomly shifted Gumbel distribution, i.e,
\begin{equation} 
\label{eq:lv-max-convergence-to-gumbel}
\max_{\Omega_\epsilon} \Phi^{\Lv_\epsilon} - m_\epsilon
\to \frac{1}{\sqrt{8\pi}} X + \frac{1}{\sqrt{8\pi}}\log \ZDM^{\Lv} + b,
\end{equation}
where $\ZDM^{\Lv} $ is an a.s.\ strictly positive random variable,
$X$ is an independent standard Gumbel random variable,
and $b$ is a deterministic constant.
\end{theorem}

\subsection{Literature and further directions}

This is the third article of a programme initiated by \cite{MR4399156} and subsequently continued in \cite{MR4665719},
which aims to study probabilistic properties of Euclidean field theories in dimension $d=2$ through a multiscale coupling to the Gaussian free field.
These Euclidean field theories are measures on the space of distributions $S'(\Omega)$ that can be formally expressed as
\begin{equation}
\label{eq:eqft-general-v}
\nu(d\phi) \propto \exp\Big[-\int_{\Omega} V(\phi_x) dx \Big] \nu_m^\GFF (d\phi)
\end{equation}
for suitable non-linear functions $V\colon \R\to \R$. 
The Liouville model \eqref{eq:lv-formal-density} is thus the special case, when $V$ is an exponential function.
Using a renormalisation group approach the goal is write the non-Gaussian and distribution valued field as a sum of a Gaussian free field and a continuous field $\Phi^\Delta$,
whose smoothness depends on the Euclidean field theory under consideration.

The structure of our work is in large parts similar to \cite{MR4399156},
where analogous results were established for the sine-Gordon field for $\beta<6\pi$, i.e., the model with formal density \eqref{eq:eqft-general-v} with $V(\phi) = \cos(\sqrt{\beta}\phi)$.
In the case of the Liouville model, the exponential unboundedness of the potential creates additional difficulties,
which we overcome with an application of the FKG inequality and a control of the maximum of the Gaussian free field along all scales; see Lemma \ref{lem:lv-fkg} and Theorem \ref{thm:lv-good-event} for the corresponding statements.
It is remarkable that a complete picture can also be provided for the present context in the full subcritical regime $\beta \in (0,8\pi)$.
Comparing the results on the difference field $\Phi^\Delta$ for the sine-Gordon model in \cite[Theorem 1.2]{MR4399156} and the Liouville model in Theorem \ref{thm:coupling-intro} reveals surprising similarities,
and in fact, the results are nearly identical when $\hdphs= \hdphs(\beta)$ and the a.s.\ finite constants (Liouville) are replaced by $\beta/4\pi$ and deterministically finite constants (sine-Gordon).
The particular value $\hdphs= 1$ corresponds to $\beta = (1- \frac{1}{\sqrt{2}}) 8\pi$.
Our current understanding is that the presence of the reparametrisation $\hdphs= \hdphs(\beta)$ is a characteristic of the model, and not an artefact of our method of proof.
Note that we do not give complementary lower bounds on the H\"older constants in Theorem \ref{thm:coupling-intro}.

The coupling of the Liouville field and the Gaussian free field is an essential input for the proof of Theorem \ref{thm:lv-max-convergence-to-gumbel}.
This result extends the literature on extreme values of log-correlated fields,
which is by now well understood and rather complete for the Gaussian case thanks to \cite{MR3433630} and \cite{MR3729618} as well as \cite{MR3509015} and \cite{MR3787554}.
For a survey of these Gaussian results we refer to \cite{MR4043225} and \cite{MR3526836}.

As demonstrated in \cite{Hofstetter2021Extremal} for the sine-Gordon field,
a multiscale coupling to the Gaussian free field as we achieve here can further be used to obtain convergence of the local extremal process,
and, following similar ideas, we believe it is possible to obtain the analogous result for the Liouville model.

The idea to use renormalisation group to study the extreme values of a non-Gaussian field was also used in \cite{BiskupHuang2023DGModel},
where a convergence for the centred maximum and the extremal process of the hierarchical discrete Gaussian model were proved.

The Liouville model goes back to the work \cite{MR292433} by H{\o}egh-Krohn and was subsequently studied in \cite{MR395578},
for which it is also know as H{\o}egh-Krohn model.
These works give a non-probabilistic construction of the Euclidean field theory in the regime $\beta < 4\pi$ and in finite volume, 
and it is furthermore shown that the infinite volume limit exists.

Notable probabilistic construction of a Liouville field theory were given in \cite{MR4238209} and \cite{MR4528973}, 
and subsequently in \cite{MR4324379}, \cite{MR4415393},\cite{MR4124523} and \cite{MR4054101}.
In these works the distribution valued measure is constructed as an invariant measure of a certain stochastic partial differential equation,
also known as the stochastic quantisation equation.
That such an equation is well-posed goes back to the celebrated works \cite{MR3274562} and \cite{MR3406823}.

We emphasise that the approach to construct a measure with the interpretation \eqref{eq:lv-density} is different from the method of stochastic quantisation,
and thus, does not require a well-posed stochastic partial differential equation in the first place.
Using the Polchinski renormalisation group approach developed rigorously in \cite{MR4303014},
we build the target measure \eqref{eq:lv-density} by continuously adding contributions on all scales.
This stochastic process is described by a backward stochastic differential equation with an index set having the interpretation of scale.
The solution of this stochstic differential equation is a smooth process for any $t>0$ and only distribution valued at $t=0$.

A recent construction of the Liouville measure using similar ideas was given in \cite{MR4672110} using a variational representation of the partition function of \eqref{eq:lv-formal-density},
which is essentially equivalent to the Polchinski renormalisation group approach.

A hierarchical version of the Liouville model and its relation to the sinh-Gordon model was studied by both authors in \cite{HofstetterZeitouniDecay2024},
which gives evidence for the existence of a massless Liouville model in infinite volume. 
In the present work, we do not attempt at any point to take either $m\to 0$ nor the volume to $\infty$,
however, we believe that our methods can be used to obtain, at least for a fixed mass $m>0$, a coupling in infinite volume, which locally has the properties as in Theorem \ref{thm:lv-coupling}.

We end this section with comments on possible extensions of our work and future directions. The Liouville model \eqref{eq:lv-formal-density} is closely related to the sinh-Gordon model, where $V$ in \eqref{eq:eqft-general-v} is a hyperbolic cosine.
On the level of regularisations this gives measures of the form
\begin{equation}
\label{eq:sinh-density}
\nu^{\ShG_\epsilon}(d\phi)
\propto
\exp{\Big[-\lambda \int_{\Omega_\epsilon} \wick{\cosh(\sqrt{\beta} \phi_x)}_\epsilon dx \Big]} d\nu_m^{\GFF_\epsilon} (\phi),
\end{equation}
where the Wick ordering is explained as in the case of the Liouville measure in \eqref{eq:lv-density}.
Thus, the potential of the measure \eqref{eq:sinh-density} is
\begin{equation}
\lambda \int_{\Omega_\epsilon} \wick{\cosh(\sqrt{\beta} \phi_x)}_\epsilon dx 
= \frac{\lambda}{2} \int_{\Omega_\epsilon} \epsilon^{\beta/4\pi} 
\big( e^{\sqrt{\beta}\phi_x} +   e^{-\sqrt{\beta}\phi_x} \big) dx = \frac{\lambda}{2} \big( M^{+,\epsilon}(\phi) + M^{-,\epsilon}(\phi) \big)
\end{equation}
where, under the measure $\nu_m^{\GFF_\epsilon}$, $M^{+,\epsilon}(\phi)$ and $M^{-,\epsilon}(\phi)$ are regularisations of two non-independent Gaussian multiplicative chaoses.
Conceptually, the sinh-Gordon measure is harder to analyse, and to obtain a complete picture as for the Liouville model,
we believe it is essential to have information the joint distribution of $M^{+,\epsilon}(\phi)$ and $M^{-,\epsilon}(\phi)$ under the Gaussian reference measure.
One expects that the analysis of the sinh-Gordon model is similar to that of the Liouville model.
A key estimate in our work, Lemma \ref{lem:lv-fkg}, would need to be replaced,
as its proof relies on the monotonicity of both the exponential function and its derivative,
which is no longer true for the hyperbolic cosine.

Similar difficulties arise for the critical parameter $\beta_c= 8\pi$, in which case the regularised multiplicative chaos in \eqref{eq:lv-density} should be replaced by a derivative martingale,
i.e., the regularised potential is
\begin{equation}
\label{eq:critical-v0}
v_0^{\text{crit}}(\phi) =  \epsilon^2 \sum_{x\in \Omega_\epsilon} (\frac{2}{\sqrt{2\pi}}\log \frac{1}{\epsilon} - \phi_x) e^{\sqrt{\beta_c} \phi_x - \frac{\beta_c}{4\pi} \log \frac{1}{\epsilon}} .
\end{equation}
The fact that the expectation of the limit of \eqref{eq:critical-v0} under the Gaussian free field is infinite
and that the H\"older continuity estimates on $\Phi^\Delta$ degenerate as $\beta\to \beta_c$ suggests that a rigorous construction of a critical Liouville measure is a challenging and interesting problem.


\subsection{Notation}

We use $\varphi$ and $\phi$ for deterministic fields (either on $\Omega$ or on $\Omega_\epsilon$) and denote the evaluation by $\varphi_x$.
Random fields are typically denoted by $\Phi$.
If the field also depends on a scale parameter $t\geq 0$, then we write for instance $\Phi_t$ for the field and $\Phi_t(x)$ for the evalutation of $\Phi_t$ at $x$.

We write $\EE_c$ for Gaussian expectations with covariance $c$ and denote by $\zeta$ the corresponding integration variables.

Finally, we use the standard Landau $O$-notation and write
$\lesssim$ when an estimate is up to a deterministic constant, and $\lesssim_\alpha$ to emphasise that this constant depends on a parameter $\alpha$.
Random constants are always written explicitely.

\section{Renormalised potential: convergence and continuity of the gradient}
\label{sec:dcnablav}

The coupling of the Liouville field and the Gaussian free field in Theorem \ref{thm:coupling-intro} originates from a Polchinski renormalisation group approach,
which was rigorously developed in \cite{MR4303014}.
We refer to survey aritcle \cite{MR4798104} for additional details.

To make this exposition self-contained, we briefly introduce the key notions of the Polchinski renormalisation group approach and explain how the coupling is obtained.

\subsection{Renormalised potential and coupling}
\label{ssec:vt-eps}

The general idea of renormalisation group is to subdivide a possibly complicated system into smaller subsystem according to a natural length scale,
and then perfom a suitable coarse graining on each subsystem.
From this procedure one obtains a sequence of approximations to the original system, which are indexed by scale.
The new systems are typically easier to analyse and then, one aims to obtain information of the original system by understanding its approximations.

The Polchinski renormalisation group approach can be seen as a space continuous subdivision and in the context of studying the measure \eqref{eq:lv-density}, or more generally regularised measures of the form \eqref{eq:eqft-general-v},
the natural length scale comes from a scale decomposition of the Gaussian reference measure $\nu_m^{\GFF_\epsilon}$.
For our purpose we define this decomposition as a sequence operators $(c_t^\epsilon)_{t\in [0,\infty]}$ on $X_\epsilon$ given by
\begin{equation}
\label{eq:GFF-scale-regularisation}
c_t^\epsilon = \int_0^t \dot c_s^\epsilon ds, \qquad \dot c_s^\epsilon =  e^{-s(-\Delta^\epsilon+m^2)},
\end{equation}
where $\Delta^{\epsilon}$ is the lattice Laplacian defined below \eqref{eq:cov-GFF}.
Note that $\dot c_s^\epsilon$ is the lattice heat kernel on $\Omega_\epsilon$
and that $\dot c_s^\epsilon\colon X_\epsilon \to X_\epsilon$ is a positive definite operator on $X_\epsilon$ with respect to the inner product
\begin{equation}
\label{eq:inner-product-x-eps}
\avg{f,g}_{X_\epsilon} = \epsilon^2 \sum_{\Omega_\epsilon} f(x) g(x), \qquad f,g \in X_\epsilon.
\end{equation}
Thus, for $t>0$, we have that $c_t^\epsilon$ is a positive definite operator on $X_\epsilon$,
and moreover, $c_\infty^\epsilon = (-\Delta^\epsilon + m^2)^{-1}$,
which gives $(c_t^\epsilon)_{t\in [0,\infty]}$ the interpretation of a continuous decomposition of the covariance of $\nu_m^{\GFF_\epsilon}$.
The choice of the heat kernel regularisation is essential in the present context,
as we use in a crucial way that $\dot c_s^\epsilon$ in \eqref{eq:GFF-scale-regularisation} satisfies $\dot c_s^\epsilon(x,y)\geq 0$ for $x,y\in \Omega_\epsilon$.
In particular, $\dot c_s^\epsilon$ is positivity preserving in the sense that for any $f\colon \Omega_\epsilon \to \R$ with $f(y) \geq 0$ for all $y\in \Omega_\epsilon$, we have
\begin{equation}
\big(\dot c_s^\epsilon f \big) (x) = \epsilon^2 \sum_{y\in \Omega_\epsilon} f(y) \dot c_s^\epsilon(x,y) \geq 0.
\end{equation}

In order to quantify the properties of the decomposition \eqref{eq:GFF-scale-regularisation},
we use the notation
\begin{equation}
\label{eq:heat-kernel-char-length-scale}
\theta_t = e^{-\frac{1}{2}m^2t}, 
\qquad L_t = \sqrt{t}\wedge 1/m.
\end{equation}
Note that $\theta_t$ describes the decay of the heat kernel $\dot c_t^\epsilon$ as $t\to \infty$ and $L_t$ is its characteristic length,
so that we can view $c_t^\epsilon$ as the covariances of the Gaussian free field on scales up to $L_t$.

A key object in the Polchinski renormalisation group approach
is the renormalised potential,
which arises from integrating out only the small scale Gaussian field in the partition function of the measure \eqref{eq:lv-density}.
More precisely, we write
\begin{equation}
\label{eq:v0-eps}
v_0^\epsilon(\phi) = \lambda \int_{\Omega_\epsilon} \wick{\exp(\sqrt{\beta} \phi_x)}_\epsilon dx, \qquad \lambda >0
\end{equation}
for the exponent of the density \eqref{eq:lv-density}
and define for $t>0$ the renormalised potential $v_t^\epsilon \colon X_\epsilon \to \R$ through
\begin{equation}
\label{eq:renormalised-potential}
e^{-v_t^\epsilon(\phi)} = \EE_{c_t^\epsilon} [e^{-v_0^\epsilon(\phi + \zeta)}],
\end{equation}
where $\EE_{c_t^\epsilon}$ denotes the expectation with respect to the centred Gaussian measure with covariance $c_t^\epsilon$.
We emphasise that the Polchinski approach to obtain a coupling on the level of regularisations works in greater generality,
and when applied to lattice regularisations of measures of the form \eqref{eq:eqft-general-v},
the potential $v_0^\epsilon$ is obtained from a discrete integral over the Wick-ordering of the non-linearity $V$.
For concreteness and to introduce further notation,
we assume that $v_0^\epsilon$ is as in \eqref{eq:v0-eps},
even though most of the considerations in the rest of this subsection are also valid more generally and under mild assumptions on $v_0^\epsilon$.

Understanding $v_t^\epsilon$ and its derivatives as a function of $\phi$ is the essence of the Polchinski renormalisation group approach.
A trivial upper bound on $v_t^\epsilon$ is obtained by Jensen's inequality, which in our case reads
\begin{align}
\label{eq:jensen-for-vt}
0 \leq v_t^\epsilon(\phi) 
\leq 
\EE_{c_t^\epsilon}[v_0^\epsilon(\phi+\zeta)]
&= \lambda \int_{\Omega_\epsilon} \EE_{c_t^\epsilon}
\big[ \wick{\exp(\sqrt{\beta}(\phi_x+\zeta_x)}_{\epsilon} \big]dx
\nnb
&=
\lambda \int_{\Omega_\epsilon} \epsilon^{\beta/4\pi} e^{\sqrt{\beta} \phi_x + \frac{\beta}{2} c_t^\epsilon(x,x)} dx.
\end{align}
Here, the lower bound follows from the fact that $v_0^\epsilon(\phi) \geq 0$,
which is thanks to the non-negativity of the exponential function.
It can be shown that, as $\epsilon \to 0$, we have, for $x \in \Omega_\epsilon$,
\begin{align}
\label{eq:ct-on-diagonal}
c_t^\epsilon (x,x) = \frac{1}{2\pi} \log \frac{L_t}{\epsilon} + O_m(1) ,
\end{align}
where $L_t$ is as in \eqref{eq:heat-kernel-char-length-scale} and $O_m(1)$ bounded by a constant,
which is independent of $\epsilon>0$ and $t>0$.
From the asymptotics \eqref{eq:ct-on-diagonal} we get for the integrand on the right hand side of \eqref{eq:jensen-for-vt}
\begin{equation}
\label{eq:wick-ordering-lt}
\wick{e^{\sqrt{\beta} \phi_x}}_{L_t} \equiv \epsilon^{\beta/4\pi} e^{\sqrt{\beta}\phi_x + \frac{\beta}{2} c_t^{\epsilon} (x,x)} \sim e^{\sqrt{\beta} \phi_x} L_t^{\beta/4\pi} ,
\end{equation}
which we thus interprete as a Wick-ordering of the exponential of the field at length scale $L_t$.

Apart from the renormalised potential itself, the most important object in the coupling construction is the gradient of $v_t^\epsilon$ with respect to the field $\phi$.
Using the definition of $v_t^\epsilon$ in \eqref{eq:v0-eps}, we see that
\begin{equation}
\label{eq:nabla-v-t-definition}
\big( \nabla v_t^\epsilon (\phi) \big)_x
= \partial_{\phi_x} v_t^\epsilon (\phi)
= \frac{\EE_{c_t^\epsilon}[\partial_{\phi_x} v_0^\epsilon (\phi + \zeta) e^{-v_0^\epsilon(\phi+\zeta)}  ]}{\EE_{c_t^\epsilon}[e^{-v_0^\epsilon (\phi+\zeta)}] }
=
\sqrt{\beta} \lambda \frac{\EE_{c_t^\epsilon}[\wick{e^{\sqrt{\beta} (\phi_x + \zeta_x)}}_\epsilon e^{-v_0^\epsilon(\phi+\zeta)}  ]}{\EE_{c_t^\epsilon}[e^{-v_0^\epsilon (\phi+\zeta)}] }
,  
\end{equation}
where the last step follows from \eqref{eq:v0-eps}.
We emphasise that the derivative with respect to $\phi$ is understood as a Frech\'et derivative,
i.e., with respect to the inner product on $X_\epsilon$ given in \eqref{eq:inner-product-x-eps}.
In this article the gradient of $v_t^\epsilon$ always appears together with the heat kernel $\dot c_t^\epsilon$ applied to it,
i.e., we consider the map $\dot c_t^\epsilon \nabla v_t^\epsilon  \colon X_\epsilon \to X_\epsilon$, 
\begin{equation}
\big( \dot c_t^\epsilon \nabla v_t^\epsilon(\phi)\big)_x = 
\epsilon^2 \sum_{y\in\Omega_\epsilon} \dot c_t^\epsilon(x,y) \partial_{\phi_y} v_t^\epsilon(\phi).
\end{equation}

Under mild assumptions on the potential $v_0^\epsilon$,
which can easily be verfied for the choice \eqref{eq:v0-eps}, it is shown in \cite[Proposition 2.1]{MR4303014} that the renormalised potential satisfies the Polchinski Polchinski-PDE, i.e.,
\begin{align}
\label{eq:polchinski-pde}
\partial_t v_t^\epsilon &= \frac{1}{2} \Delta_{\dot c_t^\epsilon} v_t^\epsilon - \frac{1}{2} (\nabla v_t^\epsilon)_{\dot c_t^\epsilon}^2 ,
\end{align}
where for a function $f\colon X_\epsilon \to \R$ the scale dependent differential operators in \eqref{eq:polchinski-pde} are defined by
\begin{align}
\Delta_{\dot c_t^\epsilon} f
&= \epsilon^4 \sum_{x,y} \dot c_t^\epsilon (x,y) \partial_{\phi_x} \partial_{\phi_y} f ,
\\
(\nabla f)_{\dot c_t^\epsilon}^2
&=
\epsilon^4 \sum_{x,y} \dot c_t^\epsilon (x,y) (\partial_{\phi_x}f ) ( \partial_{\phi_y} f ) .
\end{align}
Moreover, the evolution of $v_t^\epsilon$ gives rise to an inhomogenous Markov semigroup $(\PP_{s,t}^\epsilon)_{s\leq t}$ acting on functions $F\colon X_\epsilon \to \R$ and a renormalised measures $\nu_t^{\Lv_\epsilon}$ defined by
\begin{align}
\label{eq:lv-semigroup}
(\PP_{s,t}^\epsilon F) (\varphi) &= e^{v_t^\epsilon(\varphi)} \EE_{c_t^\epsilon-c_s^\epsilon} [e^{-v_s^\epsilon (\varphi + \zeta)} F(\varphi + \zeta)] ,
\\
\label{eq:lv-renormalised-measure}
\E_{\nu_t^{\Lv_\epsilon}} [F] &= (\PP_{t,\infty}^\epsilon F)(0) = e^{v_\infty^\epsilon(0)} \EE_{c_\infty^\epsilon - c_t^\epsilon} [e^{-v_t^\epsilon (\zeta)} F(\zeta)] .
\end{align}
The term \textit{renormalised measure} reflects the fact that $\nu_t^{\Lv_\epsilon}$ generalises the density \eqref{eq:lv-density} to scale $t\geq 0$.
Indeed, denoting by $\nu_t^{\GFF_\epsilon}$ the law of the centred Gaussian measure on $X_\epsilon$ with covariance $c_\infty^\epsilon - c_t^\epsilon$,
the representation \eqref{eq:lv-renormalised-measure} shows that the density of $\nu_t^{\Lv_\epsilon}$ on $X_\epsilon$ is given by
\begin{equation}
d\nu_t^{\Lv_\epsilon}(\phi) \propto e^{-v_t^\epsilon(\phi)} d\nu_t^{\GFF_\epsilon}(\phi). 
\end{equation} 

Now, using the fact that $v_t^\epsilon$ is a solution to \eqref{eq:polchinski-pde}, 
we construct in Section \ref{sec:coupling-lattice} the process $(\Phi_t^{\Lv_\epsilon})_{t\in [0,\infty]}$ with values in $X_\epsilon$,
which satisfies the backward stochastic differential equation
\begin{equation}
\label{eq:lv-polchinski-sde}
d \Phi_t^{\Lv_\epsilon} = - \dot c_t^\epsilon \nabla v_t^\epsilon(\Phi_t^{\Lv_\epsilon})dt +  (\dot c_t^\epsilon)^{1/2} dW_t^\epsilon, \qquad \Phi_\infty^{\Lv_\epsilon} =0,
\end{equation}
where $W^\epsilon$ is a Brownian motion on $\Omega_\epsilon$ and quadratic variation $t/\epsilon^2$,
with the property that $\Phi_t^{\Lv_\epsilon}\sim \nu_t^{\Lv_\epsilon}$.
In particular, the field $\Phi_0^{\Lv}$ is a realisation of the measure \eqref{eq:lv-density}.

We remark that the existence of a solution to \eqref{eq:lv-polchinski-sde} is not obvious as the general solution theory of stochastic differential equations does not apply to this case.
Instead, we use particular properties of the gradient of $v_t^\epsilon$ and the Gaussian process
\begin{equation}
\label{eq:gff-process}
\Phi_t^{\GFF_\epsilon} = \int_t^\infty (\dot c_s^\epsilon)^{1/2} dW_s^\epsilon,
\end{equation}
which drives the stochastic differential equation \eqref{eq:lv-polchinski-sde}.
In fact, as we show in Theorem \ref{thm:lv-coupling}, our argument also applies to the case $\epsilon=0$,
in which the lattice heat kernel $\dot c_s^\epsilon$ is replaced by its continnuum counterpart
\begin{equation}
\label{eq:dc0}
\dot c_t^0(x,y)
=
e^{-t (-\Delta+m^2)}(x,y) 
=
\sum_{n\in\Z^2} \frac{e^{-|x-y+L n|^2/4t-m^2 t}}{4\pi t}
,
\qquad L=1.
\end{equation}

The coupling between the Liouville field and the Gaussian free field is then obtained from \eqref{eq:lv-polchinski-sde} by formal integration on $[t,\infty]$, i.e., using the notation \eqref{eq:gff-process} we have
\begin{equation}
\label{eq:coupling-from-sde}
\Phi_t^{\Lv_\epsilon} = \Phi_t^{\Delta_\epsilon} + \Phi_t^{\GFF_\epsilon}, \qquad t\geq 0,
\end{equation}
where the difference field $\Phi_t^{\Delta_\epsilon}$ is given by
\begin{equation}
\label{eq:definition-difference-field}
\Phi_t^{\Delta_\epsilon} = - \int_t^\infty \dot c_s^\epsilon \nabla v_s^{\epsilon}(\Phi_s^{\Lv_\epsilon}) ds.
\end{equation}
In particular, the evaluation of \eqref{eq:coupling-from-sde} at $t=0$ gives a realisation of the Gaussian free field and the Liouville field on the same probability up to an additive field $\Phi_0^{\Delta_\epsilon}$.
Now, the process $\Phi^{\Delta_\epsilon}$ depends on the solution of \eqref{eq:lv-polchinski-sde} through the gradient of $v_t^\epsilon$, which is a complicated and non-explicit function.
However, as we show in Theorem \ref{thm:lv-coupling-bounds-lattice},
we are able to overcome this substantial difficulty and bound $\Phi_t^{\Delta_\epsilon}$ by terms
which depend only on the Gaussian process $\Phi^{\GFF_\epsilon}$.

In the next section, we establish estimates on the gradient of $v_t^\epsilon$,
which then serve as the main input to the proof of the statements in Section \ref{sec:coupling-lattice}.

\subsection{Estimates and convergence for the gradient of the renormalised potential}

An important and distinct characteristic of the Liouville model is the fact that the potential $v_0^\epsilon$ in \eqref{eq:v0-eps} and its derivatives with respect to the field are positive and monotonic functions.
With the identity \eqref{eq:nabla-v-t-definition} in mind, this has the following implications for the gradient of $v_t^\epsilon$.

\begin{lemma}
\label{lem:lv-fkg}
For any $\phi \in X_\epsilon$ and any $y\in \Omega_\epsilon$ we have
\begin{equation}
\label{eq:lv-fkg}
\EE_{c_t^\epsilon}[\wick{ e^{\sqrt{\beta}(\zeta_y + \phi_y)} }_{\epsilon} e^{-v_0^\epsilon(\zeta+\phi)}]
\leq 
\EE_{c_t^\epsilon}[\wick{ e^{\sqrt{\beta}(\zeta_y+\phi_y)} }_{\epsilon}]  \EE_{c_t^\epsilon}[e^{-v_0^\epsilon(\zeta+\phi)}].
\end{equation}
In particular, we have
\begin{equation}
\label{eq:lv-nablav-pointwise-upper}
0 \leq \partial_{\phi_y} v_t^\epsilon(\phi) \leq \lambda \sqrt{\beta} \wick{ e^{\sqrt{\beta}\phi_y} }_{L_t} .
\end{equation}
\end{lemma}

\begin{proof}
Note that for every $\phi \in X_\epsilon$, we have that
$\zeta \mapsto \wick{e^{\sqrt{\beta}(\zeta_y + \phi_y)}}_\epsilon$
is an increasing and  $\zeta \mapsto e^{-v_0^\epsilon(\phi+ \zeta)}$ is a decreasing function with respect to the standard partial ordering on $\R^{\Omega_\epsilon}$.
Moreover, the centred Gaussian measure with covariance $c_t^\epsilon$ has positive correlations
i.e.\ if $\zeta \sim \cN(0, c_t^\epsilon)$, then
\begin{equation}
\E[\zeta_x \zeta_y] = c_t^\epsilon(x,y) = 
\int_0^t \dot c_s^{\epsilon}(x,y) ds
\geq 0 ,
\end{equation}
where the lower bound follows from the fact that $\dot c_s^\epsilon(x,y) \geq 0$ for all $s\geq 0$.
As shown in \cite{MR665603} this implies that the centred Gaussian measure on $X_\epsilon$ with covariance matrix $c_t^\epsilon$ is associated,
and thus, \eqref{eq:lv-fkg} follows from the FKG inquality as stated for instance in \cite[Lemma A.1]{MR4515694}.

With the notation introduced in \eqref{eq:wick-ordering-lt}, we therefore obtain from \eqref{eq:nabla-v-t-definition} for $y\in \Omega_\epsilon$ and $\phi \in X_\epsilon$
\begin{equation}
0 \leq \partial_{\phi_y} v_t^\epsilon(\phi) \leq \lambda \sqrt{\beta}\EE_{c_t^\epsilon} [ e^{\sqrt{\beta}(\phi_y + \zeta_y)}  ] = \lambda \sqrt{\beta} \wick{e^{\sqrt{\beta}\phi_y}}_{L_t} ,
\end{equation}
which is the inequality \eqref{eq:lv-nablav-pointwise-upper}.
\end{proof}

The following statements give the convergence of $\dot c_t^\epsilon \nabla v_t^\epsilon$ as $\epsilon \to 0$,
as well as continuity when seen as a function on $X_\epsilon$.
The proofs of these results are provided in Section \ref{ssec:proof-of-gradient-convergence} and Section \ref{ssec:proof-of-gradient-bounds}.

Below, we denote, for $x\in \Omega$, by $x^\epsilon$ the point in $\Omega_\epsilon$,
which is closest to $x$, i.e., $x^\epsilon$ is the unique element in $\Omega_\epsilon$,
such that $x\in x^\epsilon + (-\epsilon/2, \epsilon/2]^2$.

We further need the discrete gradients with respect to $\Omega_\epsilon$,
which we denote by $\partial_\epsilon$.
More precisely, denoting by $e$ one of the $4$ unit directions in $\Z^2$,
i.e., $e \in \{ (\pm 1,0), (0, \pm 1) \} $,
we set $\partial_{\epsilon}^e f(x) = \epsilon^{-1}(f(x+\epsilon e)-f(x))$ for a function $f\colon \Omega_\epsilon \to \R$.
More generally, for $k \in \N$ we write
$\partial_\epsilon^k f$ for the matrix-valued function consisting of all iterated discrete gradients
$\partial^{e_1}_\epsilon \cdots \partial^{e_k}_\epsilon f$ where $e_1, \dots, e_k$
are unit directions in $\Z^2$.
For $\epsilon=0$, we similarly denote the true spatial derivatives by $\partial$ and $\partial^k$ respectively.

\begin{theorem}
\label{thm:lv-dcnablav-limit}
Let $\beta\in (0,8\pi)$ and $\lambda >0$.
For all $t>0$ there exist functions $\dot c_t^0 \nabla v_t^0 \colon C(\Omega) \to C^\infty(\Omega)$,
such that $\dot c_t^\epsilon \nabla v^\epsilon_t$ converges to $\dot c_t^0\nabla v_t^0$ in the sense that, for any 
$\varphi \in C(\Omega)$ and any $\varphi^\epsilon \in X_\epsilon$ such that $\sup_{x \in \Omega} |\varphi^\epsilon(x^\epsilon)-\varphi(x)| \to 0$,
\begin{equation}
\label{eq:lv-dcnablav-limit}
L_t^{2} \norm{\dot c_t^0\nabla v_t^0(\varphi)-\dot c_t^\epsilon\nabla v_t^\epsilon(\varphi^\epsilon)}_{L^\infty(\Omega_\epsilon)}
\to 0
,
\end{equation}
as $\epsilon \to 0$.
Moreover, we have, for $k\in \N_0$, $t\geq \epsilon^2$ and $\varphi \in X_\epsilon$,
\begin{align}
\label{eq:lv-dcnablav-eps-upper-bound}
0 \leq \big(\dot c_t^\epsilon \nabla v_t^\epsilon (\varphi) \big)_x
&\leq \lambda \sqrt{\beta} \int_{\Omega_\epsilon} \dot c_t^\epsilon(x,y)\wick{e^{\sqrt{\beta} \varphi_y}}_{L_t} dy,
\\
\label{eq:lv-partial-dcnablav-eps-bd}
L_t^{k} \norm{\partial_\epsilon^k\dot c_t^\epsilon \nabla v_t^\epsilon (\varphi)}_{L^\infty(\Omega_\epsilon)}
&\leq O_{\beta,k}(\theta_t \lambda L_t^{\beta/4\pi}) e^{\sqrt{\beta} \max_{\Omega_\epsilon} \varphi},
\end{align}
and similarly, for $t>0$ and $\varphi \in C(\Omega)$,
\begin{align}
\label{eq:lv-dcnablav-0-upper-bound}
0 \leq \big(\dot c_t^0 \nabla v_t^0 (\varphi) \big)_x
&\leq \lambda \sqrt{\beta} \int_{\Omega} \dot c_t^0(x,y)\wick{e^{\sqrt{\beta} \varphi_y}}_{L_t} dy ,
\\
\label{eq:dcnablav-0-bd}
L_t^{k} \norm{\partial^k\dot c_t^0\nabla v_t^0(\varphi)}_{L^\infty(\Omega)}
&\leq O_{\beta,k}(\theta_t \lambda L_t^{\beta/4\pi}) e^{\sqrt{\beta} \max_{\Omega} \varphi}.
\end{align}
\end{theorem}

\begin{theorem}
\label{thm:dcnablav-bd-cont}
Let $k\in \N_0$ and let $t\geq \epsilon^2$ and $\varphi, \varphi' \in X_\epsilon$.
Then
\begin{align}
\label{eq:dcnablav-eps-cont}
L_t^{k} \norm{\partial_\epsilon^k \dot c_t^\epsilon \nabla v_t^\epsilon(\varphi)
&-\partial_\epsilon^k \dot c_t^\epsilon\nabla v_t^\epsilon(\varphi')}_{L^\infty(\Omega_\epsilon)}
\nnb
&\leq O_{\beta,k}( \theta_t \lambda L_t^{\beta/4\pi} )
\big( e^{\sqrt{\beta} \max_{\Omega_\epsilon}\varphi} + e^{\sqrt{\beta} \max_{\Omega_\epsilon} \varphi'} \big)
 \| \varphi - \varphi' \|_{L^\infty(\Omega_\epsilon)}
 .
\end{align}
Similarly, we have for $t>0$ and $\varphi, \varphi' \in C(\Omega)$
\begin{align}
\label{eq:dcnablav-0-cont}
L_t^{k} \norm{\partial^k \dot c_t^0 \nabla v_t^0(\varphi)
&-\partial^k \dot c_t^0 \nabla v_t^0(\varphi')}_{L^\infty(\Omega)}
\nnb
&\leq O_{\beta,k}( \theta_t \lambda L_t^{\beta/4\pi} )
\big( e^{\sqrt{\beta} \max_{\Omega}\varphi} + e^{\sqrt{\beta} \max_{\Omega}\varphi'} \big)
 \| \varphi - \varphi' \|_{L^\infty(\Omega)}.
\end{align}
\end{theorem}

\begin{remark}
For $\epsilon>0$, $t\geq \epsilon^2$ and $\varphi \in X_\epsilon$ the functions $\dot c_t^\epsilon \nabla v_t^\epsilon(\varphi)$ are elements of $X_\epsilon$,
and we denote their evaluation at $x\in \Omega_\epsilon$ by
\begin{equation}
\big(\dot c_t^\epsilon \nabla v_t^\epsilon(\varphi) \big)_x \equiv \dot c_t^\epsilon \nabla v_t^\epsilon(\varphi, x) .
\end{equation}
Similarly, for $t>0$, $\varphi \in C(\Omega)$ and $x\in \Omega$, we write
\begin{equation}
\big(\dot c_t^0 \nabla v_t^0(\varphi) \big)_x \equiv \dot c_t^0 \nabla v_t^0(\varphi, x) .
\end{equation}
\end{remark}

\subsection{Gaussian multiplicative chaos for the small scale Gaussian field}

The convergence of the renormalised potential and its gradient in the previous section involve the expectation over the exponential of a centred Gaussian field with covariance $c_t^\epsilon$,
which we think of as the short scale part of the Gaussian free field.
As we show below in Lemma \ref{lem:cov-ct}, this field is log-correlated for any $t>0$,
and thus, the convergence of $v_t^\epsilon$ and its gradient are closely connected to the theory of Gaussian multiplicative chaos for log-correlated fields.
The purpose of this section is to recall standard results on this object and adapt it to our setting.
These results are a key input to the construction of the functions $\dot c_t^0 \nabla v_t^0$ above,
and hence, the proof of Theorem \ref{thm:lv-dcnablav-limit}.

In order to discuss the convergence we compare functions defined on the discretised torus $\Omega_\epsilon$ and on the continuum torus $\Omega$,
for which we use the piecewise constant extension $E_\epsilon$ defined as follows.
For $f\colon \Omega_\epsilon^n \to \R$, $n\in \N$ define $E_\epsilon f\colon \Omega^n \to \R$, 
\begin{equation}
\label{eq:def-piecewise-constant-extension}
E_\epsilon f(x_1, \ldots, x_n) = f(x_1^\epsilon, \ldots, x_n^\epsilon)
\end{equation}
with $x^\epsilon$ as above Theorem \ref{thm:lv-dcnablav-limit}.
In fact, we are only concerned with the cases $n=1$ and $n=2$.

Let now $(\GFFtd)_{\epsilon>0}$ be centred Gaussian fields on $\Omega_\epsilon$ with covariance $c_t^\epsilon$.
We extend $\GFFtd$ to $\Omega$ using the piecewise constant extension $E_\epsilon$ to obtain a field
$E_\epsilon \GFFtd$ with covariance
\begin{equation}
\cov (E_\epsilon \GFFtd) = E_\epsilon c_t^\epsilon,
\end{equation}
i.e., also the covariance of $E_\epsilon \GFFtd$ is piecewise constant.

Let further $c_t^0(x,y) = \int_0^t \dot c_s^0(x,y)ds$ with $\dot c_s^0$ being the continuum heat kernel on $\Omega$ defined in \eqref{eq:dc0}. 
We note that, for any $t>0$, $c_t^0$ is a smooth function on $\Omega\times \Omega \setminus \{ (x,x)\colon x\in \Omega \}$ and that $c_t^0(x,x) = \infty$.
Moreover, for $x,y\in \Omega$, $x\neq y$, we have
\begin{equation}
E_\epsilon c_t^\epsilon(x,y) = c_t^\epsilon(x^\epsilon, y^\epsilon) \to c_t^0(x,y)
\end{equation}
as $\epsilon \to 0$,
and, as recorded in Lemma \ref{lem:C-limit},
this convergence is uniform on compact subsets of $x\neq y$.
The following statement shows that, for every $t>0$, the field $(E_\epsilon \GFFtd)_{\epsilon \geq 0}$ is log-correlated.

\begin{lemma}
\label{lem:cov-ct}
Let $t>0$ and let $c_t^\epsilon$ be as in \eqref{eq:GFF-scale-regularisation}.
Then we have for $\epsilon >0$ and  $x,y \in \Omega$
\begin{align}
E_\epsilon c_t^\epsilon(x,x) &= \frac{1}{2\pi} \log \frac{L_t}{\epsilon} + O_m(1) ,
\\
E_\epsilon c_t^\epsilon(x,y) &= \frac{1}{2\pi}\log\frac{L_t}{|x-y|\vee \epsilon} + O_m(1) ,
\end{align}
where $O_m(1)$ denotes a constant which depends on $m$ and remains bounded as $\epsilon \to 0$ .
\end{lemma}

\begin{proof}
The first statement is exactly Lemma \ref{lem:C-limit}, specifically \eqref{e:c0-limit}.
The second statement follows from \eqref{e:cxy-limit}.
\end{proof}

For the following arguments we require the fields $\GFFtd$, and hence $E_\epsilon \GFFtd$,
to be realised on the same probability space in such a way that there exists a distribution valued random field $\GFFtc$
so that, for all $f\in L^2(\Omega)$, 
\begin{equation}
\label{eq:tgff-test-function-convergence-in-probability}
\int_\Omega E_\epsilon \GFFtd (x) f (x) dx \to \int_\Omega \Phi^0(x) f(x) dx
\end{equation}
in probability as $\epsilon \to 0$.
This can be achived by representing the fields as a random Fourier series.
Following in large parts the notation in \cite[Section 2.4]{MR4399156}, we write $\Omega^*= 2\pi \Z^2$ for the Fourier dual of $\Omega$ and denote by $\hat q^0(k)$, $k\in \Omega^*$ the Fourier multipliers of $q^0\colon \Omega \to \R$ with $\big[q^0 * q^0\big](x-y) =  c_t^0$. 
Similarly, we write $\Omega_\epsilon^*= \{k\in 2\pi \Z^2 \colon -\pi/\epsilon <k_i\leq \pi/\epsilon \}$ for the Fourier dual of $\Omega_\epsilon^*$ and $\hat q^\epsilon(k)$ for the Fourier multipliers of the function $q^\epsilon \colon \Omega_\epsilon \to \R$ with $\big[q^\epsilon* q^\epsilon\big](x-y) = c_t^\epsilon$, where here, $*$ denotes the discrete convolution on $\Omega_\epsilon$.
Since $c_t^0$ and $c_t^\epsilon$ are symmetric function,
the Fourier multipliers are non-negative,
and it can be further shown that, for $\epsilon \geq 0$ and with the convention $\Omega_0^*= \Omega^*$,
\begin{align}
\label{eq:fourier-multipliers-small-scale-gff-estimate}
0 \leq \hat q^\epsilon(k) \lesssim \frac{1}{\sqrt{|k|^2 + m^2}}, \qquad k\in \Omega_\epsilon^*
\end{align}

Now, let $\{X(k), k\in \Omega^*\}$ be a collection of independent standard complex Gaussians subject to $\bar X(k) = X(-k)$ and define, for $\epsilon >0$,
\begin{equation}
\label{eq:fourier-series-small-scales-discrete}
\GFFtd(x) = \sum_{k\in \Omega_\epsilon^*} \hat q^\epsilon(k)X(k) e^{ik\cdot x} , \qquad x\in \Omega_\epsilon,
\end{equation}
which is a centred Gaussian fields on $\Omega_\epsilon$ with covariance $c_t^\epsilon$.
Similarly, we define
\begin{equation}
\label{eq:fourier-series-small-scales-continuum}
\GFFtc(x) = \sum_{k\in \Omega*} \hat q^0(k) X(k) e^{ik\cdot x}, \qquad x\in \Omega.
\end{equation}

The following result establishes the convergence of $E_\epsilon \GFFtd$ to $\GFFtc$, and thus verfies \eqref{eq:tgff-test-function-convergence-in-probability}.
Here, we denote by $\avg{f,g}$ the pairing between functions $f,g\in L^2(\Omega)$.
We further use this notation, if $f$ is distribution valued.
\begin{lemma}
\label{lem:tgff-convergence-on-test-functions}
For any $f\in L^2(\Omega)$ we have
\begin{align}
\E\big[  \big( \avg{E_\epsilon \GFFtd, f}- \avg{\GFFtc, f}    \big)^2   \big] \to 0
\end{align}
as $\epsilon \to 0$.
\end{lemma}

\begin{proof}
In what follows, we choose the following representation of the piecewise constant embedding.
For $f\in X_\epsilon$ and $x\in \Omega$, we have
\begin{equation}
E_\epsilon f (x) = \sum_{z\in \Omega_\epsilon} f(z) \mathbf{1}_{(-\epsilon/2, \epsilon/2]^2}(x-z).
\end{equation}
Then we note that by Fubini's theorem
\begin{align}
\E\big[ \big(\avg{E_\epsilon \GFFtd,f}\big)^2 \big]
&= \E\big[
\int_{\Omega} E_\epsilon \GFFtd(x) dx
\int_{\Omega} E_\epsilon\GFFtd (y) f(y) dy
\big]
\nnb
&= \int_{\Omega} \int_{\Omega} \sum_{z\in \Omega_\epsilon} \sum_{z'\in \Omega_\epsilon} \mathbf{1}_{(-\epsilon/2, \epsilon/2]^2}(x-z) \mathbf{1}_{(-\epsilon/2, \epsilon/2]^2}(y-z') f(x) f(y)
\E[ \GFFtd(z) \GFFtd(z')  ] dx dy
\nnb
&= \int_{\Omega} \int_{\Omega} \sum_{z\in \Omega_\epsilon} \sum_{z'\in \Omega_\epsilon} \mathbf{1}_{(-\epsilon/2, \epsilon/2]^2}(x-z) \mathbf{1}_{(-\epsilon/2, \epsilon/2]^2}(y-z') f(x) f(y)
c_t^\epsilon(z,z') dx dy
\nnb
&= \int_{\Omega} \int_{\Omega} c_t^\epsilon(x^\epsilon,y^\epsilon) f(x) f(y)
 dx dy 
\to \int_{\Omega} \int_{\Omega} c_t^0(x,y)f(x) f(y) dx dy,
\end{align}
as $\epsilon \to 0$, where the last convergence follows from the uniform convergence $c_t^\epsilon(x^\epsilon,y^\epsilon) \to c_t^0(x,y)$ on compact subsets of $x\neq y$.
Similarly, we have
\begin{align}
\E\big[ \avg{E_\epsilon \GFFtd, f} \avg{\GFFtc,f} \big] 
&= \E\Big[\int_{\Omega}  \sum_{z\in \Omega_\epsilon} \mathbf{1}_{(-\epsilon/2, \epsilon/2]^2}(x-z) \GFFtd(z) f(x) dx
\int_{\Omega}\sum_{k\in \Omega^*} e^{iky} \hat q^0(k) X(k) f(y) dy
\Big]
\nnb
&=
\int_{\Omega} \int_{\Omega} \sum_{z\in \Omega_\epsilon} \mathbf{1}_{(-\epsilon/2, \epsilon/2]^2}(x-z) f(x) f(y)
\sum_{k\in \Omega_\epsilon^*}
e^{ik(z-y)} \hat q^\epsilon(k) \hat q^0(k) 
dxdy
\nnb
&\to \int_{\Omega} \int_{\Omega}  f(x) f(y) \sum_{k\in \Omega_\epsilon^*} e^{ik(x-y)}|\hat q^0(k)|^2 dx dy
= \int_{\Omega} \int_{\Omega} f(x) f(y) c_t^0(x,y) dx dy ,
\end{align}
where we used the dominated convergence theorem to get the last convergence as $\epsilon \to 0$ together with the estimates \eqref{eq:fourier-multipliers-small-scale-gff-estimate}, as well as the fact that $f\in L^2(\Omega)$.

Altogether we have, as $\epsilon \to 0$,
\begin{equation}
\E\big[  \big( \avg{E_\epsilon \GFFtd, f}- \avg{\GFFtc, f}    \big)^2   \big]
= \E \big[ \big(\avg{E_\epsilon \GFFtd,f}\big)^2   \big]
+
\E \big[ \big(\avg{\GFFtc,f}\big)^2   \big] 
- 2  \E\big[ \avg{E_\epsilon \GFFtd, f} \avg{\GFFtc,f} \big] \to 0
\end{equation}
as claimed.
\end{proof}

Note that for every $\epsilon>0$ and $x\in \Omega$, $E_\epsilon \GFFtd(x)$ is a well-defined random variable,
which allows us to define random Borel measures $M_t^\epsilon$ on $\Omega$ by
\begin{equation}
\label{eq:gmc-small-scale-approximation-piecewise-constant}
M_t^\epsilon(dx) =  e^{\sqrt{\beta}E_\epsilon\GFFtd(x) - \beta/2 \var(E_\epsilon\Phi^\epsilon(x)) } dx =  e^{\sqrt{\beta}\GFFtd(x^\epsilon) - \frac{\beta}{2} \var(\GFFtd(x^\epsilon)) } dx .
\end{equation}

The first standard result gives the convergence of $M_t^\epsilon$ as $\epsilon \to 0$ to a random Borel measure $M_t$ on $\Omega$,
which we interprete as the Gaussian multiplicative chaos associated with a Gaussian field of covariance $c_t^0$. 

\begin{lemma}
\label{lem:small-scale-gmc-convergence}
Let $\beta\in (0, 8\pi)$ and let, for $t>0$, $M_t^\epsilon$ be defined as in \eqref{eq:gmc-small-scale-approximation-piecewise-constant}.
Then there is a random Borel measure $M_t(dx)$ on $\Omega$ with $M_t(A) >0$ and $M_t(A) < \infty$ a.s.\ for any open and non-empty set $A\subseteq \Omega$,
such that for any $\psi \in C(\Omega)$
\begin{equation}
\label{eq:l1-convergence-gmc-on-psi}
\int_{\Omega} \psi d M_t^\epsilon \to \int_{\Omega} \psi d M_t
\end{equation}
in $L^1$ as $\epsilon \to 0$.
\end{lemma}


\begin{proof}
We apply \cite[Theorem 25]{MR3475456} to the covariance kernels $K^\epsilon$ and $K^0$ on $\Omega\times \Omega \setminus \{ (x,y) \colon x\neq y\}$, given by
\begin{align}
K^\epsilon (x,y) &= c_t^\epsilon(x^\epsilon,y^\epsilon) ,
\\
K^0(x,y) &= c_t^0(x,y) .
\end{align}
To verify the first assumption in this statement, we need to see that $\big( M_t^\epsilon(\Omega)\big)_{\epsilon>0}$ is uniformly integrable.
For $\beta<4\pi$ we have
\begin{equation}
\sup_{\epsilon>0} \E\big[\big(M_t^\epsilon(\Omega) \big)^2\big] <\infty
\end{equation}
from which the uniform integrability follows.
For $\beta\in [4\pi,8\pi)$ we use similar arguments as in \cite[Section 3.3.1]{BerestyckiPowell2025GFFLv} to show that $\big( M_t^\epsilon(\Omega)\big)_{\epsilon>0}$ is also uniformly integrable in this case.

The second assumption is verfied by Lemma \ref{lem:tgff-convergence-on-test-functions}.
To verify the third assumption,
we have to show that $K^\epsilon \to K^0$ in the sense that, for any $\delta >0$,
\begin{equation}
\label{eq:set-for-non-convergence-of-covariance-kernals}
(\Leb \otimes \Leb) \big( \{ (x,y)\in \Omega\times \Omega \colon |K^\epsilon(x,y) - K^0(x,y)   | \geq \delta  \} \big) \to 0
\end{equation}
as $\epsilon \to 0$,
where $\Leb$ denotes the Lebesgue measure on $\Omega$.
To see \eqref{eq:set-for-non-convergence-of-covariance-kernals}, 
we first denote $D= \{ (x_1, x_2) \in \Omega\times \Omega \colon x_1 = x_2 \}$ and $D_\rho = \{ (x_1, x_2) \in \Omega\times \Omega \colon \dist\big((x_1,x_2), D  \big) < \rho \}$,
where $\dist$ denotes the Euclidean metric on $\Omega$.
Note that there is a constant $c>0$ such that, for any $\rho>0$,
\begin{equation}
\Big(\Leb \otimes \Leb \Big) (D_\rho) \leq c \rho^2.
\end{equation}
Then we have
\begin{align}
\label{eq:lebesgue-not-convergence-set}
(\Leb \otimes \Leb )&\big( \{ (x,y)\in \Omega\times \Omega\colon |K^\epsilon(x,y) - K^0(x,y)   |\geq \delta  \} \big)
\nnb 
&\leq c\rho^2 + (\Leb \otimes \Leb) \big( \{ (x,y) \in D_\rho^c \colon |K^\epsilon(x,y) - K^0(x,y)   | \geq \delta  \} \big).
\end{align}
Now, $D_\rho^c$ is a compact subset of $D^c$,
and thus, we have by Lemma \ref{lem:C-limit}
\begin{equation}
\sup_{(x,y) \in D_\rho^c}| K^\epsilon(x,y) - K^0(x,y) | \to 0
\end{equation}
as $\epsilon \to 0$.
Thus, we have from \eqref{eq:lebesgue-not-convergence-set}
\begin{equation}
\limsup_{\epsilon\to 0} (\Leb \otimes \Leb) \big( \{ (x,y)\in \Omega\times \Omega\colon |K^\epsilon(x,y) - K^0(x,y)   |\geq \delta  \} \big)
\leq c\rho^2.
\end{equation}
Since $\rho>0$ was arbitrary, the claimed convergence \eqref{eq:set-for-non-convergence-of-covariance-kernals} follows.
\end{proof}

We also need the following auxiliary result for the multiplicative chaos $M_t$ integrated against functions that depend smoothly on a parameter.
To emphasise with respect to which variable the function is integrated, we write, for $f\in C(\Omega)$,
\begin{equation}
M_t(f) = \int_{\Omega} f(y) dM_t(y).
\end{equation}

\begin{lemma}
\label{lem:gmc-differentiable}
Let $h\colon \Omega\times \Omega\to \R$,
such that $ h\in C^{\infty}(\Omega\times \Omega)$.
Then, $\int_{\Omega} h(\cdot, y) dM_t(y) \in C^\infty(\Omega)$ a.s.\ and for any $k\in \N$
\begin{equation}
\label{eq:gmc-differentiable}
\partial^k \Big( \int_{\Omega} h(\cdot,y) dM_t(y)\Big)
= \int_{\Omega} \partial^k h(\cdot,y) dM_t(y) \qquad \text{a.s.}
\end{equation}
\end{lemma}

\begin{proof}
In what follows, we use the convention that the derivatives $\partial^k$ are with respect to the first argument of $h$.
Then, since $\partial^k h(\cdot,y) \in C^\infty(\Omega)$, we have
\begin{equation}
|\partial^k h(\cdot, y) |
\leq \sup_{x,y\in \Omega} | \partial^k h(x,y) | < \infty ,
\end{equation}
and the last disply is integrable for a.e.\ configuration of $M_t$. 
Thus, \eqref{eq:gmc-differentiable} follows from differentiation under the integral sign.
\end{proof}

With the convergence of the Gaussian multiplicative chaos of $\GFFtd$ at hand,
we now state and prove the convergence of the renormalised potential.
Here, we write, for any $\psi \in C(\Omega)$,
\begin{equation}
\label{eq:gmc-measure-on-test-function}
M_t(\psi) = \int_{\Omega} \psi dM_t ,
\end{equation}
and define, for $\varphi \in C(\Omega)$,
\begin{equation}
\label{eq:vt-continuum-definition}
v_t^0 (\varphi) = - \log \E[e^{-\lambda M_t(e^{\sqrt{\beta}\varphi})}]. 
\end{equation}
where $\lambda>0$ is as in \eqref{eq:lv-density}.

\begin{lemma}
\label{lem:vt-convergence}
Let $t>0$ and let $v_t^\epsilon$ be as in \eqref{eq:renormalised-potential}.
Then for any $\varphi\in C(\Omega)$ and $\varphi^\epsilon \in \Omega_\epsilon$ with $\sup_{\Omega} |\varphi^\epsilon(x^\epsilon) - \varphi(x)| \to 0$,
we have
\begin{equation}
v_t^\epsilon (\varphi^\epsilon) \to v_t^0(\varphi).
\end{equation}
\end{lemma}

\begin{proof}
In light of the definition of the renormalised potential
we have to show that
\begin{equation}
\label{eq:convergence-renormalised-potential}
\EE_{c_t^\epsilon}[e^{-v_0^\epsilon(\zeta+\varphi^\epsilon)}] = 
\E[e^{-v_0^\epsilon(\GFFtd + \varphi^\epsilon)}] \to 
\E[e^{-\lambda M_t(e^{\sqrt{\beta}\varphi})}].
\end{equation}
Since $v_0^\epsilon(\GFFtd + \varphi^\epsilon)\geq 0$, we have $e^{-v_0^\epsilon(\GFFtd + \varphi^\epsilon)} \leq 1$ for all $\epsilon>0$.
Thus, to prove the convergence \eqref{eq:convergence-renormalised-potential},
it suffices to show that
\begin{equation}
\label{eq:convergence-potential-v0}
v_0^\epsilon(\GFFtd + \varphi^\epsilon)
=
\lambda M_t^\epsilon(e^{\sqrt{\beta}\varphi^\epsilon})
\to 
\lambda M_t(e^{\sqrt{\beta}\varphi})
\end{equation}
in $L^1$ as $\epsilon\to 0$.
To prove this convergence we first write
\begin{align}
\label{eq:v0-gmc-convergence-triangle-inequality}
\big|v_0^\epsilon(\GFFtd + \varphi^\epsilon)
- \lambda M_t(e^{\sqrt{\beta}\varphi})
\big|
&\leq 
\big|v_0^\epsilon(\GFFtd + \varphi^\epsilon)
- 
v_0^\epsilon(\GFFtd + \varphi)
\big|
+
\big| 
v_0^\epsilon(\GFFtd + \varphi)
- \lambda M_t(e^{\sqrt{\beta}\varphi}) \big|
\nnb 
& = 
\big|
\epsilon^2 \sum_{x\in \Omega_\epsilon} 
\big( e^{\sqrt{\beta} \varphi^\epsilon(x)} - e^{\sqrt{\beta} \varphi(x)} \big)
e^{\sqrt{\beta}\GFFtd(x) - \frac{\beta}{2} \var(\GFFtd(x)) }
\big| \nnb
&\qquad\qquad\qquad +
\big|
\lambda M_t^\epsilon (e^{\sqrt{\beta}\varphi}) - \lambda M_t (e^{\sqrt{\beta}\varphi})
\big|
\nnb
&\leq \max (e^{\|\varphi\|_{L^\infty(\Omega_\epsilon)}}, e^{\|\varphi^\epsilon\|_{L^\infty(\Omega_\epsilon)}}) (e^{\sqrt{\beta} \| \varphi^\epsilon - \varphi\|_{L^\infty(\Omega_\epsilon)}} -1) M_t^\epsilon(\Omega)
\nnb
&\qquad \qquad \qquad + 
\lambda \big| M_t^\epsilon(e^{\sqrt{\beta}\varphi}) - M_t(e^{\sqrt{\beta}\varphi})\big| ,
\end{align}
where we used that
\begin{equation}
\|e^{\sqrt{\beta} \varphi^\epsilon(x)} - e^{\sqrt{\beta} \varphi(x)} \|_{L^\infty(\Omega_\epsilon)}
\leq \max (e^{\|\varphi\|_{L^\infty(\Omega_\epsilon)}}, e^{\|\varphi^\epsilon\|_{L^\infty(\Omega_\epsilon)}}) (e^{\sqrt{\beta} \| \varphi^\epsilon - \varphi\|_{L^\infty(\Omega_\epsilon)}} -1).
\end{equation}
Now, the second term on the right hand side of \eqref{eq:v0-gmc-convergence-triangle-inequality} converges to $0$ in $L^1$ as $\epsilon \to 0$ by Lemma \ref{lem:small-scale-gmc-convergence} applied with $\psi = e^{\sqrt{\beta}\varphi}$.
To see that also the first term converges to $0$ in $L^1$, we note that 
\begin{equation}
\sup_{\epsilon>0} \E\big[ M_t^\epsilon(\Omega) \big] \leq 1.
\end{equation}
By the assumption on $\varphi^\epsilon$ and $\varphi$, we then have
\begin{equation}
\| \varphi^\epsilon - \varphi\|_{L^\infty(\Omega_\epsilon)}
= \sup_{x\in \Omega} |\varphi(x^\epsilon)- \varphi^\epsilon(x^\epsilon)  |
\leq \sup_{x\in \Omega} |\varphi(x) - \varphi^\epsilon(x^\epsilon)| \to 0
\end{equation}
as $\epsilon \to 0$,
from which \eqref{eq:convergence-potential-v0} follows.
\end{proof}

\subsection{Convergence of the gradient: Proof of Theorem \ref{thm:lv-dcnablav-limit}}
\label{ssec:proof-of-gradient-convergence}

In this section we prove the convergence of the functions $\dot c_t^\epsilon \nabla v_t^\epsilon \colon X_\epsilon \to X_\epsilon$
with the limit given by the ratio of expectation values that involve the multiplicative chaos $M_t$ tested against suitable functions.

\begin{proof}[Proof of Theorem \ref{thm:lv-dcnablav-limit}]
We first recall that,
for $y\in \Omega_\epsilon$ and $\varphi^\epsilon\in X_\epsilon$,
we have
\begin{equation}
\partial_{\phi_y} v_t^\epsilon(\varphi^\epsilon)
= \big(\nabla v_t^\epsilon(\varphi^\epsilon)\big)_y
= \lambda \sqrt{\beta} 
\frac{
\EE_{c_t^\epsilon}[\wick{e^{\sqrt{\beta}(\varphi_y^\epsilon+\zeta_y)}}_\epsilon e^{-v_0^\epsilon(\varphi^\epsilon+\zeta) }]
}{
\EE_{c_t^\epsilon}[e^{-v_0^\epsilon(\varphi^\epsilon+\zeta)}]
}
,
\end{equation}
and thus, for $x\in \Omega_\epsilon$ and with the notation introduced in \eqref{eq:gmc-small-scale-approximation-piecewise-constant}, we have
\begin{align}
\label{eq:dcnablav-eps-as-gmc}
 \dot c_t^\epsilon \nabla v_t^\epsilon (\varphi^\epsilon,x) 
&= \frac{ \lambda \sqrt{\beta}}{\EE_{c_t^\epsilon}[e^{-v_0^\epsilon(\varphi^\epsilon + \zeta)}]}
\EE_{c_t^\epsilon}\Big[
\Big(\int_{\Omega_\epsilon} \dot c_t^\epsilon(x,y)  \wick{e^{\sqrt{\beta}(\varphi_y^\epsilon + \zeta_y)} }_\epsilon dy\Big)
e^{-v_0^\epsilon(\varphi^\epsilon + \zeta)} \Big] 
\nnb
&=
\frac{\lambda\sqrt{\beta}}{\E[e^{-\lambda M_t^\epsilon(e^{\sqrt{\beta}\varphi^\epsilon})
}]}
\E\Big[
M_t^\epsilon\big(
\dot c_t^\epsilon(x, \cdot) e^{\sqrt{\beta} \varphi^\epsilon}
\big)
e^{-\lambda M_t^\epsilon(e^{\sqrt{\beta}\varphi^\epsilon})
}\Big]  .
\end{align}
By Lemma \ref{lem:vt-convergence}, we have, as $\epsilon \to 0$,
\begin{equation}
\label{eq:exp-vt-convergence}
\E[e^{-\lambda M_t^\epsilon(e^{\sqrt{\beta}\varphi^\epsilon})
}]
\to e^{-v_t^0(\varphi)}
\end{equation}
with $v_t^0(\varphi)$ defined in \eqref{eq:vt-continuum-definition}.
Thus, to obtain the existence of functions $\dot c_t^0\nabla v_t^0$ it remains to establish the convergence of the second expectation value on the right hand side of \eqref{eq:dcnablav-eps-as-gmc}.
More precisely, we show that, as $\epsilon \to 0$,
\begin{equation}
\label{eq:convergence-second-factor-dc0nablav0}
\E \big[ M_t^\epsilon \big(
\dot c_t^\epsilon(x, \cdot) e^{\sqrt{\beta} \varphi^\epsilon}
\big)
e^{-\lambda M_t^\epsilon(e^{\sqrt{\beta}\varphi^\epsilon})}
\big]
\to 
\E \big[ M_t\big(
\dot c_t^0(x, \cdot) e^{\sqrt{\beta} \varphi}
\big)
e^{-\lambda M_t(e^{\sqrt{\beta}\varphi}) }
\big] .
\end{equation}
To see this, we first note that 
\begin{align}
0 \leq 
M_t^\epsilon\big(
\dot c_t^\epsilon(x, \cdot) e^{\sqrt{\beta} \varphi^\epsilon}
\big)
\leq \|\dot c_t^\epsilon(x,\cdot) \|_{L^\infty(\Omega_\epsilon)} \|e^{\sqrt{\beta}\varphi^\epsilon} \|_{L^\infty(\Omega_\epsilon)} M_t^\epsilon (\Omega)
\end{align}
as well as
\begin{equation}
e^{-\lambda M_t^\epsilon(e^{\sqrt{\beta}\varphi^\epsilon})}
\leq e^{-\lambda M_t^\epsilon(\Omega)
e^{\sqrt{\beta}\min_{\Omega_\epsilon} \varphi^\epsilon}}
\leq 
e^{-\lambda M_t^\epsilon(\Omega) e^{-\sqrt{\beta}\|\varphi^\epsilon\|_{L^\infty(\Omega_\epsilon) }} }
.
\end{equation}
The assumptions on $\varphi^\epsilon$, $\varphi$ guarantee that there is $\epsilon_0>0$ depending on $\varphi^\epsilon$ and $\varphi$ such that 
\begin{equation}
\label{eq:linfty-phi-eps-upper-bound}
\sup_{\epsilon_0 > \epsilon>0} \|\varphi^\epsilon\|_{L^\infty(\Omega_\epsilon)} \leq c_0
\end{equation} 
for a constant $c_0>0$ depending only on $\varphi$.
Indeed, we first observe that
\begin{equation}
\label{eq:upper-bound-varphi-eps-l-infty}
\| \varphi^\epsilon \|_{L^\infty(\Omega_\epsilon)} \leq \sup_{x\in\Omega} | \varphi^\epsilon(x^\epsilon) - \varphi(x) | + \| \varphi \|_{L^\infty(\Omega)}
\end{equation}
and thus, since $\sup_{x\in\Omega} | \varphi^\epsilon(x^\epsilon) - \varphi(x) | \to 0$ as $\epsilon \to 0$,
we can choose $\epsilon_0>0$,
such that the first term on the right hand side of \eqref{eq:upper-bound-varphi-eps-l-infty} is bounded by $1$ for all $0<\epsilon<\epsilon_0$, which gives $c_0 = 1 + \|\varphi\|_{L^\infty(\Omega)}$.
Thus, using that 
\begin{equation}
\label{eq:mexp-m-upper-bound}
M_t^\epsilon(\Omega) e^{-\lambda M_t^\epsilon(\Omega)e^{-\sqrt{\beta}c_0} } \leq C_0
\end{equation}
for some non-random constant $C_0>0$ depending only on $c_0>0$ and $\lambda>0$, 
we have, for all $\epsilon<\epsilon_0$,
\begin{equation}
\label{eq:upper-bound-dcnablavt-uniform}
0 \leq M_t^\epsilon\big(
\dot c_t^\epsilon(x, \cdot) e^{\sqrt{\beta} \varphi^\epsilon}
\big)
e^{-\lambda M_t^\epsilon(e^{\sqrt{\beta}\varphi^\epsilon})}
\leq 
 C_0 \sup_{\epsilon>0} \Big( \|\dot c_t^\epsilon(x,\cdot) \|_{L^\infty(\Omega_\epsilon)} \|e^{\sqrt{\beta}\varphi^\epsilon} \|_{L^\infty(\Omega_\epsilon)}\Big) 
 \equiv \bar C_0 
< \infty,
\end{equation}
which shows that the function in the expectation on the right hand side of \eqref{eq:dcnablav-eps-as-gmc} are bounded by a non-random constant
uniformly in $\epsilon >0$.
By Lemma \ref{lem:pttorus}, specifically \eqref{e:pttoruslimit},
we have, for $t>0$ and $x\in \Omega$ as $\epsilon \to 0$,
\begin{equation}
\sup_{x,y\in \Omega} |\dot c_t^\epsilon(x^\epsilon,y^\epsilon) - \dot c_t^0(x,y) | \to 0 ,
\end{equation}
where $\dot c_t^0$ is the continuous heat kernel defined in \eqref{eq:dc0},
as well as
\begin{equation}
\sup_{y\in \Omega}
|e^{\sqrt{\beta} \varphi^\epsilon(y^\epsilon)} - e^{\sqrt{\beta}\varphi(y)}|
\leq 
\max (e^{\sqrt{\beta}\|\varphi\|_{L^\infty(\Omega_\epsilon)}}, e^{\sqrt{\beta}\|\varphi^\epsilon\|_{L^\infty(\Omega_\epsilon)}}) (e^{\sqrt{\beta} \| \varphi^\epsilon - \varphi\|_{L^\infty(\Omega_\epsilon)}} -1)
\to 0 .
\end{equation} 
Thus, for any fixed $x\in\Omega$, the functions $F_x^\epsilon \colon \Omega_\epsilon \to \R, \, F_x^\epsilon(y) =  \dot c_t^\epsilon(x^\epsilon,y) e^{\sqrt{\beta} \varphi^\epsilon(y)}$ converge to the function $F_x^0 \colon \Omega\to \R, \, F_x^0(y) = \dot c_t^0 (x,y) e^{\sqrt{\beta}\varphi(y)}$ in the sense that
\begin{equation}
\sup_{y \in \Omega}| F_x^\epsilon (y^\epsilon) - F_x^0(y)| \to 0
\end{equation}
as $\epsilon\to 0$.
Since $\dot c_t^0(x,\cdot) \in C^\infty(\Omega)$ and $\varphi\in C(\Omega)$, it follows that $F_x^0 \in C(\Omega)$.
Therefore, using the notation introduced in \eqref{eq:gmc-measure-on-test-function}, 
a similar estimate as in \eqref{eq:v0-gmc-convergence-triangle-inequality} gives that
\begin{equation}
\label{eq:convergence-dctvtdmt}
\int_{\Omega_\epsilon} F_x^\epsilon(y) M_t^\epsilon(dy) \to M_t(F_x^0)
\end{equation}
in $L^1$ as $\epsilon \to 0$,
which in particular implies \eqref{eq:convergence-second-factor-dc0nablav0}.
The limit in \eqref{eq:convergence-dctvtdmt} now permits us to define 
\begin{equation}
[\dot c_t^0 \nabla v_t^0](\varphi,x)
=
\big(\dot c_t^0 \nabla v_t^0(\varphi)\big)_x
\equiv e^{v_t^0(\varphi)} \E[ M_t(F_x^0)  e^{-\lambda M_t(e^{\sqrt{\beta}\varphi}) }].
\end{equation}
In fact, the same argument with $\partial_\epsilon \dot c_t^\epsilon$ in place of $\dot c_t^\epsilon$ gives 
the existence of functions $\partial^k \dot c_t^0 \nabla v_t^0$ defined by
\begin{equation}
\label{eq:dkctnablavt-0-definition}
[\partial^k \dot c_t^0 \nabla v_t^0 ](\varphi,x) \equiv
\big( \partial^k \dot c_t^0 \nabla v_t^0(\varphi) \big)_x
=  e^{v_t^0(\varphi)} \E[M_t\big( (\partial^k F^0)_x \big) e^{-\lambda M_t(e^{\sqrt{\beta} \varphi})} ],
\end{equation}
where 
$ (\partial^k F^0 )_x(y) = (\partial^k \dot c_t^0) (x,y) e^{\sqrt{\beta} \varphi(y)}$.

We now verify that, for all $\varphi \in C(\Omega)$,
\begin{equation}
\label{eq:derivatives-dcnablav0}
\partial^k \big(\dot c_t^0 \nabla v_t^0(\varphi)\big) =  (\partial^k \dot c_t^0 \nabla v_t^0)(\varphi),
\end{equation}
which then implies that $\dot c_t^0 \nabla v_t^0 (\varphi) \in C^\infty(\Omega)$.
To this end, we first note that
\begin{equation}
\partial^k \big( \dot c_t^0 \nabla v_t^0 (\varphi) \big)
= 
\frac{\lambda \sqrt{\beta}}{\E[e^{-\lambda M_t(e^{\sqrt{\beta} \varphi} ) }]}
\partial^k \Big(\E[ M_t (F_x^0) e^{-\lambda M_t(e^{\sqrt{\beta}\varphi})}] \big)
,
\end{equation}
and thus, the statement follows from differentiation under the integral sign provided that 
\begin{enumerate}
\item the expectation of $M_t(F_x^0)e^{-\lambda M_t(e^{\sqrt{\beta} \varphi})}$ is finite for every $x\in \Omega$,
\item the derivative $\partial^k \big( M_t (F_x^0) e^{-\lambda M_t(e^{\sqrt{\beta} \varphi})}\big)$ exists for all $x\in \Omega$ a.s.,
\item there is a positive random variable $Z$ with $\E[Z] < \infty$ and $| \partial^k M_t (F_x^0) e^{-\lambda M_t(e^{\sqrt{\beta} \varphi})} | \leq Z$.
\end{enumerate}
The statement (i) follows from 
\begin{equation}
| M_t(F_x^0)e^{-\lambda M_t(e^{\sqrt{\beta} \varphi})} |
\leq 
M_t(\Omega) \| \dot c_t^0(x,\cdot) \|_{L^{\infty}(\Omega)}e^{\sqrt{\beta}\max \varphi},
\end{equation}
and the right hand side is integrable, since $\E[M_t(\Omega)] < \infty$.
To see (ii), we use Lemma \ref{lem:gmc-differentiable}, which in particular gives
\begin{equation}
\partial^k M_t(F_x^0)
= M_t ( (\partial^k F^0)_x ) \qquad \text{a.s.,}
\end{equation}
where $(\partial^k F)_x$ is as below \eqref{eq:dkctnablavt-0-definition}.
Finally, (iii) follows from
\begin{equation}
| \partial^k M_t(F_x^0)e^{-\lambda M_t(e^{\sqrt{\beta} \varphi})} |
\leq 
M_t(\Omega) \| \partial^k \dot c_t^0(x,\cdot) \|_{L^{\infty}(\Omega)}e^{\sqrt{\beta}\max \varphi} \equiv Z,
\end{equation}
which completes the proof of \eqref{eq:derivatives-dcnablav0}.

With the existence of the limiting function $\dot c_t^0 \nabla v_t^0$ at hand, we now prove the convergence \eqref{eq:lv-dcnablav-limit},
i.e., we show that, as $\epsilon \to 0$,
\begin{align}
\label{eq:gradient-convergence-gmc-formulation}
\Big \|
\frac{\E\Big[
M_t^\epsilon\big(
\dot c_t^\epsilon(x, \cdot) e^{\sqrt{\beta} \varphi^\epsilon}
\big)
e^{-\lambda M_t^\epsilon(e^{\sqrt{\beta}\varphi^\epsilon})
}\Big] }{\E[e^{-\lambda M_t^\epsilon(e^{\sqrt{\beta}\varphi^\epsilon})
}]}
-
\frac{\E\Big[
M_t\big(
\dot c_t^0(x, \cdot) e^{\sqrt{\beta} \varphi}
\big)
e^{-\lambda M_t(e^{\sqrt{\beta}\varphi})
}\Big] }{\E[e^{-\lambda M_t(e^{\sqrt{\beta}\varphi})
}]}
\Big\|_{L^\infty(\Omega_\epsilon)}
\to 0.
\end{align}
By the convergence \eqref{eq:exp-vt-convergence} and the fact that $\E[e^{-\lambda M_t^\epsilon(e^{\sqrt{\beta}\varphi^\epsilon})
}]\geq C_\varphi$ for some constant $C_\varphi$ only depending on $\varphi$,
which holds by Jensen's inequality,
it suffices to show that, as $\epsilon \to 0$,
\begin{align}
\big\| \E \big[ M_t^\epsilon (F_x^\epsilon) e^{-\lambda M_t^\epsilon(e^{\sqrt{\beta} \varphi^\epsilon})}
\big]
-
\E\big[
M_t(F_x^0) e^{-\lambda M_t(e^{\sqrt{\beta} \varphi})}
\big]
\big\|_{L^\infty(\Omega_\epsilon)}
\to 0 .
\end{align}
To this end, we write
\begin{align}
\label{eq:dctnablavt-convergence-with-gmcs-triangle}
\big \| \E \big[ M_t^\epsilon (F_x^\epsilon) &e^{-\lambda M_t^\epsilon(e^{\sqrt{\beta} \varphi^\epsilon})} \big]
-
\E \big[ M_t(F_x^0) e^{-\lambda M_t(e^{\sqrt{\beta} \varphi})} \big]
\big\|_{L^\infty(\Omega_\epsilon)} 
\nnb
&\leq \big\| \E \big[ M_t^\epsilon (F_x^\epsilon)e^{-\lambda M_t^\epsilon(e^{\sqrt{\beta} \varphi^\epsilon})} \big]
 - \E \big[ M_t^\epsilon(F_x^0 ) e^{-\lambda M_t^\epsilon(e^{\sqrt{\beta} \varphi^\epsilon})} \big] \big\|_{L^\infty(\Omega_\epsilon)}
\nnb
&\qquad +
\big\| \E \big[ M_t^\epsilon(F_x^0) e^{-\lambda M_t^\epsilon(e^{\sqrt{\beta} \varphi^\epsilon})} \big] -  \E\big[ M_t(F_x^0) e^{-\lambda M_t(e^{\sqrt{\beta} \varphi})}
 \big]\big\|_{L^\infty(\Omega)}
\end{align}
and show that both terms on the last right hand side converge to $0$ as $\epsilon \to 0$.
For the first term, we have
\begin{align}
\big\| \E \big[ M_t^\epsilon (F_x^\epsilon)e^{-\lambda M_t^\epsilon(e^{\sqrt{\beta} \varphi^\epsilon})} \big]
&- \E \big[ M_t^\epsilon(F_x^0 ) e^{-\lambda M_t^\epsilon(e^{\sqrt{\beta} \varphi^\epsilon})} \big] \big\|_{L^\infty(\Omega_\epsilon)}
\nnb
&\leq
\sup_{x\in \Omega_\epsilon} \E \big[ M_t^\epsilon \big(| F_x^\epsilon - F_x^0|\big) e^{-\lambda M_t^\epsilon(e^{\sqrt{\beta} \varphi^\epsilon})}  \big]
\nnb
&\leq
\sup_{x\in \Omega_\epsilon} \sup_{y\in \Omega_\epsilon} |\dot c_t^\epsilon(x, y) - \dot c_t^0 (x,y)| \E \big[ M_t^\epsilon( e^{\sqrt{\beta}\varphi^\epsilon} ) \big] ,
\end{align}
which converges to $0$ as $\epsilon \to 0$ by the uniform convergence of $\dot c_t^\epsilon$ to $\dot c_t^0$ in Lemma \ref{lem:pttorus}.

We now turn to the second term on the right hand side of \eqref{eq:dctnablavt-convergence-with-gmcs-triangle}.
To see that it converges to $0$, we denote
\begin{equation}
g^\epsilon \colon \Omega\to \R, \, g^\epsilon(x)
= \E\big[ M_t^\epsilon (F_x^0) e^{-\lambda M_t^\epsilon(e^{\sqrt{\beta}\varphi^\epsilon})}
\big].
\end{equation}
Note that $g^\epsilon$ is continuous and the family is $\{g^\epsilon \colon \epsilon>0 \}$ is uniformly bounded by \eqref{eq:mexp-m-upper-bound} and equicontinuous. 
Indeed, for $x,y \in \Omega$, we have
\begin{align}
| g^\epsilon(x) - g^\epsilon (y) |
&= \big| \E\big[ M_t^\epsilon (F_x^0 - F_y^0) e^{-\lambda M_t^\epsilon(e^{\sqrt{\beta}\varphi^\epsilon})}
\big] \big|
\leq
\E \big[ M_t^\epsilon(|F_x^0 - F_y^0|) \big]
\nnb
&\leq
e^{-\sqrt{\beta} \| \varphi\|_{L^\infty(\Omega)}} \int_{\Omega} |\dot c_t^0(x,z) - \dot c_t^0(y,z) | dz ,
\end{align}
and thus, the equicontinuity of $\{g^\epsilon\colon \epsilon>0  \}$ follows by Lemma \ref{lem:dc-difference}.
By Lemma \ref{lem:small-scale-gmc-convergence} we have for any $x\in \Omega$
\begin{equation}
g^\epsilon(x) \to \E\big[ M_t (F_x^0) e^{-\lambda M_t(e^{\sqrt{\beta}\varphi})}
\big] \equiv g^0,
\end{equation}
and thus, we have by the Arz\`ela-Ascoli theorem that $g_0$ is continuous and
\begin{equation}
\label{eq:uniform-convergence-by-arz-asc}
\sup_{x\in \Omega} | g^{\epsilon}(x) - g(x)  | \to 0
\end{equation}
This shows that the second term on the right hand side of \eqref{eq:dctnablavt-convergence-with-gmcs-triangle} converges to $0$ as $\epsilon \to 0$, and thus, completes the proof of \eqref{eq:lv-dcnablav-limit}.

It remains to prove the additional estimates on $\dot c_t^\epsilon \nabla v_t^\epsilon$ and its derivatives for $\epsilon\geq 0$.
To prove \eqref{eq:lv-dcnablav-eps-upper-bound},
we let $x\in \Omega_\epsilon$ and $\varphi \in X_\epsilon$ and observe that, for $t\geq \epsilon^2$, the estimate \eqref{eq:lv-nablav-pointwise-upper} together with $\dot c_t^\epsilon(x,y) \geq 0$ implies that
\begin{align}
\label{eq:dcnablav-eps-upper-lower}
0 \leq \big( \dot c_t^\epsilon \nabla v_t^\epsilon(\varphi) \big)_x
&= \int_{\Omega_\epsilon} \dot c_t^\epsilon (x,y) \partial_{\phi_y} v_t^\epsilon(\varphi) dy
\leq \lambda \sqrt{\beta} \int_{\Omega_\epsilon} \dot c_t^\epsilon(x,y) \wick{e^{\sqrt{\beta} \varphi(y)}}_{L_t}  dy
\end{align}

To prove \eqref{eq:lv-dcnablav-0-upper-bound} we argue as follows.
We first assume that $x\in \Omega$ is such that there is $\epsilon_0>0$ such that $x\in \Omega_\epsilon$ for all $\epsilon<\epsilon_0$.
In other words, we first consider $x\in \Omega$ that are eventually contained in the dyadic discretisation of the torus.
%
Let $\varphi \in C(\Omega)$ and let $\varphi_\epsilon$ be its restriction to $\Omega_\epsilon$.
Then, for $\epsilon<\epsilon_0$, we have by \eqref{eq:dcnablav-eps-upper-lower}
\begin{align}
\label{eq:dcnablav-eps-upper-lower-dyadic-x}
0 \leq \big( \dot c_t^\epsilon \nabla v_t^\epsilon(\varphi_\epsilon) \big)_x
\leq \lambda \sqrt{\beta} \int_{\Omega} E_\epsilon \dot c_t^\epsilon(x,y) e^{\sqrt{\beta} E_\epsilon\varphi_\epsilon(y)} dy,
\end{align}
where $E_\epsilon$ denotes the extension from $\Omega_\epsilon$ to $\Omega$ by piecewise constant interpolation.
Since $\varphi\in C(\Omega)$, we have 
\begin{equation}
\label{eq:phi-phi-eps-convergence}
\sup_{y\in \Omega}| \varphi^\epsilon(y^\epsilon) - \varphi^\epsilon(y) | =   \sup_{y\in \Omega}| \varphi(y)-E_\epsilon \varphi_\epsilon(y) | \to 0.
\end{equation}
Moreover, we have by Lemma \ref{lem:pttorus}, specifically \eqref{e:pttoruslimit}, for $t>0$ and as $\epsilon \to 0$,
\begin{equation}
\sup_{y\in \Omega}| E_\epsilon \dot c_t^\epsilon(x,y) - \dot c_t^0(x,y) | 
=
\sup_{y\in \Omega} |\dot c_t^\epsilon(x,y^\epsilon) - \dot c_t^0(x,y) | \to 0.
\end{equation}
Thus, using the dominated convergence theorem on the integral in right hand side of \eqref{eq:dcnablav-eps-upper-lower-dyadic-x} and the convergence \eqref{eq:lv-dcnablav-limit} on the left hand side of \eqref{eq:dcnablav-eps-upper-lower}, the estimate \eqref{eq:lv-dcnablav-0-upper-bound} follows for $x$ dyadic.
For general $x\in \Omega$, we approximate $x$ by a dyadic sequence and use the previous result together with the continuity of $\dot c_t^0 \nabla v_t^0(\varphi)$.

Next, we show the bound \eqref{eq:lv-partial-dcnablav-eps-bd}.
To this end, we note that for $\varphi \in X_\epsilon$ and $k\geq 1$
\begin{align}
\partial_\epsilon^k \dot c_t^\epsilon \nabla v_t^\epsilon (\varphi,x)
=
\int_{\Omega_\epsilon} \partial_\epsilon^k \dot c_t(x, y) \partial_{\phi_y} v_t^\epsilon(\varphi) dy,
\end{align}
so that, using Lemma \ref{lem:lv-fkg}, 
\begin{align}
\big| \partial_\epsilon^k \dot c_t^\epsilon \nabla v_t^\epsilon (\varphi, x) \big| 
\leq 
\lambda \sqrt{\beta} \int_{\Omega_\epsilon} \big| \partial_\epsilon^k &\dot c_t^\epsilon(x,y) \big| \wick{e^{\sqrt{\beta} \varphi_y}}_{L_t}  dy 
\nnb
&\lesssim
 \lambda \sqrt{\beta} \| \partial_\epsilon^k \dot c_t^\epsilon(x,\cdot) \|_{L^1(\Omega_\epsilon)} 
 L_t^{\beta/4\pi}e^{\sqrt{\beta} \max_{\Omega_\epsilon}\varphi},
\end{align}
which, using Lemma \ref{lem:C-limit}, specifically \eqref{e:c-limit}, then gives
\begin{align}
L_t^{k}\| \partial_\epsilon^k \dot c_t^\epsilon \nabla v_t^\epsilon (\varphi) \|_{L^\infty(\Omega_\epsilon)}
\lesssim
L_t^{k} \sup_{x\in \Omega_\epsilon}
\|\dot c_t^\epsilon (x,\cdot )\|_{L^1(\Omega_\epsilon)} 
& \lambda \sqrt{\beta} L_t^{\beta/4\pi} e^{\sqrt{\beta} \max_{\Omega_\epsilon} \varphi}
\nnb
&\leq\lambda \sqrt{\beta} O_{k,\beta} (\theta_t) L_t^{\beta/4\pi} e^{\sqrt{\beta} \max_{\Omega_\epsilon}\varphi}.
\end{align}

To prove the estimate \eqref{eq:dcnablav-0-bd},
we take $\varphi \in C(\Omega)$ we let $\varphi_\epsilon$ be the restriction of $\varphi$ to $\Omega_\epsilon$.
Then, \eqref{eq:phi-phi-eps-convergence} holds,
and by \eqref{eq:lv-partial-dcnablav-eps-bd} we have
\begin{equation}
L_t^k\| \partial_\epsilon^k \dot c_t^\epsilon \nabla v_t^\epsilon (\varphi_\epsilon) \|_{L^\infty(\Omega_\epsilon)}
\leq  \lambda  O_{k, \beta} (\theta_t) L_t^{\beta/4\pi} e^{\sqrt{\beta} \max_{\Omega_\epsilon} \varphi_\epsilon} ,
\end{equation}
so that \eqref{eq:dcnablav-0-bd} follows from the convergence \eqref{eq:lv-dcnablav-limit} (which can be obtained for $\partial^k \dot c_t \nabla v_t$ with similar arguments only replacing $\dot c_t$ by $\partial^k \dot c_t$) and \eqref{eq:phi-phi-eps-convergence} toghether with the fact that $\partial^k \dot c_t^0 \nabla v_t^0 (\varphi)$ is continuous.
\end{proof}

\begin{remark}
We do not define $\nabla v_t^0$ pointwise,
but only in the composition with the heat kernel $\dot c_t^0$.
In fact, our proof does not allow us to define a function $\nabla v_t^0 \colon C(\Omega) \to C(\Omega)$.
\end{remark}

\subsection{Continuity of the gradient: Proof of Theorem \ref{thm:dcnablav-bd-cont}}
\label{ssec:proof-of-gradient-bounds}

We now turn to the proof of the continuity estimates of the gradient of $v_t^\epsilon$ in Theorem \ref{thm:dcnablav-bd-cont}.
To prove these results, we use the following estimates on the Hessian of the renormalised potential, defined by
\begin{equation}
\label{eq:hess-v-t-definition}
\big( \He v_t^\epsilon (\varphi) \big)_{x,y} = \partial_{\phi_x} \partial_{\phi_y} v_t^\epsilon (\varphi),
\end{equation}
which, for fixed $\epsilon>0$ and $\varphi \in X_\epsilon$, we view as a linear operator $X_\epsilon \to X_\epsilon$, 
\begin{equation}
(\He v_t^\epsilon (\varphi) f )_x \mapsto \epsilon^2 \sum_{y\in \Omega_\epsilon} \big(\He v_t^\epsilon (\varphi)\big)_{x,y} f_y.
\end{equation}
As for the gradient, the derivatives in \eqref{eq:hess-v-t-definition} are understood as Frech\'et derivative,
i.e., with respect to the inner product on $X_\epsilon$ defined in \eqref{eq:inner-product-x-eps}.
In particular, we have for instance 
\begin{equation}
\big( \He v_0^\epsilon (\varphi) \big)_{x,y} = \lambda \beta \epsilon^{\beta/4\pi} e^{\sqrt{\beta} \varphi_x} \mathbf{1}_{x=y}.
\end{equation}

Here and below, we denote for a linear operator $A^\epsilon \colon X_\epsilon \to X_\epsilon$ its operator norm when seen as an operator $L^\infty(\Omega_\epsilon)\to L^\infty(\Omega_\epsilon)$ by
\begin{equation}
\label{eq:infty-operator-norm}
\opnorm{ A^\epsilon }{\epsilon} = \sup_{x\in \Omega_\epsilon}\epsilon^2 \sum_{y\in \Omega_\epsilon} |A^\epsilon(x,y)| .
\end{equation}
Similarly, we write for a linear operator of the form $A\colon L^\infty(\Omega) \to L^\infty(\Omega),\, Af(x) = \int_{\Omega} A(x,y) f(y) dy$
\begin{equation}
\opnorm{A}{} = \sup_{x\in \Omega} \int_{\Omega} |A(x,y)| dy .
\end{equation}

\begin{lemma}
\label{lem:he-pos-definite}
Let $\epsilon>0$. Then we have for any $t\geq 0$ and any $\varphi \in \Omega_\epsilon$
\begin{equation}
\label{eq:he-pos-definite}
\He v_t^\epsilon(\varphi) \geq 0
\end{equation}
as quadratic forms on $X_\epsilon$.
In particular, $v_t^\epsilon \colon X_\epsilon \to \R$ is convex.
\end{lemma}

\begin{proof}
We first recall that
\begin{equation}
\label{eq:recall-vt-proof-he-pos-def}
v_t^\epsilon (\varphi) = - \log \EE_{c_t^\epsilon}[e^{-v_0^\epsilon(\varphi+\zeta)}].
\end{equation}
To ease the notation, we drop $\epsilon>0$ from $c_t^\epsilon$ throughout the rest of the proof.
A change of variable $\zeta' = \zeta + \varphi$ in the Gaussian measure permits us to write \eqref{eq:recall-vt-proof-he-pos-def} as
\begin{align}
\label{eq:Lv-renormalised-potential-log}
v_t^\epsilon (\varphi) 
&= 
- \log \frac{1}{Z_t^\epsilon} \int_{\R^{\Omega_\epsilon}} e^{-v_0^\epsilon(\zeta + \varphi)} e^{-\frac{1}{2}\zeta c_t^{-1} \zeta } d \zeta \nnb
& = \frac{1}{2}\varphi c_t^{-1} \varphi   - \log \frac{1}{Z_t^\epsilon} \int_{\Omega_\epsilon} e^{-v_0^\epsilon(\zeta')  + \zeta' c_t^{-1} \varphi  - \frac{1}{2}\zeta' c_t^{-1} \zeta' } d\zeta',
\end{align}
where $Z_t^\epsilon$ denotes the normalisation constant of the centred Gaussian measure on $X_\epsilon$ with covariance $c_t$.
Let now $\mu_{t,\varphi}^\epsilon$ be the probability measure on $X_\epsilon$ defined through
\begin{equation}
d \mu_{t,\varphi}^\epsilon (\zeta) \propto e^{-v_0^\epsilon(\zeta) + \zeta c_t^{-1} \varphi - \frac{1}{2} \zeta c_t^{-1} \zeta } d\zeta.
\end{equation}
Then differentiating \eqref{eq:Lv-renormalised-potential-log} with respect to $\varphi$ gives
\begin{equation}
\nabla v_t^\epsilon(\varphi)
= c_t^{-1} \varphi
- \frac{\int_{\R^{\Omega_\epsilon}} c_t^{-1}\zeta e^{-v_0^\epsilon(\zeta) + \zeta c_t^{-1} \varphi} d\zeta  }{ \int_{\R^{\Omega_\epsilon}} e^{-v_0^\epsilon(\zeta) + \zeta c_t^{-1} \varphi} d\zeta } 
= c_t^{-1} \varphi - \E_{\mu_{t,\varphi}^\epsilon} [c_t^{-1} \zeta].
\end{equation}
Differentiating once again, it can easily be verified that 
\begin{align}
\label{eq:he-after-shift}
\He v_t^\epsilon(\varphi)
&= c_t^{-1} - \cov_{\mu_{t,\varphi}^\epsilon} (c_t^{-1} \zeta)
= c_t^{-1}  - c_t^{-1} \cov_{\mu_{t,\varphi}^\epsilon} (\zeta) c_t^{-1}. 
\end{align}
Now, let $U_{t, \varphi}^\epsilon = v_0^\epsilon(\zeta) - \zeta c_t^{-1} \varphi + \frac{1}{2}\zeta c_t^{-1} \zeta$, so that
\begin{equation}
d \mu_{t,\varphi}^\epsilon (\zeta) \propto e^{-U_{t, \varphi}^\epsilon(\zeta)} d\zeta.
\end{equation}
Then we have
\begin{equation}
\label{eq:lower-bound-u}
\He U_{t,\varphi}^\epsilon(\zeta) =  \beta\lambda^2 \diag \big( \wick{e^{\sqrt{\beta} \zeta_x} }_\epsilon \big) + c_t^{-1} \geq c_t^{-1} ,
\end{equation}
and thus, by the Brascamp-Lieb inequality \cite[Theorem 4.1]{MR0450480}, we have
\begin{equation}
\cov_{\mu_{t,\phi}^\epsilon} (\zeta) \leq c_t,
\end{equation}
which, together with \eqref{eq:he-after-shift}, implies \eqref{eq:he-pos-definite}.
\end{proof}

\begin{lemma}
\label{lem:he-bound-entries}
Let $\epsilon >0$ and let $x,y\in \Omega_\epsilon$.
Then, for $\varphi \in X_\epsilon$,
\begin{equation}
\label{eq:he-bound-entries}
\big| \big(\He v_t^\epsilon(\varphi)\big)_{x,y} \big|
\leq \frac{\beta \lambda}{2} \big( \wick{e^{\sqrt{\beta} \varphi_x} }_{L_t} + \wick{e^{\sqrt{\beta} \varphi_y} }_{L_t} \big).
\end{equation}
\end{lemma}

\begin{proof}
We first note that for $x\neq y$ and $\varphi \in \Omega_\epsilon$ we have
\begin{align}
\big( \He v_t^\epsilon (\varphi) \big)_{x,y}
&= 
- \beta \lambda^2
\Big( 
\frac{\EE_{c_t^\epsilon}[\wick{e^{\sqrt{\beta}(\varphi_x+\zeta_x)}}_\epsilon \wick{e^{\sqrt{\beta} (\varphi_y + \zeta_y)}}_\epsilon e^{-v_0^\epsilon(\varphi + \zeta)}]}{\EE_{c_t^\epsilon}[e^{-v_0^\epsilon(\varphi+ \zeta)}]} 
\nnb
&\qquad\qquad\qquad- 
\frac{\EE_{c_t^\epsilon}[\wick{e^{\sqrt{\beta}(\varphi_x+\zeta_x)}}_\epsilon e^{-v_0^\epsilon(\varphi + \zeta)}]}{\EE_{c_t^\epsilon}[e^{-v_0^\epsilon(\varphi+ \zeta)}]} 
\frac{\EE_{c_t^\epsilon}[ \wick{e^{\sqrt{\beta} (\varphi_y + \zeta_y)}}_\epsilon e^{-v_0^\epsilon(\varphi + \zeta)}]}{\EE_{c_t^\epsilon}[e^{-v_0^\epsilon(\varphi+ \zeta)}]}
\Big)
,
\end{align}
and similarly, for $x\in \Omega_\epsilon$,
\begin{align}
\big( \He v_t^\epsilon (\varphi) \big)_{x,x}
&= 
\beta \lambda 
\frac{\EE_{c_t^\epsilon}[\wick{e^{\sqrt{\beta}(\varphi_x+\zeta_x)}}_\epsilon  e^{-v_0^\epsilon(\varphi + \zeta)}]}{\EE_{c_t^\epsilon}[e^{-v_0^\epsilon(\varphi+ \zeta)}]}
\nnb
&- \beta \lambda^2
\Big( 
\frac{\EE_{c_t^\epsilon}\big[\big(\wick{e^{\sqrt{\beta}(\varphi_x+\zeta_x)}}_\epsilon\big)^2 e^{-v_0^\epsilon(\varphi + \zeta)}\big]}{\EE_{c_t^\epsilon}[e^{-v_0^\epsilon(\varphi+ \zeta)}]} 
\nnb
&\qquad\qquad\qquad- 
\big( \frac{\EE_{c_t^\epsilon}[\wick{e^{\sqrt{\beta}(\varphi_x+\zeta_x)}}_\epsilon e^{-v_0^\epsilon(\varphi + \zeta)}]}{\EE_{c_t^\epsilon}[e^{-v_0^\epsilon(\varphi+ \zeta)}]}
\big)^2 
\Big).
\end{align}
Thus, we have that
\begin{equation}
\He v_t^\epsilon (\varphi) = \beta \lambda D^\epsilon(\varphi) - \beta \lambda^2 C^\epsilon(\varphi) ,
\end{equation}
where $C^\epsilon(\phi)$ is the covariance matrix of $\wick{e^{\sqrt{\beta}(\varphi+ \zeta)}}_\epsilon$ under the measure $\tilde \mu_{t,\varphi}$ defined through
\begin{equation}
d \tilde \mu_{t,\varphi} \propto e^{-v_0^\epsilon(\varphi+\zeta) + \frac{1}{2}\zeta c_t^{-1} \zeta} d\zeta
\end{equation}
and $D^\epsilon(\phi)$ is the diagonal matrix
\begin{equation}
D^\epsilon(\phi) = \diag\Big( 
\frac{\EE_{c_t^\epsilon}[\wick{e^{\sqrt{\beta}(\phi_x+\zeta_x)}}_\epsilon  e^{-v_0^\epsilon(\phi + \zeta)}]}{\EE_{c_t^\epsilon}[e^{-v_0^\epsilon(\phi+ \zeta)}]}
 \Big).
\end{equation}
Since $C^\epsilon(\varphi)$ is positive definite, we have, for any $\varphi \in X_\epsilon$,
\begin{equation}
\label{eq:he-upper-lower}
0 \leq \He v_t^\epsilon (\varphi) \leq \beta \lambda D^\epsilon(\varphi)
\end{equation}
as quadratic forms on $X_\epsilon$,
where the lower bound follows from Lemma \ref{lem:he-pos-definite}.
Now, using Lemma \ref{lem:lv-fkg}, we have for $x\in \Omega_\epsilon$
\begin{equation}
0 \leq \big(\He v_t^\epsilon (\varphi) \big)_{x,x} \leq \beta \lambda D^\epsilon(\varphi)_{x,x} \leq \beta \lambda \wick{e^{\sqrt{\beta}\varphi_x} }_{L_t},
\end{equation}
which is \eqref{eq:he-bound-entries} for $x=y$.
Moreover, evaluating the quadratic form inequality \eqref{eq:he-upper-lower} on $\mathbf{1}_{x} \pm \mathbf{1}_y$ gives
\begin{equation}
\big| 2 \big(\He v_t^\epsilon(\varphi) \big)_{x,y} \big| \leq \beta \lambda \Big( D^\epsilon(\varphi) \big)_{x,x} + \big( D^\epsilon(\varphi) \big)_{y,y} \Big)
\leq 
\beta \lambda \big( \wick{e^{\sqrt{\beta}\varphi_x} }_{L_t} + \wick{e^{\sqrt{\beta}\varphi_y} }_{L_t} \big),
\end{equation}
where we used Lemma \ref{lem:lv-fkg} to obtain the second inequality.
This yields \eqref{eq:he-bound-entries} for $x\neq y$.
\end{proof}

\begin{proof}[Proof of Theorem \ref{thm:dcnablav-bd-cont}]

We first give the proof of \eqref{eq:dcnablav-eps-cont} for $k=0$.
The general case $k\geq 1$ is identical when $\dot c_t^\epsilon$ is replaced by $\partial_\epsilon^k \dot c_t^\epsilon$.
For $\epsilon>0$, $t\geq \epsilon^2$ and $\varphi, \varphi' \in X_\epsilon$ we define the function $F\colon [0,1] \to X_\epsilon$,
\begin{equation}
F(s) = \dot c_t^\epsilon \nabla v_t^\epsilon \big(\varphi + s(\varphi'-\varphi)\big).
\end{equation}
Then, we have by the chain rule
\begin{equation}
\label{eq:cnablav-ds}
\frac{d}{ds} F (s) = \dot c_t^\epsilon \He v_t^\epsilon\big( \varphi_s\big) \cdot (\varphi- \varphi') ,
\end{equation}
where we set $\varphi_s = \varphi + s(\varphi- \varphi')$ for $s\in [0,1]$.
With the notation introduced in \eqref{eq:infty-operator-norm}, we have from \eqref{eq:cnablav-ds}
\begin{equation}
\label{eq:dfds-upper-bound}
\big\|\frac{d}{ds} F(s) \big\|_{L^\infty(\Omega_\epsilon)}
\leq \opnorm{ \big( \dot c_t^\epsilon \He v_t^\epsilon( \varphi_s)\big) }{\epsilon} \|\varphi- \varphi'\|_{L^\infty(\Omega_\epsilon)}. 
\end{equation}
By Lemma \ref{lem:he-bound-entries} we have
\begin{align}
\opnorm{\big( \dot c_t^\epsilon \He v_t^\epsilon (\varphi_s)\big) }{\epsilon}
&= 
\sup_{x\in \Omega_\epsilon}
\epsilon^2 \sum_{y\in \Omega_\epsilon} \big|\dot c_t^\epsilon(x,y)\big( \He v_t^\epsilon(\varphi_s) \big)_{x,y} \big| \nnb
&\leq
\frac{\lambda \beta}{2}
\sup_{x\in \Omega_\epsilon}
\epsilon^2 \sum_{y\in \Omega_\epsilon} \dot c_t^\epsilon(x,y) \big( \wick{e^{\sqrt{\beta} \varphi_s(x)}}_{L_t} + \wick{e^{\sqrt{\beta} \varphi_s(y)}}_{L_t} \big) 
\nnb
&\leq 
\lambda \beta \sup_{z\in \Omega_\epsilon}
\wick{e^{\sqrt{\beta} \varphi_s(z)}}_{L_t} 
\sup_{x\in \Omega_\epsilon} \|\dot c_t^\epsilon(x,\cdot) \|_{L^1(\Omega_\epsilon)}
\nnb
&\lesssim_{\beta} 
\lambda \beta L_t^{\beta/4\pi}
\big( s  e^{\sqrt{\beta} \max \varphi}  + (1-s) e^{\sqrt{\beta} \max \varphi'}  \big)
\sup_{x\in \Omega_\epsilon} \|\dot c_t^\epsilon(x,\cdot) \|_{L^1(\Omega_\epsilon)} .
\end{align}
Hence, the claimed inequality \eqref{eq:dcnablav-eps-cont} follows from integrating the last inequality on $[0,1]$ together with \eqref{eq:dfds-upper-bound} and Lemma \ref{lem:C-limit}, specifically \eqref{e:c-limit}.

To prove the estimate \eqref{eq:dcnablav-0-cont},
we let $\varphi_\epsilon$ and $\varphi_\epsilon'$ be the restrictions of $\varphi$ and $\varphi'$ to $\Omega_\epsilon$.
Then, since $\varphi \in C(\Omega)$ we have
\begin{equation}
\label{eq:varphi-pvarphi-convergence-2}
\sup_{x\in \Omega} |\varphi_\epsilon(x^\epsilon) - \varphi(x) | \to 0
\end{equation}
as $\epsilon \to 0$ and similarly for $\varphi_\epsilon'$ and $\varphi'$.
By \eqref{eq:dcnablav-eps-cont} we have
\begin{equation}
\| \dot c_t^\epsilon \nabla v_t^\epsilon (\varphi_\epsilon)
-
\dot c_t^\epsilon \nabla v_t^\epsilon (\varphi_\epsilon')
\|_{L^\infty(\Omega_\epsilon)}
\leq\lambda
O_\beta(\theta_t)
L_t^{\beta/4\pi} \big( e^{\sqrt{\beta} \max \varphi_\epsilon}  + e^{\sqrt{\beta} \max \varphi_\epsilon'} \big)
\|\varphi_\epsilon - \varphi_\epsilon' \|_{L^\infty(\Omega_\epsilon)}
,
\end{equation}
and thus, \eqref{eq:dcnablav-0-cont} follows from the convergence \eqref{eq:lv-dcnablav-limit} and \eqref{eq:varphi-pvarphi-convergence-2} together with the fact that $\dot c_t^0 \nabla v_t^0 (\varphi)$ is continuous.
\end{proof}

\section{Coupling to the Gaussian free field}
\label{sec:coupling-lattice}

In this section we state the main results which enter in the proof of Theorem \ref{thm:coupling-intro}.
As we explained in Section \ref{ssec:vt-eps},
we obtain a Liouville measure and the coupling to the Gaussian free field in Theorem \ref{thm:coupling-intro} from studying its regularisations \eqref{eq:lv-density},
which are constructed using the SDE \eqref{eq:lv-polchinski-sde}.
As we show with Theorem \ref{thm:lv-coupling} and Theorem \ref{thm:law-phi-0} below,
a solution of \eqref{eq:lv-polchinski-sde} exists and is unique,
and moreover, for $\epsilon>0$, its solution at $t=0$ is distributed as \eqref{eq:lv-density}.

The bounds on $\Phi^{\Delta}$ in Theorem \ref{thm:coupling-intro} then follow from the analogous bounds on the regularised fields $\Phi^{\Delta_\epsilon}$, which,
as we show in Theorem \ref{thm:lv-coupling},
are uniform in $\epsilon>0$.
The convergence of the lattice field as $\epsilon \to 0$ is then a consequence of the bounds on $\Phi^{\Delta_\epsilon}$ together with the convergence of $\Phi_t^{\Delta_\epsilon}$ to a continuum field $\Phi_t^{\Delta_0}$ stated in Theorem \ref{thm:lv-phi-delta-t-limit}. 

\subsection{Coupling of Gaussian lattice fields}

To prove these statements, it is convienient to work on one common probability space, where all solutions $\Phi^{\Lv_\epsilon}$ to \eqref{eq:lv-polchinski-sde} are realised simultaneously for all $\epsilon\geq 0$.
The purpose of this construction is twofold.
First, it allows to phrase the continuity estimates on the difference field $\Phi^{\Delta_\epsilon}$ in Theorem \ref{thm:coupling-intro} in a clean way for both $\epsilon>0$ and $\epsilon=0$
with all random bounds on the H\"older constants in \eqref{eq:intro-phi-delta-bounds} being independent of $\epsilon\geq 0$.
Second, it simplifies the proof of the convergence $(\Phi_0^{\Lv_0})_{\epsilon>0}$ to a limit $\Phi_0^{\Lv_0}$ in $H^{-\kappa}(\Omega)$ for any $\kappa>0$ as $\epsilon \to 0$. 

The coupling of the processes $\Phi^{\Lv_\epsilon}$ comes from the coupling of Gaussian processes $(\Phi^{\GFF_\epsilon})_{\epsilon>0}$ and $\Phi^{\GFF_0}$,
which drive the SDEs \eqref{eq:lv-polchinski-sde}.
This is similar to \cite[Section 3.1]{MR4399156} and,
for consistency, we follow in large parts the notation in this reference.

As in Section \ref{ssec:proof-of-gradient-convergence}, we write $\Omega^*= 2\pi \Z^2$ for the Fourier dual of $\Omega$ and denote by $\hat q_t^0(k)$, $k\in \Omega^*$ the Fourier multiplier of $q_t^0\colon \Omega \to \R$ with $\big[q_t^0 * q_t^0 \big](x-y) =  \dot c_t^0(x,y)$. 
Similarly, we write $\Omega_\epsilon^*= \{k\in 2\pi \Z^2 \colon -\pi/\epsilon <k_i\leq \pi/\epsilon \}$ for the Fourier dual of $\Omega_\epsilon^*$ and $\hat q_t^\epsilon(k)$ for the Fourier multipliers of the function $q_t^\epsilon \colon \Omega_\epsilon \to \R$ with $\big[q_t^\epsilon * q_t^\epsilon\big](x-y) = \dot c_t^\epsilon (x,y)$,
where here, $*$ denotes the discrete convolution on $\Omega_\epsilon$.
Let further $W$ be a cylindrical Brownian motion in $L^2(\Omega)$ defined on some probability space, which admits a Karhunen–Lo\`eve representation
\begin{equation}
\label{eq:cylindrical-bm-kl-representation}
W_t = \sum_{k \in \Omega^*} e^{ik\cdot (\cdot)} \hat W_t(k) ,
\end{equation}
where $\big( \hat W(k) \big)_{k\in \Omega^*}$ are independent complex standard Brownian motions subject to $\hat W(k)=\overline{\hat W(-k)}$
for $k \neq 0$ and $\hat W(0)$ is a real standard Brownian motion.
Here, the convergence of the sum over $k \in \Omega^*$ is in $C([0,\infty),H^{-1-\kappa}(\Omega))$ for any $\kappa>0$.

As shown in \cite[(3.15)]{MR4399156},
we obtain from $W$ for each $\epsilon>0$ a family of independent real Brownian motions $\big(W^\epsilon(x) \big)_{x\in \Omega_\epsilon}$ with variance $t/\epsilon^2$ through
\begin{equation}
W^\epsilon_t(x) = \Pi_\epsilon W_t(x) = \sum_{k \in \Omega_\epsilon^*} e^{ik\cdot x} \hat W_t(k),
\qquad x\in \Omega_\epsilon,
\end{equation}
where $\Pi_\epsilon$ denotes the restriction of the Fourier series \eqref{eq:cylindrical-bm-kl-representation} to coefficients in $\Omega_\epsilon^*$.

We now define for $\epsilon>0$ and $t\geq 0$
\begin{equation}
\label{eq:gff-eps-def}
\Phi_t^{\GFF_\epsilon}(x) = \int_t^\infty q_s^\epsilon \, dW_s^\epsilon
= \sum_{k \in \Omega_\epsilon^*} e^{ik\cdot x} \int_t^\infty \hat q_s^\epsilon(k) d\hat W_s(k), \qquad x\in \Omega_\epsilon.
\end{equation}
It can easily be verified that $\Phi_t^{\GFF_\epsilon}$ is a Gaussian field on $\Omega_\epsilon$ with covariance $c_\infty^\epsilon-c_t^\epsilon$.
Similarly, we define for $t\geq 0$
\begin{equation}
\label{eq:gff-0-def}
\Phi_t^{\GFF_0} = \int_t^\infty q_s^0 dW_s \equiv
\sum_{k\in \Omega^*}  e^{ik\cdot(\cdot)} \int_t^\infty \hat q_s^0(k) d\hat W_s(k).
\end{equation}
Here, for $t>0$, the convergence of the sum in \eqref{eq:gff-0-def} is in $H^{\kappa}(\Omega)$ for any $\kappa>0$,
while for $t=0$, the convergence is in $H^{-\kappa}(\Omega)$ for any $\kappa>0$,
which can be seen from the decay of the Fourier coefficents $\hat q_t^0(k)$.
Thus, for $t>0$, $\Phi_t^{\GFF_0}$ is well defined smooth centred Gaussian field with covariance
\begin{equation}
\E\big[  \avg{\Phi_t^{\GFF_0}, f} \avg{\Phi_t^{\GFF_0},g}   \big] 
= \int_{\Omega} \int_{\Omega} \big(c_\infty^0-c_t^0\big)(x-y) f(x) g(y) dx dy, \qquad f,g\in L^2(\Omega),
\end{equation}
and the last identity also holds true for $t=0$.
Note that all processes $(\Phi^{\GFF_\epsilon})_{\epsilon>0}$ and $\Phi^{\GFF_0}$ are realised on same probability space,
where the cylindrical Brownian motion $W$ is defined.
This coupling is underlying all statements and proofs throughout the rest of this work.

Finally, we denote the forward and backward filtrations associated with $W$ by $(\cF_t)$ and $(\cF^t)$, which are obtained by completing the $\sigma$-algebras generated by $\{W_s-W_0\colon s \leq t\}$ and $\{W_s-W_t\colon s\geq t\}$ respectively.
Note that the processes $\Phi^{\GFF_\epsilon} = (\Phi^{\GFF_\epsilon}_t)_{t\geq 0}$
for both $\epsilon >0$ and $\epsilon=0$ are adapted to the backward filtration $(\cF^t)_{t\geq 0}$.

\subsection{Coupling of Liouville lattice fields}

The first result in this section is the well-posedness of the (backward) SDE \eqref{eq:lv-polchinski-sde} for $\epsilon >0$ and $\epsilon = 0$.
Here and afterwards, $C_0$ denotes the space of continuous functions vanishing at infinity.

\begin{theorem}
\label{thm:lv-coupling}
For $\epsilon>0$, there is a unique $\cF^t$-adapted process $\Phi^{\Lv_\epsilon} \in C_0([0,\infty), X_\epsilon)$ such that
$\Phi_t^{\Lv_\epsilon} - \Phi_t^{\GFF_\epsilon} \leq 0$ for all $t\geq 0$
and
\begin{equation}
\label{eq:Phi-Lv-eps-coupling}
\Phi_t^{\Lv_\epsilon}
= - \int_t^\infty \dot c^\epsilon_s \nabla v_{s}^\epsilon(\Phi_s^{\Lv_\epsilon}) \, ds
      + \Phi_t^{\GFF_\epsilon} \qquad \text{a.s.}
\end{equation}

Analogously, there is a unique $\cF^t$-adapted process $\Phi^{\Lv_0}$
with $\Phi^{\Lv_0}-\Phi^{\GFF_0}\in C_0([0,\infty), C(\Omega))$ such that $\Phi_t^{\Lv_0} - \Phi_t^{\GFF_0} \leq 0$ for all $t\geq 0$
and
\begin{equation}
\label{eq:Phi-Lv-0-coupling}
\Phi_{t}^{\Lv_0} = - \int_t^\infty \dot c^0_s \nabla v_s^0(\Phi_s^{\Lv_0}) \, ds
   + \Phi_t^{\GFF_0} \qquad \text{a.s.}
\end{equation}
In particular, for any $t>0$, $\Phi^{\GFF}_0-\Phi_t^{\GFF}$ is independent of $\Phi_t^{\Lv}$ for both $\epsilon>0$ and $\epsilon=0$.
\end{theorem}

The next result shows that the solution of \eqref{eq:Phi-Lv-eps-coupling} evaluated at $t=0$ is indeed a realisation of the Liouville measure.

\begin{theorem}
\label{thm:law-phi-0}
For $\epsilon>0$ we have that $\Phi_0^{\Lv_\epsilon}$ is distributed as the Liouville measure defined in \eqref{eq:lv-density}.
\end{theorem}

As we explained in Section \ref{ssec:vt-eps}, the finite variation integral in the SDEs \eqref{eq:Phi-Lv-eps-coupling} and \eqref{eq:Phi-Lv-0-coupling} is closely related to the field $\Phi^\Delta$ in Theorem \eqref{thm:coupling-intro},
and hence, it plays an important role for the coupling between the Gaussian free field and the Liouville field.
In what follows we use the notation
\begin{equation}
\label{eq:coupling-at-scale-t}
\Phi_t^{\Lv_\epsilon} = \Phi_t^{\GFF_\epsilon} + \Phi_t^{\Delta_\epsilon}, \qquad \Phi_t^{\Delta_\epsilon} = -\int_t^\infty \dot c_s^\epsilon \nabla v_s^\epsilon (\Phi_s^{\Lv_\epsilon})ds,
\end{equation}
for both $\epsilon>0$ and $\epsilon=0$.
With this definition the field $\Phi^\Delta$ appearing in Theorem \ref{thm:coupling-intro} is equal to $\Phi_0^{\Delta_0}$.

To prove the coupling between the Gaussian free field and the Liouville field in Theorem \ref{thm:coupling-intro} we analyse $\Phi_t^{\Delta_\epsilon}$ and establish bounds that are uniform in $\epsilon>0$.
There are two main technical difficulties we overcome.
First, the gradient of the renormalised potential appearing in $\Phi_t^{\Delta_\epsilon}$ is a complicated object,
which is defined through derivatives of the logarithm of a Gaussian expectation value.
Second, the argument appearing of the gradient of the renormalised potential in \eqref{eq:coupling-at-scale-t} is a random field.
We overcome both diffuculties with the bounds on $\partial_{\phi_x} v_t^\epsilon$ obtained in Lemma \ref{lem:lv-fkg},
which in particular imply that the difference field $\Phi_t^{\Delta_\epsilon}$ is non-positive.
The estimates in the following result are valid for both $\epsilon>0$ and $\epsilon>0$,
for which we do not write this parameter explicitely.
The maximum is thus either over $\Omega_\epsilon$ for $\epsilon>0$ or over $\Omega$ for $\epsilon=0$.

\begin{theorem}
\label{thm:lv-coupling-bounds-lattice}
Let $\beta \in (0, 8\pi)$ and let $\lambda >0$. The following estimates hold uniformly for $\epsilon>0$ or $\epsilon=0$ with constants independent of $\epsilon$.
Let $\hdphs=\hdphs(\beta)$ be as in \eqref{eq:parameter-hoelder-phase}.
For any $\delta \in (0,2-\hdphs)$ the difference field $\Phi^\Delta$ satisfies a.s.\ the bound
\begin{equation}
\label{eq:lv-phi-delta-0-t-bound}
\max_x |\Phi^{\Delta}_0(x)-\Phi^{\Delta}_t(x)|
\leq
O_{\beta,\delta} (L_t^\delta)
\end{equation}
in the sense there exists a non-random constant $C_{\beta,\delta}>0$, such that a.s.\ there is $t_0>0$ with
\begin{equation}
\max_x |\Phi^{\Delta}_0(x)-\Phi^{\Delta}_t(x)| \leq C_{\beta,\delta} L_t^{\delta}, \qquad t\leq t_0.
\end{equation}
Moreover, $\Phi_t^{\Delta}$ satisfies the following H\"older continuity estimates:
for any $\delta >0$, there are positive and a.s.\ finite random variables $M_{\beta,\delta}$, which are uniform in $\epsilon \geq 0$ and $t\geq 0$,
such that
\begin{align}
\label{eq:lv-phi-delta-bd-max-continuity}
\begin{split} 
\max_{x} |\Phi_t^{\Delta}(x)|
+
\max_{x} |\partial \Phi_t^{\Delta} (x)| 
+
\max_{x,y} \frac{| \partial \Phi_t^\Delta(x)- \partial \Phi_t^\Delta(y)|}{|x-y|^{1-\hdphs-\delta}}
&\leq \lambda M_{\beta,\delta},
\qquad (0<\hdphs < 1),
%
\\
\max_{x} |\Phi_t^{\Delta}(x)|+ \max_{x,y}
\frac{|\Phi_t^\Delta(x)-\Phi_t^\Delta(y)|}{|x-y|^{2-\hdphs-\delta}}
&\leq \lambda M_{\beta, \delta},
\qquad (1 \leq \hdphs <2).
\end{split}
\end{align}

In addition, for any $t>0$ and $k \in \N$, there are positive and a.s.\ finite random variables $\tilde M_{\beta,k}$,
such that
\begin{equation}
\label{eq:lv-del-phi-delta-bd}
L_t^k\|\partial^k\Phi^\Delta_t\|_{L^\infty(\Omega)}
\leq \lambda \tilde M_{\beta,k}.
\end{equation}
In particular, $\Phi^\Delta_t \in C^\infty(\Omega)$ for any $t>0$.
\end{theorem}

\subsection{Convergence of lattice fields}

At this point, apart from its existence, we do not have any information on the solution $\Phi^{\Lv_0}$ of \eqref{eq:Phi-Lv-0-coupling},
nor on its relation to the solutions $\Phi^{\Lv_\epsilon}$ of \eqref{eq:Phi-Lv-eps-coupling}.
The next result shows the convergence as $\epsilon\to 0$,
thereby giving $\Phi_0^{\Lv_0}$ the interpretation of the continuum Liouville model \eqref{eq:lv-formal-density}.

In what follows we also view $\Phi_t^{\GFF_\epsilon}$ in \eqref{eq:gff-eps-def} as a field on $\Omega$ using the isometric embedding $I_\epsilon$.
More precisely, we let $I_\epsilon \Phi_t^{\GFF_\epsilon}$ be the field on $\Omega$ with Fourier coefficients $\hat q_t^\epsilon(k)$ for $k\in \Omega_\epsilon^*$ and vanishing Fourier coefficents for $k\in \Omega^* \setminus \Omega_\epsilon^*$. 
Using the orthogonality of the complex exponentials, it can easily be verified that $I_\epsilon \colon X_\epsilon \to L^2(\Omega)$ is an isometry in the sense that for any $\Phi \in X_\epsilon$,
\begin{equation}
\| \Phi \|_{L^2(\Omega_\epsilon)} = \| I_\epsilon\Phi \|_{L^2(\Omega)}.
\end{equation}

\begin{theorem}
\label{thm:lv-phi-delta-t-limit}
Under the coupling introduced at the beginning of Section \ref{sec:coupling-lattice}
we have for any $t_0>0$, 
\begin{equation}
\label{eq:lv-phi-delta-t-limit}
\sup_{t\geq t_0}\norm{\Phi_t^{\Delta_\epsilon}- \Phi^{\Delta_0}_t}_{L^\infty(\Omega_\epsilon)} \to 0
\quad \text{as $\epsilon \to 0$ in probability.}
\end{equation}
In particular, for any $t\geq 0$, the lattice field $\Phi^{\Lv_\epsilon}_t$ converges weakly to $\Phi^{\Lv_0}_t$ in $H^{-\kappa}(\Omega)$ as $\epsilon \to 0$ for any $\kappa>0$,
where we have identified $\Phi^{\Lv_\epsilon}_t$ with the element of $C^\infty(\Omega)$ with the same Fourier coefficients for $k\in \Omega_\epsilon^*$ and vanishing Fourier coefficients for $k\not\in\Omega_\epsilon^*$.
\end{theorem}

\begin{remark}
The convergence of $I_\epsilon\Phi_0^{\Lv_\epsilon}$ to $\Phi_0^{\Lv_0}$ is understood in the sense that for any $\kappa>0$, we have
\begin{equation}
\label{eq:phi-0-lv-eps-convergence}
\| I_\epsilon \Phi_0^{\Lv_\epsilon} - \Phi_0^{\Lv_0} \|_{H^{-\kappa}(\Omega)}
\to 0
\end{equation}
as $\epsilon \to 0$ in probability.
This implies in particular, that, for any $g\in H^\kappa(\Omega)$,
\begin{equation}
I_\epsilon \Phi_0^{\Lv_\epsilon} (g) \to  \Phi_0^{\Lv_0}(g)
\end{equation}
as $\epsilon \to 0$ in probability.
\end{remark}

\subsection{Control of the Gaussian maximum along all scales}
 
One of the key results that enter the proofs of Theorem \ref{thm:lv-coupling}, Theorem \ref{thm:lv-coupling-bounds-lattice} and Theorem \ref{thm:lv-phi-delta-t-limit} are the following estimates on the maximum of the processes $\Phi^{\GFF_\epsilon}$ and $\Phi^{\GFF_0}$ defined in \eqref{eq:gff-eps-def} and \eqref{eq:gff-0-def} along all scales $t>0$.

\begin{theorem}
\label{thm:lv-good-event}
For any $\rho\in (0,1)$ we have
\begin{equation}
\label{eq:max-phi-t-0-eventually-upper-bounded}
\P\Big(\exists t_0 >0 \colon \forall t\in (0,t_0) \colon \max_{\Omega} \Phi_t^{\GFF_0} \leq ( \fom + \rho) \log \frac{1}{L_t} \Big) = 1,
\end{equation}
and similarly
\begin{equation}
\label{eq:max-phi-t-eps-eventually-upper-bounded}
\P\Big(\exists t_0 >0 \colon \forall t\in (0, t_0) \colon \sup_{ \epsilon\leq L_t} \max_{\Omega_\epsilon} \Phi_t^{\GFF_\epsilon} \leq (\fom+\rho)\log \frac{1}{L_t}  \Big) = 1.
\end{equation}
\end{theorem}

\begin{remark}
\label{rem:optimal-estimate-on-max}
The term $\fom \log \frac{1}{L_t}$ in the upper bound in \eqref{eq:max-phi-t-0-eventually-upper-bounded} and \eqref{eq:max-phi-t-eps-eventually-upper-bounded} is the first order of the maximum of $\Phi_t^\GFF$.
Thus, the statements say that a.s.\ the maximum eventually does not exceed the first order by $\rho \log \frac{1}{L_t}$ for any $\rho>0$.
Using a more careful analysis, we believe that \eqref{eq:max-phi-t-0-eventually-upper-bounded} and \eqref{eq:max-phi-t-eps-eventually-upper-bounded} hold with $\rho=0$ and possibly even
with the term in the right side of the inequalities replaced by $\fom \log \frac{1}{L_t} - c\log \log \frac{1}{L_t}$ for some $c>0$ in analogy with the result for the Branching Random Walk \cite{Hu2009Minimal}.
A rigorous proof of such an estimate is in essence not much different from the proof of Theorem \ref{thm:lv-good-event},
but would require more elaborate calculation,
which, to keep this article at a reasonable length,
we prefer not to include here.
\end{remark}

An immediate consequence of Theorem \ref{thm:lv-good-event} is the following control of the maxima of $\Phi^\GFF$ and $\Phi^{\GFF_\epsilon}$ along all scales $t>0$.

\begin{corollary}
\label{cor:gff-max-upper-bound-along-all-scales}
For any $\rho>0$ there are positive and a.s.\ finite random variables $M_{\rho}$, such that for all $t>0$,
\begin{align}
\label{eq:gff-max-0-upper-bound-along-all-scales}
\max_{\Omega} \Phi_t^{\GFF_0} \leq ( \fom + \rho )\log \frac{1}{L_t} + M_\rho ,
\\
\label{eq:gff-max-eps-upper-bound-along-all-scales}
\sup_{\epsilon \leq L_t} \max_{\Omega_\epsilon} \Phi_t^{\GFF_\epsilon} \leq ( \fom + \rho )\log \frac{1}{L_t} + M_\rho .
\end{align}
\end{corollary}

In the application of Theorem \ref{thm:lv-good-event} we choose $\rho>0$ depending on $\beta$.
For our needs such a choice is possible for every $\beta \in (0,8\pi)$.
The most important consquence of Theorem \ref{thm:lv-good-event} and Corollary \ref{cor:gff-max-upper-bound-along-all-scales} in the present context is the integrability of the maximum of the renormalised exponential denoted as
\begin{equation}
\label{eq:renormalised-gff-eps-max}
\expmax_s^\epsilon = e^{\sqrt{\beta} \max_{\Omega_\epsilon} \Phi_s^{\GFF_\epsilon}} L_{s\vee \epsilon^2}^{\beta/4\pi}.
\end{equation}
Note that since $\Phi_s^{\GFF_0}$ is a smooth field for every $s>0$,
we have that $\expmax_s^\epsilon$ is also well-defined for $\epsilon=0$, i.e., we set
\begin{equation}
\label{eq:renormalised-gff-0-max}
\expmax_s^0 = e^{\sqrt{\beta} \max_{\Omega} \Phi_s^{\GFF_0}} L_s^{\beta/4\pi}.
\end{equation}
Thus, with $\hdphs$ as in \eqref{eq:parameter-hoelder-phase} we have by Corollary \ref{cor:gff-max-upper-bound-along-all-scales} 
\begin{equation}
\label{eq:expmax-0-upper-bound-with-hdphs}
\expmax_s^0 \leq e^{\sqrt{\beta} M_\rho} e^{\sqrt{\beta} (\fom + \rho) \log \frac{1}{L_s} - \frac{\beta}{4\pi} \log \frac{1}{L_s}}
= e^{\sqrt{\beta} M_\rho} L_s^{-(\hdphs + \sqrt{\beta}\rho)} ,
\end{equation}
and similarly, 
\begin{equation}
\label{eq:expmax-eps-upper-bound-with-hdphs}
\sup_{\epsilon\leq L_s} \expmax_s^\epsilon \leq e^{\sqrt{\beta} M_\rho} L_s^{-(\hdphs + \sqrt{\beta}\rho)}.
\end{equation}
Since $L_s = \sqrt{s}$ for $s\leq 1$, this shows that we can, for $\hdphs\in (0,2)$, choose $\rho>0$, such that $\sup_{\epsilon\leq L_s} \expmax_s^\epsilon$ and $\expmax_s^0$ are a.s.\ integrable at $0$.
We record this observation in a separate statement.
\begin{corollary}
\label{cor:expmax-integrable-at-0}
Let $\beta \in (0, 8\pi)$ and let $\hdphs$ be as in \eqref{eq:parameter-hoelder-phase}.
For any $\delta \in (0,2-\hdphs)$, there is a deterministic constant $C_{\beta, \delta}>0$ such that
\begin{equation}
\label{eq:expmax-0-integrable-at-0}
\P \Big(
\exists t_0 >0 \colon \forall t\in(0, t_0) \colon \int_0^t \expmax_s^0 ds 
\leq
C_{\beta, \delta} L_t^{\delta} 
\Big) = 1,
\end{equation}
and similarly
\begin{equation}
\label{eq:expmax-eps-integrable-at-0}
\P \Big(
\exists t_0 >0 \colon \forall t\in(0, t_0) \colon \int_0^t \sup_{\epsilon\leq L_s} \expmax_s^\epsilon ds 
\leq
C_{\beta, \delta} L_t^{\delta} 
\Big) = 1 .
\end{equation}
\end{corollary}

\section{Proofs of the main results}
\label{sec:proofs}

In this section we give the proof to the statements in Section \ref{sec:coupling-lattice}.
Since Theorem \ref{thm:lv-good-event} is used in the proof of the other results,
we present its proof first.

\subsection{Control of the Gaussian maximum: Proof of Theorem \ref{thm:lv-good-event}}
\label{ssec:proof-max-control}

The idea of the proof of Theorem \ref{thm:lv-good-event} is to view the field $\Phi_t^{\GFF}$ in an approximate sense as a collection of $O(L_t^{-2})$ independent random variables.
Using standard arguments on the maximum such as Fernique's criterion and Borell's inequality,
one can show that the upper bounds in \eqref{eq:max-phi-t-0-eventually-upper-bounded} and \eqref{eq:max-phi-t-eps-eventually-upper-bounded} hold along deterministic subsequences $(t_k)_{k\in \N}$ with $t_k \to 0$ as $k\to \infty$,
and then discuss control the intermediate values $t\in [t_{k+1}, t_k]$ with similar arguments.
To keep this section short, we carry out both steps simultaneously by subdividing the index spaces $\Omega_\epsilon\times (0,1]$ and $\Omega \times (0,1]$ as follows.

We first subdivide $\Omega$ into subboxes of sidelength $\bxlen{t}\approx L_t$, so that the field $\Phi_t^\GFF$ is approximately constant on each subbox.
We refer to this construction as the $t$-subdivision.
In order for the subdivision to be compatible with the length of the torus,
we choose
$\bxlen{t} = 2^{-\ceil{ \log_2 \frac{1}{L_t} }}$,
so that $\bxlen{t} \leq L_t$ with $\bxlen{t} = L_t\big(1 + o_t(1)\big)$,
where $o_t(1)\to 0$ as $t\to 0$.
Moreover, we write $N_t$ for the number of boxes in the $t$-subdivision and note that, since $L_t = \sqrt{t}\wedge 1/m$, we have
\begin{equation}
N_t = \frac{1}{\bxlen{t}^2} = \frac{1}{t} \big(1+o_t(1)\big)
\end{equation}
as $t\to 0$.
For $i= 1, \ldots, N_t$, we denote the $i$-th subbox in the $t$-subdivision by $\bx{t}{i}$ and its centre by $z_i$.
Moreover, we write $\bx{t}{i,\epsilon}=\epsilon \Z^2 \cap \bx{t}{i}$ for the $\epsilon$-discretisation of $\bx{t}{i}$ and,
recalling the notation introduced above Theorem \ref{thm:lv-dcnablav-limit},
we write $z_i^\epsilon \in \Omega_\epsilon$ for its centre.
To keep the notation clear, we omit the dependence on $t$ from the notation of $z_i$ and $z_i^\epsilon$.
Moreover, we use the convention $z_i^0 = z_i$ and $\bx{t}{i,0}=\bx{t}{i}$.

For a given $\rho >0$ we define the sequence $(t_k)_{k\in \N}$ by $t_k = k^{-p}$ with $p>\frac{1}{\sqrt{2\pi}\rho}$
and subdivide $\Omega\times  (0,1]$ into disjoint sets $\bx{t_{k+1}}{i} \times (t_{k+1}, t_k]$, $k\in \N$, $i=1, \ldots N_{t_{k+1}}$.
%
As we see below the field $\Phi_t^\GFF(x)$ is approximately constant for $(x,t)\in \bx{t}{i} \times (t_{k+1},t_k]$.

Similarly, we subdivide $\Omega_\epsilon \times (0,1]$ into sets $\bx{t_{k+1}}{i,\epsilon}\times(t_{k+1}, t_k]$, $k\in \N$, $i=1,\ldots, N_{t_{k+1}}$.
Then, for $\epsilon>0$ and $\epsilon=0$ we define the fields
\begin{equation}
\label{eq:definition-G-on-boxes}
\Gs{i,\epsilon}_t(x) = \Phi_t^{\GFF_\epsilon}(x) - \Phi_{t_k}^{\GFF_\epsilon}(z_i^\epsilon), \qquad  i=1, \ldots, N_{t_{k+1}}, \, 
(x,t) \in \bx{t_{k+1}}{i,\epsilon}\times(t_{k+1},t_k]. 
\end{equation}
Note that, since $\Phi_t^{\GFF}$ is invariant under translations on $\Omega$,
the fields $\Gs{i,\epsilon}_{ t}$ are identically distributed.
For the proof of \eqref{eq:max-phi-t-eps-eventually-upper-bounded} it is necessary to control also the lattice fields $\Phi_t^{\GFF_\epsilon}$ for different regularisations.
%
Then we additionally view $\Phi_t^{\GFF_\epsilon}$ as a process of $\epsilon>0$,
rather than treating $\epsilon>0$ as a fixed parameter.
Note that for different $\epsilon_1, \epsilon_2>0$,
the fields $\Phi_t^{\GFF_{\epsilon_1}}$ and $\Phi_t^{\GFF_{\epsilon_2}}$ are defined on different index sets.
It is therefore convenient to use the piecewise constant extension $E_\epsilon$ defined in \eqref{eq:def-piecewise-constant-extension} and view them all as fields on $\Omega$.
When $\Phi_t^{\GFF_\epsilon}$ is seen also as a field indexed by $\epsilon\leq L_t$,
we use a separate notation and define
\begin{equation}
\bGs{i,\epsilon}_t(x) = \Phi_t^{\GFF_\epsilon} (x^\epsilon) - \Phi_{t_k}^{\GFF_\epsilon} (z_i^\epsilon), \qquad  i= 1,\ldots N_{t_{k+1}}, \, (x,t) \in \bx{t}{i} \times (t_{k+1}, t_k], \, \epsilon\leq L_t.
\end{equation}

We put forward the following result, which allows us to control the change of $\bGs{i,\epsilon}_t$ in $\epsilon\leq L_t$,
and thus, reduces the discussion of $\bGs{i,\epsilon}_{t}$ in the proof of \eqref{eq:max-phi-t-eps-eventually-upper-bounded} to that of $\Gs{i,\epsilon}_{t}$.
Its proof involves rather tedious estimates on the Fourier coefficients of $\dot c_t^\epsilon$, and appears in Appendix \ref{app:eps-change-of-gff}.
Here and below, we use the short hand notation
\begin{equation}
\label{eq:def:gaussian-field-indexed-in-epsilon}
\rem_t (x,\epsilon) = \Phi_t^{\GFF_\epsilon}(x^\epsilon) \qquad x\in \Omega, \, \epsilon \leq L_t.
\end{equation}

\begin{lemma}
\label{lem:gaussian-fields-eps-change}
Let $t>0$ and let $\rem_t$ be as in \eqref{eq:def:gaussian-field-indexed-in-epsilon}.
Then, for $\epsilon_1 \leq \epsilon_2 \leq L_t$ and $x\in \Omega$, we have
\begin{equation}
\label{eq:gaussian-fields-eps-change}
\E\big[ ( \rem_t(x,\epsilon_2) - \rem_t(x,\epsilon_1) )^2\big] \lesssim (\epsilon_2- \epsilon_1)/L_t .
\end{equation}
In particular, we have, for $x,y \in \bx{t}{i}$,
\begin{align}
\label{eq:relation-eps-gs-to-non-eps-gs}
\E[\big( \bGs{i,\epsilon_2}_t(x) - \bGs{i,\epsilon_1}_t(y)\big)^2]
\lesssim 
\frac{\epsilon_2 - \epsilon_1}{L_t} + 
\E[ \big( \Gs{i,\epsilon_1}_t(x^{\epsilon_1}) - \Gs{i,\epsilon_1}_t(y^{\epsilon_1}) \big)^2] . 
\end{align}
\end{lemma}

The following result gives a uniform control over $\Gs{i,\epsilon}_t$ for each $i=1,\ldots, N_{t_{k+1}}$,
which then translates to $\bGs{i,\epsilon}_t$ by Lemma \ref{lem:gaussian-fields-eps-change}.

\begin{lemma}
\label{lem:max-g-exp-bounded-fernique}
Let $\epsilon \geq 0$ and let $\Gs{i,\epsilon}_t$ be as in \eqref{eq:definition-G-on-boxes}.
Then, for all $(x,t),(y,u) \in \bx{t_{k+1}}{i,\epsilon} \times (t_{k+1},t_k]$,
we have 
\begin{align}
\label{eq:Gi-difference-l2}
\E\big[\big(\Gs{i,\epsilon}_t(x) - \Gs{i,\epsilon}_u(y) \big)^2 \big]
&\lesssim \frac{|x-y|}{ L_{t_{k+1}}} + \big| \log \frac{1}{t} - \log \frac{1}{u} \big|,
\\
\label{eq:Gi-variance}
\E\big[\big(\Gs{i,\epsilon}_t(x)\big)^2\big] &\lesssim 1 .
\end{align}
In particular, there is a constant $C>0$, which is independent of $t>0$, such that
\begin{align}
\label{eq:max-g-exp-bounded-fernique}
\E\Big[ \sup_{t\in (t_{k+1},t_k]} \max_{x\in \bx{t_{k+1}}{i,\epsilon}} \Gs{i,\epsilon}_t(x) \Big] \leq C ,
\\
\label{eq:max-bg-exp-bounded-fernique}
\E\Big[ \sup_{t\in (t_{k+1}, t_k]} \sup_{\epsilon \leq L_{t_{k}}} \max_{x\in \bx{t_{k+1}}{i,\epsilon}} \bGs{i,\epsilon}_t(x) \Big] \leq C .
\end{align}
Finally, with $C$ as above, there exists a constant $\bar c>0$ so that for any $r > 2C$,
\begin{align}
\label{eq:max-g-tail}
\P\big( \sup_{t\in (t_{k+1}, t_k]} \max_{x\in \bx{t_{k+1}}{i,\epsilon}} \Gs{i,\epsilon}_t(x)\geq r \big) \leq \exp(-r^2/\bar c),
\\
\label{eq:max-bg-tail}
\P\big(  \sup_{t\in (t_{k+1}, t_k]} \sup_{\epsilon\leq L_{t_k}}\max_{x\in \bx{t_{k+1}}{i,\epsilon}}  \bGs{i,\epsilon}_t(x)\geq r \big) \leq \exp(-r^2/\bar c).
\end{align}
\end{lemma}

In the proof of this statement we use standard estimates on the heat kernel $p_t^{\epsilon,1}$ on $\Omega_\epsilon$ and $p_t^{0,1}$ on $\Omega $ following the notation introduced in \cite[Appendix A]{MR4399156} for tori of general length.
The superscript $1$ refers to the fact that $\Omega$ is the unit torus.
For convenience, these estimates are also recorded in Appendix \ref{app:covariance-estimates}.
Note that we have
\begin{align}
\dot c_t^\epsilon(x,y) &= p_t^{\epsilon,1}(x-y), \qquad x,y\in \Omega_\epsilon, \\
\dot c_t^0 (x,y) &= p_t^{0,1} (x-y), \qquad  x,y\in \Omega.
\end{align}

\begin{proof}
We treat the cases $\epsilon>0$ and $\epsilon=0$ simultaneously.

For $(x,t),(y,t) \in \bx{t_{k+1}}{i,\epsilon} \times (t_{k+1},t_k]$ and $t\geq \epsilon^2$ we have
\begin{align}
\label{eq:g-spatial-difference-l2}
\E\big[\big( \Gs{i,\epsilon}_t(x) - \Gs{i,\epsilon}_t(y)  \big)^2\big] 
&= \E\big[ \big( \Phi_t^{\GFF_\epsilon}(x) - \Phi_t^{\GFF_\epsilon}(y) \big)^2 \big]
= \int_t^\infty \big( \dot c_s^\epsilon(x,x) + \dot c_s^\epsilon(y,y) - 2\dot c_{s}^\epsilon(x,y) \big)ds
\nnb
&=2 \int_t^\infty \big( p_s^{\epsilon,1}(0) - p_s^{\epsilon,1}(x-y) \big)e^{-s m^2} d\tau
\lesssim \frac{|x-y|}{L_{t_{k+1}}} ,
\end{align}
where the last estimate follows from \eqref{e:ptbounds} applied with $|\alpha| = 1$, since $0 <|x-y|\leq \bxlen{t_{k+1}} \leq  L_{t_{k+1}}$; see also \cite[A.6]{MR4303014}.
Next, assuming that $t\leq u$ without restriction, we have
\begin{align}
\label{eq:g-scale-difference-l2}
\E\big[\big( \Gs{i,\epsilon}_t(y) - \Gs{i,\epsilon}_u(y)  \big)^2\big] 
&= \E\big[ \big( \Phi_t^{\GFF_\epsilon}(y) - \Phi_u^{\GFF_\epsilon}(y) \big)^2 \big]
\nnb
&= \int_t^u \big( \dot c_s^\epsilon(y,y)  \big)ds\lesssim \int_t^u \frac{1}{r} dr = \log \frac{1}{t} -\log \frac{1}{u}.
\end{align}
Combining \eqref{eq:g-spatial-difference-l2} and \eqref{eq:g-scale-difference-l2} 
the claimed inequality \eqref{eq:Gi-difference-l2} follows.

Similarly, we have, for $(x,t)\in \bx{t_{k+1}}{i,\epsilon}\times (t_{k+1}, t_k]$,
\begin{align}
\E\big[\big(\Gs{i,\epsilon}_t(x) \big)^2\big] 
&\lesssim
\E\big[ \big(\Phi_t^{\GFF_\epsilon}(x) - \Phi_t^{\GFF_\epsilon}(z_i^\epsilon) \big)^2 \big]
+
\E\big[\big( \Phi_{t}^{\GFF_\epsilon}(z_i^\epsilon)-\Phi_{t_k}^{\GFF_\epsilon}(z_i^\epsilon)\big)^2\big]
\nnb
&\lesssim \frac{|x-z_i^\epsilon|}{L_{t_{k+1}}}
+
\big( \log\frac{1}{t} -\log\frac{1}{t_k}\big) \lesssim 1.
\end{align}
Thus, the estimate \eqref{eq:max-g-exp-bounded-fernique} follows from Fernique's criterion \cite{MR0266263}.

Similarly,
the estimate \eqref{eq:max-bg-exp-bounded-fernique} is obtained from previous estimates together with Lemma \ref{lem:gaussian-fields-eps-change} and Fernique's criterion.

Finally, the bound \eqref{eq:max-g-tail} follows from \eqref{eq:Gi-variance} and \eqref{eq:max-g-exp-bounded-fernique} and the Borell-Tsirelson-Sudakov inequality.
Similarly, \eqref{eq:max-bg-tail} follows with the same argument using now \eqref{eq:Gi-variance} and together with \eqref{eq:relation-eps-gs-to-non-eps-gs} and \eqref{eq:max-bg-exp-bounded-fernique}.
\end{proof}

\begin{proof}[Proof of Theorem \ref{thm:lv-good-event}]
We first give the proof of \eqref{eq:max-phi-t-0-eventually-upper-bounded}.
Note first that by the standard Gaussian estimate
\begin{equation}
\P(Z \geq a)
\leq 
e^{-\frac{a^2}{2\sigma^2}}, \qquad Z\sim \cN(0,\sigma^2), \, a\geq 0
\end{equation}
applied with $Z= \Phi_{t_k}^{\GFF_0}(z_i)$ satisfying  $\var(\Phi_{t_k}^{\GFF_0}(z_i)) = \frac{1}{2\pi} \log \frac{1}{L_{t_k}} + O_m(1)$,
we obtain by a union bound, for any $\rho'\in (0,\rho)$,
\begin{align}
\P\big(
\max_{i=1,\ldots, N_{t_{k+1}}} \Phi_{t_k}^{\GFF_0}(z_i) \geq \big(\fom + \rho' \big) \log \frac{1}{ L_{t_k}}
\big)
&\lesssim
N_{t_{k+1}} \exp\Big(
-\frac{1}{2} \frac{(\fom + \rho')^2}{\frac{1}{2\pi}}\log\frac{1}{L_{t_k}} 
\Big)
\nnb
&\lesssim_{\rho'}
\big(L_{t_k} \big)^{\sqrt{8\pi} \rho'}\leq k^{-p\sqrt{2\pi} \rho'}.
\end{align}
We now choose $\rho'>0$ close enough to $\rho$ so that $p\sqrt{2\pi }\rho'>1$,
which is possible due to our choice of $p$.
Then we have
\begin{equation}
\label{eq:gffs-along-tk-summable}
\sum_{k\in \N} \P\big(
\max_{i=1,\ldots, N_{t_{k+1}}} \Phi_{t_k}^{\GFF_0}(z_i) \geq \big(\fom + \rho' \big) \log \frac{1}{L_{t_k}}
\big)<\infty.
\end{equation}
With $\Gs{i,0}_t$ as in \eqref{eq:definition-G-on-boxes}, we have from \eqref{eq:max-g-tail} that
\begin{align}
\label{eq:tail-bounds-g-summable}
\sum_{k\in \N} \sum_{i=1,\ldots,N_{t_{k+1}}}
\P\big(\sup_{t\in (t_{k+1}, t_k]} \sup_{x\in \bx{t_{k+1}}{i}} &\Gs{ i,0}_t(x)\geq (\rho-\rho') \log \frac{1}{L_{t_k}}\big)
\nnb
&\lesssim
\sum_{k\in \N} N_{t_{k+1}}  \exp\Big(- (\rho-\rho')^2 (\log \frac{1}{L_{t_k}})^2/\bar c \Big)<\infty.
\end{align}
Now, we observe that 
\begin{align}
\label{eq:max-large-inclusion}
\Big\{ &\exists t\in (t_{k+1},t_k], {x \in \Omega} \colon \Phi_t^{\GFF_0}(x)\geq \big(\frac{2}{\sqrt{2\pi}}+
\rho\big)\log \frac{1}{L_t}\Big\}
\nnb
&\subseteq
\Big\{ \exists i\in \{1,\ldots,N_{t_{k+1}}\} \colon
\Phi_{t_k}^{\GFF_0}(z_i) \geq \big(\frac{2}{\sqrt{2\pi}} +
\rho'\big) \log \frac{1}{L_{t_k}}
\nnb
&\qquad \qquad \qquad\qquad \qquad \qquad \qquad \qquad\, \text{ or } \,
\sup_{t\in (t_{k+1},t_k]} \sup_{x\in \bx{t_{k+1}}{i} } \Gs{i,0}_t(x) \geq (\rho-\rho')\log\frac{1}{L_{t_k}} \Big\}.
\end{align}
Using \eqref{eq:gffs-along-tk-summable} and \eqref{eq:tail-bounds-g-summable},
we conclude from \eqref{eq:max-large-inclusion} that
\begin{equation} 
\label{eq:max-gffs-on-intervals-summable}
\sum_{k\in \N} \P\big(
\sup_{ t\in (t_{k+1},t_k]} \sup_{x\in \Omega}
\Phi_t^{\GFF_0}(x)
\geq \big(\frac{2}{\sqrt{2\pi}} + \rho\big) \log \frac1{L_t}\big) < \infty.
\end{equation}
By the Borel-Cantelli lemma, this completes the proof of \eqref{eq:max-phi-t-0-eventually-upper-bounded}.

Similarly, considering $\epsilon\leq L_t$ as an extra parameter in addition to
$(t,x)$ and using \eqref{eq:max-bg-tail} instead of \eqref{eq:max-g-tail},
we obtain analogously to \eqref{eq:max-gffs-on-intervals-summable} that 
\begin{equation}
\label{eq:max-gffs-eps-on-interval-summable}
\sum_{k\in \N} \P\big( \sup_{t\in (t_{k+1},t_k]}  \sup_{\epsilon\leq L_{t_k}} \max_{x\in \Omega_\epsilon}  \Phi_t^{\GFF_\epsilon}(x) \geq \big(\frac{2}{\sqrt{2\pi}} +
\rho\big)\log \frac1{L_t}\big)<\infty.
\end{equation}
Another application of the Borel-Cantelli Lemma then yields that
\begin{equation}
\P\Big(\exists t_0 >0 \colon \forall t\in (0, t_0) \colon \sup_{\epsilon \leq L_t} \max_{\Omega_\epsilon} \Phi_t^{\GFF_\epsilon} \leq (\fom+\rho)\log \frac{1}{L_t}  \Big) = 1.
\end{equation}
\end{proof}

\begin{proof}[Proof of Corollary \ref{cor:gff-max-upper-bound-along-all-scales}]
For $\rho>0$ let $t_0$ and $t_0'$ be as in \eqref{eq:gff-max-0-upper-bound-along-all-scales} and \eqref{eq:gff-max-eps-upper-bound-along-all-scales}.
Then, for $t>0$,
\begin{align}
\max_{\Omega} \Phi_t^{\GFF_0} 
&\leq
(\fom + \rho) \log \frac{1}{L_t} + \sup_{s\geq t_0} \max_{\Omega} \Phi_s^{\GFF_0},
\\
\sup_{\epsilon\leq L_t} \max_{\Omega_\epsilon} \Phi_t^{\GFF_\epsilon} 
&\leq
(\fom + \rho) \log \frac{1}{L_t} + \sup_{s\geq t_0'} \sup_{\epsilon>0} \max_{\Omega_\epsilon} \Phi_s^{\GFF_\epsilon}.
\end{align}
Now, we have that $\max_{\Omega} \Phi_s^{\GFF_0} \to 0$ and as $s\to 0$ a.s.\
and similarly $\sup_{ \epsilon\leq L_s} \max_{\Omega_\epsilon}  \Phi_s^{\GFF_\epsilon} \to 0$ a.s.\ as $s\to 0$.
Moreover, since $\max_{\Omega} \Phi_t^{\GFF_0}$ and $\sup_{ \epsilon\leq L_t} \max_{\Omega_\epsilon} \Phi_t^{\GFF_\epsilon}$ are continuous processes on $(0,\infty)$,
and since $t_0, t_0'>0$ a.s.,
the statement follows with
\begin{equation}
M_\rho =
\sup_{s\geq t_0} \max_{\Omega} \Phi_s^{\GFF_0} 
+
\sup_{s\geq t_0'} \sup_{\epsilon\leq L_s} \max_{\Omega_\epsilon} \Phi_s^{\GFF_\epsilon} .
\end{equation}
\end{proof}

\begin{proof}[Proof of Corollary \ref{cor:expmax-integrable-at-0}]
We first show \eqref{eq:expmax-0-integrable-at-0}.
For $\delta \in (0, 2-\hdphs)$ we choose $\rho>0$ according to
\begin{equation}
\sqrt{\beta}\rho = (2-\hdphs-\delta).
\end{equation}
Let $\cA_{\rho}$ be the event in \eqref{eq:max-phi-t-0-eventually-upper-bounded}.
Then we have on $\cA_\rho$ for $t\leq t_0$
\begin{align}
\label{eq:integral-A-small-scales}
\int_0^t \expmax_s^0 ds
\leq \int_0^t e^{ \big(\sqrt{\beta} (\frac{2}{\sqrt{2\pi}} + \rho) - \frac{\beta}{4\pi}  \big)\log \frac{1}{L_s} } ds
= \int_0^t L_s^{-(\hdphs  + \sqrt{\beta} \rho ) } ds.
\end{align}
Now, by the choice of $\rho$, we have that $\hdphs + \sqrt{\beta} \rho = 2-\delta <2$,
so that the integrand of the last integral is integrable at $0$.
Thus, there is an exlpicit non-random constant $C_{\beta,\delta}>0$ such that
\begin{equation}
\int_0^t \expmax_s^0 ds 
\leq C_{\beta,\delta} L_t^{2-(\hdphs + \sqrt{\beta} \rho)}
= C_{\beta,\delta} L_t^{\delta},
\end{equation}
which proves \eqref{eq:expmax-0-integrable-at-0}.

For the proof of \eqref{eq:expmax-eps-integrable-at-0} we use \eqref{eq:max-phi-t-eps-eventually-upper-bounded} instead of \eqref{eq:max-phi-t-0-eventually-upper-bounded} and follow the same argument.
\end{proof}

\subsection{Strong SDE existence: Proof of Theorem \ref{thm:lv-coupling}}

The following proof is a variation of the argument for the proof of \cite[Theorem 3.1]{MR4399156},
where the analogous statement is proved in the case of the sine-Gordon model using Picard interation.
We emphasise that in the present case the the SDEs \eqref{eq:Phi-Lv-eps-coupling} and \eqref{eq:Phi-Lv-0-coupling} do not have Lipshitz coefficients,
and thus, standard results for existence and uniqueness of SDEs do not apply to our setting.
As we demonstrate below, using the negativity of the difference field together with the monotonicity of the exponential function,
we are able to give an estimate on the Picard iterates by terms which only depend on the driving Gaussian process $\Phi^\GFF$.
This then allows to construct solutions to the SDEs pathwise, i.e., for almost every sample path of the process $(\Phi_t^\GFF)_{t\in[0,\infty]}$.

\begin{proof}[Proof of Theorem~\ref{thm:lv-coupling}]
We only discuss the case $\epsilon=0$. The case $\epsilon>0$ is almost identical when invoking the statement for $\epsilon>0$ instead of the analogues for $\epsilon=0$.

For $t>0$ and $D,E \in C([t,\infty), C(\Omega))$, let
\begin{align}
\label{e:F-Picard-def}
F_t^0(D,E) = -\int_t^\infty \dot c_s^0\nabla v_s^0(E_s+D_s) \, ds.
\end{align}
To prove the statement we show that $F_t^0(\cdot, E)$ has a unique fixed point for $E=\Phi^{\GFF_0}$.
It is clear that any solution to the SDE \eqref{eq:Phi-Lv-0-coupling} satisfies $\Phi_t^{\Lv_0} - \Phi_t^{\GFF_0} \leq 0$,
since $\big[\dot c_s^0 \nabla v_s^0\big](\varphi) \geq 0$ for all $\varphi \in C(\Omega)$ and $s\geq t$ by \eqref{eq:lv-dcnablav-0-upper-bound}.
Hence, we may restrict in what follows to $D\in L^\infty(\Omega)$ with $D_s(x) \leq 0$ for all $x\in \Omega$ and $s\geq t$.
Let $E \in C([t,\infty),C(\Omega))$ and $D,\tilde D \in L^\infty(\Omega)$ with $D,D' \leq 0$.
By \eqref{eq:dcnablav-0-cont} applied with $\varphi = E_s + D_s$ and $\varphi' = E_s + D_s'$ 
we then have
\begin{align}
\label{e:F-Lip}
\norm{F_t^0(D,E)-F_t^0(\tilde D,E)}_{L^\infty(\Omega)}
\lesssim_{\beta, \lambda}
\int_t^\infty
O(\theta_s) 
e^{\sqrt{\beta} \max E_s} 
L_s^{\beta/4\pi}
\|D_s-\tilde D_s\|_{L^\infty(\Omega)} \, 
 ds ,
\end{align}
where we used that $\| \dot c_s^0 \|_\infty \lesssim \theta_s$,
which holds by
\eqref{e:c-limit} with $k=0$.
We emphasise that the process multiplying the norm inside the integral in \eqref{e:F-Lip} only depends on $E$, and not on $D,D'$.
With the choice $E=\Phi^{\GFF_0}$ we now show uniqueness and existence of the solution to the SDE \eqref{eq:Phi-Lv-0-coupling} for a.e.\ realisation of $\Phi^{\GFF_0}$ on $[t,\infty)$ for $t >0$, and then extend this solution to $[0,\infty)$ by taking $t\to 0$.

Suppose first that there are two solutions $\Phi^{\Lv_0}$ and $\tilde \Phi^{\Lv_0}$ to \eqref{eq:Phi-Lv-0-coupling}
satisfying $D\coloneqq\Phi^{\Lv_0}-\Phi^{\GFF_0} \in C_0([t_0,\infty),L^\infty(\Omega))$
and $\tilde D\coloneqq\tilde\Phi^{\Lv_0}-\Phi^{\GFF_0} \in C_0([ t_0,\infty),L^\infty(\Omega))$.
Then, by \eqref{e:F-Lip}, we have
\begin{align}
\norm{D_t-\tilde D_t}_{L^\infty(\Omega)}
&= \norm{F_t^0(D,\Phi^{\GFF_0})-F_t^0(\tilde D,\Phi^{\GFF_0})}_{L^{\infty}(\Omega)}
\nnb
&\lesssim_{\beta,\lambda} \int_t^\infty O(\theta_s) L_s^{\beta/4\pi} e^{\sqrt{\beta} \max \Phi_s^{\GFF_0}} \norm{D_s-\tilde D_s}_{L^\infty(\Omega)} \, ds \nnb
&\leq \int_t^\infty O(\theta_s) \expmax_s^0  \norm{D_s-\tilde D_s}_{L^\infty(\Omega)} \, ds ,
\end{align}
where $\expmax^0$ is the process defined in \eqref{eq:renormalised-gff-0-max}.
Thus, $f(t) = \norm{D_t-\tilde D_t}_{L^\infty(\Omega)}$ is bounded with $f(t)\to 0$ as $t\to\infty$
and it satisfies
\begin{equation}
f(t) \leq a + b\int_t^\infty \theta_s \expmax_s^0   f(s) \, ds
\end{equation}
with $a=0$ and a deterministic constant $b>0$.
By Corollary \ref{cor:gff-max-upper-bound-along-all-scales} we have 
\begin{equation}
\int_t^\infty \theta_s \expmax_s^0 \, ds
\leq \int_0^\infty \theta_s \expmax_s^0 \, ds < \infty  \qquad \text{a.s.,}
\end{equation}
and thus, Gronwall's inequality implies that 
\begin{equation}
f(t) \leq a \exp\pa{b\int_0^\infty \theta_s \expmax_s^0 \, ds} = 0 \qquad \text{a.s.}
\end{equation}
This shows $D=\tilde D$ a.s.\ on $[t, \infty)$.

To show the existence of a solution, we use Picard iteration.
Let $\|D\|_t = \sup_{s\geq t} \norm{D_s}_{L^\infty(\Omega)}$.
For any $t>0$ and $E\in C([t,\infty), C(\Omega))$ fixed,
set $D^0=0$ and $D^{n+1}=F_t^0(D^n,E)$ with $F_t^0$ as in \eqref{e:F-Picard-def}.
Then
\begin{align}
\|D^1\|_t = \|F_t^0(0,E)\|_t
&\lesssim_{\beta,\lambda} \int_t^\infty 
O(\theta_s) \expmax_s^0 ds,
\\
\|D^{n+1}-D^n\|_t
&\lesssim_{\beta,\lambda} \int_t^\infty 
O(\theta_s) \expmax_s^0  \|D^{n}-D^{n-1}\|_{s} ds,
\end{align}
and from the elementary identity 
\begin{equation}
\int_t^\infty ds \, g(s) \pa{\int_s^\infty ds' \, g(s')}^{k-1} =  \frac{1}{k} \pa{\int_t^\infty ds\, g(s)}^{k}
\end{equation}
applied with the process $g(s) = O(\theta_s) \expmax_s^0$, we conclude that
\begin{equation}
\|D^{n+1}-D^n\|_t
\lesssim_{\beta,\lambda} \frac{1}{n!} \Big( \int_t^\infty O(\theta_s) \expmax_s^0 ds \Big)^n 
.
\end{equation}
By Corollary \ref{cor:gff-max-upper-bound-along-all-scales}, the integral on the right hand side is finite for every $t\geq 0$ a.s.,
which shows that the bound on right-hand side is summable in $n\in \N$.
Thus, we have that $D^n \to D$ for some $D = D^*(E) \in C_0([t,\infty),C(\Omega))$ in $\|\cdot\|_t$,
and the limit satisfies $F_s^0(D^*(E),E)=D^*(E)_s$ for $s\geq t$.
By uniqueness, $D^*(E)_t$ is consistent in $t$ and we can thus define
$D^*(E) \in C_0((0,\infty), L^\infty(\Omega))$.
In summary, $\Phi_t^{\Lv_0} = D^*(\Phi_t^{\GFF_0}) + \Phi_t^{\GFF_0}$ is the desired solution for $t>0$.

Further noticing that, for $t>0$,
\begin{equation}
D^*(\Phi^{\GFF_0})_t = -\int_t^\infty \dot c_s^0 \nabla v_s^0(\Phi_s^{\Lv_0}) \, ds,
\end{equation}
we may extend the solution to $t=0$ by continuity.
Indeed, by \eqref{eq:lv-dcnablav-0-upper-bound},
and the fact that $\Phi_t^{\Lv_0} = \Phi_t^{\GFF_0} + \Phi_t^{\Delta_0} \leq \Phi_t^{\GFF_0}$ for all $t>0$,
we have 
\begin{equation}
\Big\| \int_0^t \dot c_s^0 \nabla v_s^0(\Phi_s^{\Lv_0}) \, ds \Big\|_{L^\infty(\Omega)}
\leq
\int_0^t \big\|\dot c_s^0 \nabla v_s^0 (\Phi_s^{\Lv_0}) \big\|_{L^\infty(\Omega)} ds
\lesssim_{\beta,\lambda} \int_0^t O(\theta_s)
\expmax_s^0 ds .
\end{equation}
By Corollary \ref{cor:expmax-integrable-at-0}
the integrand on the right hand side is a.s.\ integrable at $0$,
and thus the limit $\lim_{t\to 0} D^*(\Phi^{\GFF_0})_t$ exists a.s.\ in $C(\Omega)$.
\end{proof}

\subsection{Realisation of the Liouville measure: Proof of Theorem \ref{thm:law-phi-0}}

The following proof is again a variation of the argument given in \cite[Section 3.4]{MR4399156} for the sine-Gordon model.
We present the details here to demonstrate that this assertion also holds in the case of the Liouville model.

\begin{proof}[Proof of Theorem~\ref{thm:law-phi-0}]
Throughout this proof, $\epsilon>0$ is fixed and dropped from the notation. Moreover, we write $\nu_t \equiv \nu_t^\Lv$ for the renormalised measure defined in \eqref{eq:lv-renormalised-measure}, i.e., for  $f\colon X_\epsilon \to \R$ smooth and bounded and $t\geq 0$, we have
\begin{equation}
\E_{\nu_t} [f] = e^{v_\infty(0)} \EE_{c_\infty -c_t}[f(\zeta) e^{-v_t(\zeta)}],
\end{equation}
where $\EE_{c_t}$ is the expectation of the Gaussian measure with covariance $c_t$
and $v_t$ is as in Section~\ref{ssec:vt-eps}.
In \cite[Section~2]{MR4303014}, it is shown that
for every bounded and smooth function $f\colon X_\epsilon \to \R$,
\begin{equation}
\label{e:nu0nutf}
\E_{\nu_0}f
= \E_{\nu_t} \PP_{0,t} f,
\end{equation}
where
\begin{equation}
  \PP_{s,t} f(\phi) = e^{v_t(\phi)} \EE_{c_t-c_s}(e^{-v_s(\phi+\zeta)} f(\phi+\zeta)).
\end{equation}
By \cite[Proposition~2.1]{MR4303014}, this semigroup is characterised by its infinitesimal generator
\begin{equation}
  \LL_t = \frac12 \Delta_{\dot c_t} - (\nabla v_t,\nabla)_{\dot c_t},
\end{equation}
in the sense that
\begin{equation}
\label{eq:semigroup-p-derivatives}
\ddp{}{s} \PP_{s,t} f = -\PP_{s,t} \LL_s f,
\qquad
\ddp{}{t} \PP_{s,t} f = +\LL_t \PP_{s,t} f.
\end{equation}
We make the comment that the ergodicity assumption \cite[(2.5)]{MR4303014} on the semigroup $\PP_{s,t}$ is not needed to have \eqref{eq:semigroup-p-derivatives}, 
but only enters in the proof of \cite[(2.10)]{MR4303014}.
While it is not hard to show that the ergodicity assumption also holds for in the present case,
we only need the weaker statement below, namely that $\nu_t$ converges weakly to $\delta_0$ as $t\to \infty$.
We give a direct proof of this fact further below. 

As shown in \cite[Remark~2.2]{MR4303014}, the semigroup $\PP_{s,t}$ has a stochastic representation, which we recall here for convenience and to further introduce notation.
Let $(W_t)$ be an $X_\epsilon$-valued Brownian motion normalised so that
$(W_t(x))$ is a Brownian motion with quadratic variation $t/\epsilon^2$
for each $x\in \Omega_\epsilon$, with completed forward and backward filtrations $(\cF_t)$
and $(\cF^t)$.
For any $T>0$, the time-reversed process $\tilde W^{T}_t = W_T-W_{T-t}$ is then again a Brownian motion.
Let $\tilde\Phi^T$ be the unique strong solution to the following (forward) SDE:
\begin{equation}
  d\tilde\Phi_t^T = -\dot c_{T-t} \nabla v_{T-t}(\tilde\Phi_t^T) \, dt + q_{T-t} \, d\tilde W_t^T,
  \qquad (0 \leq t \leq T).
\end{equation}
This SDE does not have Lipschitz coefficients, as the gradient of $v_t$ is an exponential function in the field,
but using similar arguments as in the proof of Theorem \ref{thm:lv-coupling},
one can prove that it has a unique strong solution for every $T>0$.
Similarly to \cite[Remark 2.2]{MR4303014} we have by It\^o's formula 
\begin{equation}
\ddp{}{t} \E_{\tilde\Phi_0^T=\phi} f(\tilde\Phi_t^T)
=
\E_{\tilde\Phi_0^T=\phi} (L_{T-t}f(\tilde\Phi_t^T)).
\end{equation}
Therefore,
$t\mapsto \E_{\tilde\Phi_0^T=\phi} f(\tilde\Phi_t^T)$
defines an inhomogeneous Markov semigroup with generator $L_{T-t}$ and uniqueness
of such semigroups implies $\PP_{T-t,T} f(\phi) = \E_{\tilde \Phi_0^T=\phi} f(\tilde \Phi_t^T)$.
We again refer to \cite[Section 3.4]{MR4399156} for additional details.
We now reverse the $t$-direction by setting $\Phi_t^T = \tilde\Phi_{T-t}^T$.
Thus, using the change of variable $s\mapsto T-s$ on the second line of the following display, we have
\begin{align}
\label{e:varphiT-SDE}
\Phi_t^T
&= \Phi_T^T - \int_0^{T-t} \dot c_{T-s} \nabla v_{T-s}(\tilde\Phi_s^T) \, ds + \int_0^{T-t} q_{T-s} \, d\tilde W_s^T
\nnb
&= \Phi_T^T - \int_t^T \dot c_{s} \nabla v_{s}(\Phi_s^T) \, ds + \int_t^T q_{s} \, dW_s
.
\end{align}
Then $(\Phi^T_t)$ is adapted to the backward filtration $\cF^t$, and
\begin{equation}
\PP_{0,T}f (\phi)
= \E_{\Phi_T^T=\phi} f(\Phi_0^T).
\end{equation}
In particular, by the semigroup property \eqref{e:nu0nutf}, if $\Phi_T^T$ is distributed according to $\nu_T$,
then $\Phi_0^T$ is distributed according to the Liouville measure $\nu_0 = \nu^{\Lv_\epsilon}$.

It therefore suffices to show that, as $T\to\infty$, the solution \eqref{e:varphiT-SDE}
with $\Phi_T^T$ distributed according to $\nu_T$ converges to
the solution $(\Phi_t)_{t\geq 0}$ to \eqref{eq:Phi-Lv-eps-coupling}
constructed using the same Brownian motion $W$.
Setting $\Phi_\infty=0$ we first note that
\begin{equation}
\label{eq:difference-phi-phiT}
\Phi_t - \Phi_t^T
= (\Phi_\infty-\Phi_T^T)
- \int_t^T \qa{\dot c_{s} \nabla v_{s}(\Phi_s)-\dot c_{s} \nabla v_{s}(\Phi_s^T)} \, ds  - \int_T^\infty \dot c_{s} \nabla v_{s}(\Phi_s) \, ds + \int_T^\infty q_{s} \, dW_s.
\end{equation}
Since $\epsilon>0$ is fixed, we may use any norm on $X_\epsilon$ and denote it by $\norm{\cdot}$.
The first, third, and fourth terms on the right-hand side above are independent of $t$ and
we claim that they converge to $0$ in probability as $T \to\infty$.
For the first term this follows from the weak convergence of the
measure $\nu_T$ to $\nu_\infty\equiv \delta_0$.
The latter convergence holds by the following argument. Let $f\colon X_\epsilon \to \R$ be bounded and continuous.
Then we have
\begin{equation}
\label{eq:expectation-under-nu-T}
\E_{\nu_T} [f] = \frac{\int_{X_\epsilon} f(\phi) e^{-v_T(\phi)} e^{-\frac{1}{2}\phi \cdot (c_\infty-c_T)^{-1}\phi} d\phi }
{\int_{X_\epsilon} e^{-v_T(\phi)} e^{-\frac{1}{2}\phi \cdot (c_\infty-c_T)^{-1}\phi} d\phi }
= \frac{\EE_{c_\infty - c_T}[f e^{-v_T}]}{\EE_{c_\infty - c_T}[ e^{-v_T}]}
.
\end{equation}
We claim that the last right hand side converges to $f(0)$ as $T\to \infty$. 
To see this we write $Z_T$ for the normalisation constant of centred Gaussian measure on $X_\epsilon$ with covariance $c_T$, i.e.,
\begin{equation}
Z_T = \int_{X_\epsilon} e^{-\frac{1}{2}\zeta c_T^{-1} \zeta} d\zeta.
\end{equation}
Since $c_T^{-1} \downarrow c_\infty^{-1}$ as $T\to \infty$ in the sense of quadratic forms, we have that $Z_T \uparrow Z_\infty$ as $T\to \infty$.
Then
\begin{align}
\label{eq:difference-exp-vt-vinfty}
e^{-v_T(\phi)} - e^{-v_\infty(\phi)} 
&= \EE_{c_T}[ e^{-v_0(\phi+ \zeta)}] - \EE_{c_\infty}[e^{-v_0(\phi+ \zeta)}]
\nnb
&= \frac{\int_{X_\epsilon} e^{-v_0(\phi+\zeta)} e^{-\frac{1}{2} \zeta c_T^{-1} \zeta} d\zeta}{Z_T}
- 
\frac{\int_{X_\epsilon} e^{-v_0(\phi+\zeta)} e^{-\frac{1}{2} \zeta c_\infty^{-1} \zeta} d\zeta}{Z_\infty}
\nnb
&= \frac{
Z_\infty \int_{X_\epsilon} e^{-v_0(\phi+ \zeta)} \big(e^{-\frac{1}{2}\zeta c_T^{-1}\zeta}\!-\!e^{-\frac{1}{2}\zeta c_\infty^{-1}\zeta} \big) d\zeta
- \big( Z_T\!-\!Z_\infty\big) \int_{X_\epsilon} e^{-v_0(\phi+ \zeta) } e^{-\frac{1}{2}\zeta c_\infty^{-1} \zeta} d\zeta
}
{Z_T Z_\infty},
\end{align}
and thus, using that $c_T^{-1} \geq c_\infty^{-1}$ as quadratic forms on $X_\epsilon$
as well as $|e^{-v_0(\phi+ \zeta)} | \leq 1$,
we have
\begin{align}
\label{eq:estimate-difference-exp-vt-vinfty}
\big| e^{-v_T(\phi)} - e^{-v_\infty(\phi)} \big|
\leq 
\frac{\int_{X_\epsilon} \big| e^{-\frac{1}{2}\zeta c_T^{-1} \zeta} - e^{-\frac{1}{2}\zeta c_\infty^{-1} \zeta}   \big| d\zeta}{Z_T}
+
\frac{\big|Z_T - Z_\infty \big|}{Z_T}
=
\frac{2 \big|Z_T - Z_\infty \big|}{Z_T}.
\end{align}
The right hand side of the previous display does not depend on $\phi$ and converges to $0$ as $T\to \infty$.
Thus, we have
\begin{align}
\EE_{c_\infty - c_T} [e^{-v_T}] &= \EE_{c_\infty - c_T} [e^{-v_\infty}] + o_T(1)  ,
\\
\EE_{c_\infty - c_T} [f e^{-v_T}] &= \EE_{c_\infty - c_T} [f e^{-v_\infty}] + \| f\|_{L_\infty(X_\epsilon)} o_T(1) ,
\end{align}
where $o_T(1) \to 0$ as $T\to \infty$.
From \eqref{eq:expectation-under-nu-T} we then deduce that
\begin{equation}
\label{eq:EnuTf-oT-terms}
\\E_{\nu_T} [f] 
= \frac{\EE_{c_\infty - c_T}[f e^{-v_T}]}{\EE_{c_\infty - c_T}[ e^{-v_T}]}
=  \frac{\EE_{c_\infty - c_T}[f e^{-v_\infty}] +\|f\|_{L^\infty(X_\epsilon)} o_T(1)}{\EE_{c_\infty - c_T}[ e^{-v_\infty} ]  + o_T(1)} .
\end{equation}
Since $c_\infty -c_T \downarrow 0$ in the sense of quadratic forms on $X_\epsilon$, we have that the centred Gaussian measures with covariance $c_\infty - c_T$ converge to the Dirac measures $\delta_0$ on $X_\epsilon$.
Hence, as $T\to \infty$, we have
\begin{align}
\EE_{c_\infty - c_T}[f e^{-v_\infty}] \to  f(0) e^{-v_\infty(0)}, 
\\
\EE_{c_\infty - c_T}[e^{-v_\infty}] \to e^{-v_\infty(0)},
\end{align}
which, together with \eqref{eq:EnuTf-oT-terms}, shows that
\begin{equation}
\\E_{\nu_T} [f] \to \frac{f(0) e^{-v_\infty(0)}}{e^{-v_\infty(0)}} = f(0)
\end{equation}
as claimed.
Since $\Phi_T^T \to \Phi_\infty$ in distribution and the limiting random field is constant,
we have that this convergence is also in probability.
This concludes the proof of the convergence to $0$ of the first term in \eqref{eq:difference-phi-phiT}.

By \eqref{eq:lv-partial-dcnablav-eps-bd} and $\Phi_s \leq \Phi_s^\GFF$, which holds by Theorem \ref{thm:lv-coupling}, the third term is bounded by
\begin{equation}
\label{eq:upper-bound-third-law-of-phi-0}
\int_T^\infty \norm{\dot c_s \nabla v_s(\Phi_s)} \, ds
\leq \int_T^\infty O_{\beta,m}(\theta_s) L_s^{\beta/4\pi} e^{\sqrt{\beta} \max_{\Omega_\epsilon} \Phi_s} \, ds
\leq
\int_T^\infty O_{\beta,m}(\theta_s) L_s^{\beta/4\pi} e^{\sqrt{\beta} \max_{\Omega_\epsilon} \Phi_s^\GFF}\, ds,
\end{equation}
where the last display converges to $0$ a.s.\ as $T\to \infty$ by Corollary \ref{cor:gff-max-upper-bound-along-all-scales}.

The fourth term in \eqref{eq:difference-phi-phiT} is a Gaussian field on $\Omega_\epsilon$ with covariance matrix
$c_\infty - c_T \to 0$
as $T\to\infty$,
which, since $\Omega_\epsilon$ is finite, implies that it converges to $0$ in probability.

In summary, have shown that there is $R_T$ such that $\|R_T\| \to 0$ in probability, and
\begin{equation}
\Phi_t - \Phi_t^T
= - \int_t^T \qa{\dot c_{s} \nabla v_{s}(\Phi_s)-\dot c_{s} \nabla v_{s}(\Phi_s^T)} \, ds  + R_T.
\end{equation}
By the continuity estimate \eqref{eq:dcnablav-eps-cont}, and again using $\Phi_s \leq \Phi_s^\GFF$ and $\Phi_s^T \leq \Phi_s^{\GFF,T}\equiv \Phi_s^{\GFF,T} = \int_s^T q_s dW_s$,
\begin{equation}
  \int_t^T \norm{\dot c_{s} \nabla v_{s}(\Phi_s)-\dot c_{s} \nabla v_{s}(\Phi_s^T)} \, ds
  \leq \int_t^T  M_s \norm{\Phi_s-\Phi_s^T}  \, ds
  .
\end{equation}
where
\begin{equation}
M_t=O_\beta(\theta_t) L_t^{\beta/4\pi} \big(
e^{\sqrt{\beta} \max_{\Omega_\epsilon} \Phi_t^{\GFF,T}} + e^{\sqrt{\beta} \Phi_t^\GFF}
\big).
\end{equation}
Thus, we have shown that $D_t = \Phi_{t}-\Phi^T_{t}$ satisfies
\begin{align}
\norm{D_t}
\leq \norm{R_T} +  \int_{t}^T  M_s \norm{D_s}  \, ds.
\end{align}
Now, we have that $\cov(\Phi_t^{\GFF,T}) = c_T- c_t \leq c_\infty - c_t = \cov(\Phi_t^\GFF)$, and thus, standard results on the maxima of Gaussian processes imply that $\max_{\Omega_\epsilon}\Phi_t^{\GFF,T}$ is stochastically dominated by $\max_{\Omega_\epsilon} \Phi_t^\GFF$.
In particular, \eqref{eq:max-phi-t-eps-eventually-upper-bounded} holds with $\Phi_t^{\GFF,T}$ in place of $\Phi_t^\GFF$.
The same version of Gronwall's inequality as in the proof of Theorem~\ref{thm:lv-coupling}
implies that
\begin{equation}
\norm{D_t}
\leq \norm{R_T}\exp\pa{\int_t^T M_s \, ds}
\leq \norm{R_T}\exp\pa{\int_0^\infty M_s \, ds}
\lesssim \norm{R_T}.
\end{equation}
This bound is uniform in $t$ and
hence $\sup_{t \in [0,T]} \norm{D_t} \to 0$ in probability as $T\to\infty$.
\end{proof}

\subsection{Estimates on the difference field: Proof of Theorem \ref{thm:lv-coupling-bounds-lattice}}

In what follows, we use the following result to argue that the contribution to $\Phi_0^{\GFF_\epsilon}$ from scales $t\leq \epsilon^2$ is negligible uniformly in $\epsilon>0$.

\begin{lemma}
\label{lem:contributions-0-eps2}
Let $\beta\in (0,8\pi)$ and let $\hdphs$ be as in \eqref{eq:parameter-hoelder-phase}.
For any $\rho < (2-\hdphs)/\sqrt{\beta}$, there is a positive and a.s.\ finite random variable $K_\rho$,
such that
\begin{equation}
\label{eq:integral-0-eps2-expmax}
\int_0^{\epsilon^2} \expmax_s^\epsilon ds \lesssim K_\rho \epsilon^{2-(\hdphs + \rho \sqrt{\beta})}.
\end{equation}
Moreover, for any $\delta\in (0,2-\hdphs)$ there is a deterministic constant $C_{\beta,\delta}>0$
\begin{equation}
\label{eq:tiny-scales-expmax-eps-integrable-at-0}
\P \Big(
\exists t_0 >0 \colon \forall t\in(0, t_0) \colon \sup_{\epsilon>0}\int_0^{t\wedge \epsilon^2} \expmax_s^\epsilon ds 
\leq
C_{\beta, \delta} L_t^{\delta} 
\Big) = 1.
\end{equation}
\end{lemma}

\begin{proof}
For $\epsilon>0$ and $t\leq \epsilon^2$ and $x\in \Omega_\epsilon$, we define
\begin{equation}
s_t^\epsilon(x) = \Phi_t^{\GFF_\epsilon}(x) - \Phi_{\epsilon^2}^{\GFF_\epsilon}(x).
\end{equation}
Then we have for $t_1\leq t_2 \leq \epsilon^2$
\begin{equation}
\E[(s_{t_2}^{\epsilon}(x) - s_{t_1}^{\epsilon}(x))^2] =
\int_{t_1}^{t_2} \dot c_s^\epsilon(x^\epsilon,x^\epsilon) ds
\lesssim \frac{t_2-t_1}{\epsilon^2},
\end{equation}
where we used that $\dot c_s^\epsilon (x,x) = p_s^{1,\epsilon}(0) \lesssim \epsilon^{-2}$ for $s\leq \epsilon^2$.
By Fernique's criterion, we have, for some constant $C>0$ which does not depend on $\epsilon>0$,
\begin{equation}
\E\big[ \sup_{t\leq \epsilon^2} s_t^\epsilon(x) \big] \leq C,
\end{equation}
and moreover, $\sup_{t\leq \epsilon^2} \var\big( s_t^\epsilon(x) \big) \leq c$ for some constant $c>0$.
Then, the Borell-Tsirelson-Sudakov inequality and a union bound imply that, for $r\geq 2C$,
\begin{equation}
\P\big( \max_{\Omega_\epsilon} \sup_{t \leq \epsilon^2} s_t^\epsilon(x) \geq r \big) \lesssim \frac{1}{\epsilon^2} e^{-r^2/c}.
\end{equation}
In particular, we have, for any $\rho>0$,
\begin{equation}
\P\big( \max_{\Omega_\epsilon} \sup_{t\leq \epsilon^2} s_t^\epsilon(x) \geq \rho \log \frac{1}{\epsilon} \big)
\lesssim \frac{1}{\epsilon^2} e^{-\big(\rho \log \frac{1}{\epsilon}\big)^2/c}
\lesssim
e^{-\big(\frac{\rho}{2} \log \frac{1}{\epsilon}\big)^2/c}.
\end{equation}
Thus, by the Borel-Cantelli Lemma, we have, for any $\rho>0$,
\begin{equation}
\label{eq:gff-tiny-scales-eventually-bounded}
\P\big(\exists \epsilon_0 >0 \colon \forall \epsilon<\epsilon_0 \colon \max_{\Omega_\epsilon} \sup_{t\leq \epsilon^2}  \Phi_t^{\GFF_\epsilon} - \Phi_{\epsilon^2}^{\GFF_\epsilon} \leq \rho \log \frac{1}{\epsilon} \big) =1.
\end{equation}

To prove \eqref{eq:integral-0-eps2-expmax} we choose $\rho < (2-\hdphs)/\sqrt{\beta}$ and note that by Corollary \ref{cor:gff-max-upper-bound-along-all-scales} there is a positive and a.s.\ finite random variable $M_{\rho/2}$ such that 
\begin{equation}
\max_{\Omega_\epsilon} \Phi_{\epsilon^2}^{\GFF_\epsilon} \leq (\fom + \frac{\rho}{2}) \log \frac{1}{\epsilon} + M_{\rho/2}.
\end{equation}
Moreover, by \eqref{eq:gff-tiny-scales-eventually-bounded}, for every $\rho>0$ there is a positive and a.s.\ finite random variable $\bar M_{\rho}$, such that for all $\epsilon<1$, we have
\begin{equation}
\Phi_t^{\GFF_\epsilon} - \Phi_{\epsilon^2}^{\GFF_\epsilon} \leq \frac{\rho}{2} \log \frac{1}{\epsilon} + \bar M_{\rho/2}.
\end{equation}
It follows that
\begin{align}
\int_0^{\epsilon^2} \expmax_s^\epsilon ds 
&\lesssim
\int_0^{\epsilon^2} L_{\epsilon^2}^{\beta/4\pi} e^{\sqrt{\beta} \max_{\Omega_\epsilon} \Phi_s^{\GFF_\epsilon} } ds
\nnb
&\leq
\int_0^{\epsilon^2} L_{\epsilon^2}^{\beta/4\pi}
e^{\sqrt{\beta}\big( \max_{\Omega_\epsilon} \big( \Phi_s^{\GFF_\epsilon} - \Phi_{\epsilon^2}^{\GFF_\epsilon} \big) + \max_{\Omega_\epsilon} \Phi_{\epsilon^2}^{\GFF_\epsilon} \Big)} ds
\nnb
&\leq
e^{\sqrt{\beta} \bar M_{\rho'/2}}e^{\sqrt{\beta} M_{\rho/2}}
\epsilon^{2 - \hdphs - \sqrt{\beta}\rho} .
\end{align}
The last right hand side is uniform in $\epsilon>0$ for $\rho < (2-\hdphs)/\sqrt{\beta}$,
and thus, \eqref{eq:integral-0-eps2-expmax} follows with $K_\rho = e^{\sqrt{\beta}( \bar M_{\rho/2}+ M_\rho/2)}$.

To prove \eqref{eq:tiny-scales-expmax-eps-integrable-at-0}, we let  $\delta \in (2-\hdphs)$ and choose $\rho>0$ according to
\begin{equation}
\delta = 2-(\hdphs + \rho \sqrt{\beta}).
\end{equation}
Then we similarly have
\begin{equation}
\int_0^{\epsilon^2 \wedge t} \expmax_s^\epsilon ds \lesssim
K_\rho (\epsilon^2 \wedge t) \epsilon^{-(\hdphs +\rho \sqrt{\beta})}
\leq K_\rho (\epsilon^2 \wedge t)^{1-\frac{1}{2}(\hdphs +\rho \sqrt{\beta})}
\leq K_\rho L_t^{2-(\hdphs + \rho \sqrt{\beta})},
\end{equation}
from which \eqref{eq:tiny-scales-expmax-eps-integrable-at-0} follows.
\end{proof}

In the arguments below, we use Lemma \ref{lem:contributions-0-eps2} when estimating integrals over all scales $t\geq 0$. 
For instance, we have
\begin{align}
\label{eq:estimate-expmax-split-integral}
\int_0^\infty O(\theta_s) \expmax_s^\epsilon ds
&= \int_0^{\epsilon^2} \expmax_s^\epsilon ds
+
\int_{\epsilon^2}^\infty O(\theta_s) \expmax_s^\epsilon ds
\nnb
&\leq
\int_0^{\epsilon^2} \expmax_s^\epsilon ds
+
\int_{\epsilon^2}^\infty O(\theta_s) \sup_{\epsilon \leq L_s}\expmax_s^\epsilon ds
\nnb
&
\leq \sup_{\epsilon>0} \int_0^{\epsilon^2} \expmax_s^\epsilon ds
+
\int_0^\infty O(\theta_s) \sup_{\epsilon \leq L_s} \expmax_s^\epsilon ds.
\end{align}
Note that the last right hand side is uniform in $\epsilon>0$ and $t\geq 0$, and moreover, by Lemma \ref{lem:contributions-0-eps2} and Corollary \ref{cor:gff-max-upper-bound-along-all-scales} a.s.\ finite.

\begin{proof}[Proof of Theorem \ref{thm:lv-coupling-bounds-lattice}]
The proofs of the resuts are similar for $\epsilon=0$ and $\epsilon>0$,
when the estimates in Theorem \ref{thm:lv-dcnablav-limit}, Corollary \ref{cor:gff-max-upper-bound-along-all-scales}, Corollary \ref{cor:expmax-integrable-at-0} and Lemma \ref{lem:dc-difference} for $\epsilon=0$ by their analogues for $\epsilon>0$.
The only essential difference is that, for $\epsilon>0$, we use Lemma \ref{lem:contributions-0-eps2} to control the contributions on sclaes $t\leq \epsilon^2$.
We demonstrate this for the proof of \eqref{eq:lv-phi-delta-bd-max-continuity},
and, to keep the notation clear, prove all other statements for $\epsilon=0$. 

We first show that for any $\beta\in (0,8\pi)$ and $\rho>0$ there is an a.s.\ finite random variable $\tilde M_{\beta,\rho}$,
which does not depend on $\epsilon>0$ and $t\geq 0$
such that
\begin{equation}
\label{eq:phi-delta-uniformly-bounded}
\sup_{\epsilon>0} \| \Phi_t^{\Delta_\epsilon}  \|_{L^\infty(\Omega_\epsilon)} \leq \tilde M_{\beta,\rho}.
\end{equation}
%
Using \eqref{eq:definition-difference-field} together with Lemma \ref{lem:lv-fkg}
we have, for $x\in \Omega_\epsilon$,
\begin{align}
\label{eq:phi-delta-eps-pointwise-upper-proof}
0 \leq - \Phi_t^{\Delta_\epsilon} (x)
&= \int_t^\infty \big( \dot c_s^\epsilon \nabla v_s^\epsilon(\Phi_s^{\Lv_\epsilon}) \big)_x ds 
\leq \lambda \sqrt{\beta} \int_t^\infty \int_{\Omega_\epsilon} \dot c_s^\epsilon(x,y) \wick{e^{\sqrt{\beta} \Phi_s^{\Lv_\epsilon}(y)} }_{L_s}dy ds
\nnb
& \leq \lambda \sqrt{\beta} \int_0^\infty \int_{\Omega_\epsilon} \dot c_s^\epsilon(x,y)  \wick{ e^{\sqrt{\beta} \Phi_s^{\GFF_\epsilon}(y)} }_{L_s} dy ds.
\end{align}
In the last step, we used that $\Phi_s^{\Lv_\epsilon} \leq \Phi_s^{\GFF_\epsilon}$, which holds by Theorem \ref{thm:lv-coupling}.
By Lemma \ref{lem:C-limit}, specifically \eqref{e:c-limit} with $k=0$,
we can further estimate the last display in \eqref{eq:phi-delta-eps-pointwise-upper-proof} to obtain
\begin{align}
\label{eq:estimate-difference-field-uniformly}
0 \leq - \Phi_t^{\Delta_\epsilon} (x) 
&\lesssim_\beta \lambda \int_0^\infty
L_{s\vee \epsilon^2}^{\beta/4\pi} e^{\sqrt{\beta} \max \Phi_s^{\GFF_\epsilon}} \sup_{x\in \Omega_\epsilon} \| \dot c_s^\epsilon \|_{L^1(\Omega_\epsilon)} ds
\nnb
&\leq \lambda \int_0^\infty L_{s\vee \epsilon^2}^{\beta/4\pi} e^{ \sqrt{\beta} \max \Phi_s^{\GFF_\epsilon}}  O(\theta_s) \, ds.
\end{align}
The bound on the right hand side is uniform in $x\in \Omega_\epsilon$, and thus
\begin{align}
\| \Phi_t^{\Delta_\epsilon} \|_{L^\infty(\Omega_\epsilon)}
&\lesssim_{\beta} \lambda \int_0^\infty L_{s\vee \epsilon^2}^{\beta/4\pi} e^{\sqrt{\beta} \max \Phi_s^{\GFF_\epsilon} } O(\theta_s) ds
\nnb
&\lesssim \lambda \Big( \int_0^{\epsilon^2} \expmax_s^\epsilon  ds
+
\int_0^\infty O(\theta_s) \sup_{\epsilon \leq L_s}\expmax_s^\epsilon ds
\Big)
,
\end{align}
where we used the estimate \eqref{eq:estimate-expmax-split-integral} in the last step.
In particular,
we have
\begin{align}
\sup_{\epsilon >0} \| \Phi_t^{\Delta_\epsilon} \|_{L^\infty(\Omega_\epsilon)}
&\lesssim_\beta \lambda \Big( \sup_{\epsilon>0} \int_0^{\epsilon^2} \expmax_s^\epsilon  ds
+
 \int_0^\infty O(\theta_s) \sup_{\epsilon \leq L_s}\expmax_s^\epsilon ds
\Big)
.
\end{align}
For $\rho <(2-\hdphs)/\sqrt{\beta}$ let $M_\rho$ be as in Corollary \ref{cor:gff-max-upper-bound-along-all-scales}.
Then we have from the previous display
together with Lemma \ref{lem:contributions-0-eps2}
\begin{equation}
\sup_{\epsilon >0} \| \Phi_t^{\Delta_\epsilon} \|_{L^\infty(\Omega_\epsilon)}
\lesssim_{\beta} \lambda \Big( K_\rho +  e^{\sqrt{\beta} M_\rho} \int_{0}^\infty L_s^{-(\hdphs+\sqrt{\beta}\rho)} O(\theta_s) ds
\Big)
\end{equation}
and the last right hand side is finite a.s.\ by the choice of $\rho$.

As explained above, we keep the notation clear and only discuss the case $\epsilon=0$ from here on.
We continue with the estimates for the gradient of $\Phi^{\Delta_0}$ in the case $\hdphs<1$.
To this end, we first note that
\begin{equation}
\partial \Phi_t^{\Delta_0}
= -\int_t^\infty \big( \partial \dot c_s^0 \nabla v_s^0\big) (\Phi_s^{\Lv_0}) ds
\end{equation}
and thus, \eqref{eq:lv-dcnablav-0-upper-bound}, $\Phi_s^{\Lv_0}\leq \Phi_s^{\GFF_0}$ and Lemma \ref{lem:C-limit}, specifically \eqref{e:c-limit} now with $k=1$,
imply
\begin{align}
\|\partial \Phi_t^{\Delta_0} \|_{L^\infty(\Omega)}
&\lesssim_{\beta} \lambda  \int_0^\infty \sup_{x\in \Omega}\|\partial \dot c_s^0 (x,\cdot) \|_{L^1(\Omega_\epsilon)} 
L_{s}^{\beta/4\pi} e^{\sqrt{\beta} \max_{\Omega} \Phi_s^{\GFF_0}} ds 
\nnb
&\lesssim
\lambda \int_0^\infty O(\theta_s) L_{s}^{-1} \expmax_s^0 ds,
\end{align}
which gives similarly to before
\begin{align}
\label{eq:grad-phi-t-delta-0-bounded-linfty}
 \|\partial \Phi_t^{\Delta_0} \|_{L^\infty(\Omega)}
&\lesssim_\beta \lambda 
\int_0^\infty O(\theta_s) L_{s}^{-1} \expmax_s^0 ds
.
\end{align}
Now, since $\hdphs<1$, we choose $\rho< (1-\hdphs)/\sqrt{\beta}$ and obtain, with $M_\rho$ as in Corollary \ref{cor:gff-max-upper-bound-along-all-scales},
\begin{align}
\|\partial \Phi_t^{\Delta_0} \|_{L^\infty(\Omega)}
&\lesssim_\beta
\lambda \Big(
e^{\sqrt{\beta} M_\rho}
\int_0^\infty L_s^{-1} L_s^{-(\hdphs + \sqrt{\beta}\rho)} O(\theta_s) ds \Big)
,
\end{align}
and the last right hand side is again a.s.\ finite by the choice of $\rho$.

To prove \eqref{eq:lv-phi-delta-0-t-bound},
we follow the same reasoning to obtain, for any $x\in \Omega$ and $t\geq 0$,
\begin{align}
\big| \Phi_0^{\Delta_0} (x) - \Phi_t^{\Delta_0} (x) \big|
& \lesssim_{\beta} \lambda \int_0^t \int_{\Omega_\epsilon} \dot c_s^0(x,y)  \wick{ e^{\sqrt{\beta} \Phi_s^{\GFF_0}(y)} }_{L_s} dy ds
\nnb
&\lesssim \lambda \int_0^t L_{s}^{\beta/4\pi} e^{\sqrt{\beta} \max \Phi_s^{\GFF_0}} O(\theta_s) ds
\lesssim
\lambda
\int_{0}^t \expmax_s^0 ds ,
\end{align}
which is again uniform in $x\in \Omega$.
Therefore, we also have
\begin{align}
\| \Phi_0^{\Delta_0} - \Phi_t^{\Delta_0} \|_{L^\infty(\Omega)}
&\lesssim_\beta \lambda
\int_0^t  \expmax_s^0 ds
,
\end{align}
and thus, \eqref{eq:lv-phi-delta-0-t-bound} follows from Corollary \ref{cor:expmax-integrable-at-0}, specifically \eqref{eq:expmax-0-integrable-at-0}.
For the corresponding statement with $\epsilon>0$, we use in addition \eqref{eq:tiny-scales-expmax-eps-integrable-at-0}.

Next, we prove the H\"older continuity estimates in \eqref{eq:lv-phi-delta-bd-max-continuity}.
We first discuss the continuity of the gradient of $\Phi_t^{\Delta_0}$ in case $\hdphs<1$.
Using 
again $\Phi_s^{\Lv_0} \leq \Phi_s^{\GFF_0}$ and Lemma \ref{lem:dc-difference}, specifically \eqref{eq:diffdc-0-difference-l1},
we have for $x,y \in \Omega$,
\begin{align}
|\partial \Phi_t^{\Delta_0}(x) - \partial \Phi_t^{\Delta_0} (y) |
&\leq \lambda \sqrt{\beta}
\int_t^\infty \int_{\Omega} | \partial \dot c_s^0(x,z) - \partial \dot c_s^0 (y,z) |
\wick{e^{\sqrt{\beta} \Phi_s^{\Lv_0}(z)}}_{L_s} dzds
\nnb
&\lesssim_\beta \lambda
\int_0^\infty \int_{\Omega} | \partial \dot c_s^0(x,z) - \partial \dot c_s^0 (y,z) | dz 
L_{s}^{\beta/4\pi} e^{\sqrt{\beta} \max_{\Omega} \Phi_s^{\GFF_0}} ds
\nnb
&\lesssim \lambda  \int_0^\infty O(\theta_s)L_{s}^{-1} (1 \wedge \frac{|x-y|}{L_{s}}) \expmax_s^0 ds
.
\end{align}
We now split the second integral according to $s\in [0, |x-y|^2]$ and $s>|x-y|^2$ to obtain
\begin{align}
\label{eq:continuity-diff-phi-t-delta-eps-two-terms}
|\partial \Phi_t^{\Delta_0}(x) - \partial \Phi_t^{\Delta_0} (y) |
&\lesssim_\beta
\lambda 
\Big( \int_{0}^{|x-y|^2} O(\theta_s) L_s^{-1} \expmax_s^0 ds
\nnb
&\qquad \qquad +
|x-y|\int_{|x-y|^2}^\infty O(\theta_s) L_s^{-2} \expmax_s^0 ds 
\Big)
\nnb
&\lesssim
\lambda 
\Big(
\int_{0}^{|x-y|^2} O(\theta_s) L_s^{-1} \expmax_s^0 ds
\nnb
&\qquad \qquad +
|x-y|\int_{|x-y|^2}^\infty O(\theta_s) L_s^{-2} \expmax_s^0 ds 
\Big)
.
\end{align}
For $\hdphs<1$ we choose again $\rho <(1-\hdphs)/\sqrt{\beta}$. Then, with $M_\rho$ as in Corollary \ref{cor:gff-max-upper-bound-along-all-scales}, we have for the first integral in \eqref{eq:continuity-diff-phi-t-delta-eps-two-terms}
\begin{align}
\int_0^{|x-y|^2} O(\theta_s) L_s^{-1} \expmax_s^0 ds
&\lesssim_{\beta, \rho}
e^{\sqrt{\beta} M_\rho}
|x-y|^{-(\hdphs +1+ \sqrt{\beta} \rho) + 2}
\end{align}
while for the second integral in \eqref{eq:continuity-diff-phi-t-delta-eps-two-terms} we similarly have
\begin{align}
\int_{|x-y|^2}^\infty O(\theta_s) L_s^{-2} \expmax_s^0 ds
&\leq
e^{\sqrt{\beta}M_\rho} \int_{|x-y|^2}^{\infty} O(\theta_s) L_s^{-(2+\hdphs+ \sqrt{\beta}\rho)} ds
\nnb
&\lesssim_{\beta,\rho} 
e^{\sqrt{\beta}M_\rho} 
\big( |x-y|^{-(\hdphs+\sqrt{\beta}\rho)} - 1
\big).
\end{align}
The latter estimate shows that the total contribution of the second term in \eqref{eq:continuity-diff-phi-t-delta-eps-two-terms} is
\begin{equation}
|x-y| \int_{|x-y|^2}^\infty O(\theta_s) L_s^{-2}  \expmax_s^0 ds
\lesssim_{\beta, \rho}
\big(e^{\sqrt{\beta}M_\rho} +1\big) |x-y|^{1-(\hdphs+\sqrt{\beta}\rho)}
\end{equation}
as needed.
This concludes the proof of \eqref{eq:lv-phi-delta-bd-max-continuity} for $\hdphs<1$.

To finish the proof of \eqref{eq:lv-phi-delta-bd-max-continuity}, 
it remains to show the continuity estimate for $\Phi_t^{\Delta_0}$ in the case $\hdphs<2$.
For $x,y \in \Omega$ and $t\geq 0$ we have,
using \eqref{eq:lv-dcnablav-0-upper-bound}, the fact that $\Phi_s^{\Lv_0} \leq \Phi_s^{\GFF_0}$ and Lemma \ref{lem:dc-difference}, specifically \eqref{eq:dc-0-difference-l1}, 
\begin{align}
| \Phi_t^{\Delta_0} (x) - \Phi_t^{\Delta_0}(y)  |
&\leq 
\lambda \sqrt{\beta} \int_0^\infty 
\int_{\Omega} \big|
\dot c_s^0 (x,z) - \dot c_s^0(y,z) 
\big|
\wick{ e^{\sqrt{\beta}\Phi_s^{\Lv_0}(z)}}_{L_s} dz ds
\nnb
&\lesssim_{\beta} 
\lambda \int_0^\infty 
\int_{\Omega}
\big| \dot c_s^0 (x,z) - \dot c_s^0(y,z) 
\big| 
\wick{ e^{\sqrt{\beta} \Phi_s^{\GFF_0}(z)} }_{L_s} dz ds
\nnb
&\lesssim_{\beta}
\lambda \int_0^\infty
O(\theta_s) \big(1 \wedge \frac{|x-y|}{L_{s}}\big) 
L_{s}^{\beta/4\pi} e^{\sqrt{\beta} \max \Phi_s^{\GFF_0}}  ds
\nnb
&\lesssim
\lambda 
\int_{0}^\infty O(\theta_s) (1 \wedge \frac{|x-y|}{L_{s}}) \expmax_s^0 ds
%
 .
\end{align}
To estimate the last disply,
we split the integral again at $s=|x-y|^2$ and write
\begin{align}
\label{eq:phi-delta-at-x-y-upper-bound}
| \Phi_t^{\Delta_0} (x) - \Phi_t^{\Delta_0} (y) |
&\lesssim_\beta
\lambda \Big(
\int_{0}^{|x-y|^2}
O(\theta_s) \expmax_s^0
ds 
 +
|x-y| \int_{|x-y|^2}^\infty O(\theta_s) L_s^{-1} \expmax_s^0 ds
\Big)
.
\end{align}
Similar to before, we have, now with $\rho < (2-\hdphs)/\sqrt{\beta}$ and $M_\rho$ as in Corollary \ref{cor:gff-max-upper-bound-along-all-scales}, for the first integral
\begin{align}
\int_0^{|x-y|^2}O(\theta_s) \expmax_s^0 ds
&\lesssim e^{\sqrt{\beta} M_\rho}
\int_0^{|x-y|^2} L_s^{-(\hdphs + \sqrt{\beta}\rho)} ds
\lesssim_{\rho, \beta} 
e^{\sqrt{\beta}M_\rho} |x-y|^{2-(\hdphs + \sqrt{\beta}\rho)}.
\end{align}
Moreover, we have for the second integral in \eqref{eq:phi-delta-at-x-y-upper-bound}
\begin{align}
\int_{|x-y|^2}^\infty O(\theta_s) L_s^{-1}  \expmax_s^0  ds
&\leq 
e^{\sqrt{\beta}M_\rho}\int_{|x-y|^2}^{\infty} O(\theta_s) L_s^{-(\hdphs+1+\sqrt{\beta}\rho)} ds
\nnb
&\leq e^{\sqrt{\beta}M_\rho}\int_{|x-y|^2}^{\infty} O(\theta_s) s^{-(\hdphs+1+\sqrt{\beta}\rho)/2} ds 
\nnb
&\lesssim_{\beta, \rho}
e^{\sqrt{\beta}M_\rho} \big(
|x-y|^{1-(\hdphs +\sqrt{\beta}\rho)}
+1\big).
\end{align}
Thus, the total contribution from the second term in \eqref{eq:phi-delta-at-x-y-upper-bound} for $1\leq \hdphs <2$ is
\begin{equation}
|x-y| \int_{|x-y|^2}^\infty O(\theta_s) L_s^{-1} \expmax_s^0 ds
\lesssim_{\beta, \rho}
\big(e^{\sqrt{\beta}M_\rho}+1\big) |x-y|^{2-(\hdphs+ \sqrt{\beta}\rho)}
\end{equation}
as needed. This concludes the proof of \eqref{eq:lv-phi-delta-bd-max-continuity} for the case $1\leq \hdphs < 2$. 

It remains to prove \eqref{eq:lv-del-phi-delta-bd}.
To this end, we first note that, for $t>0$,
\begin{equation}
\partial^k \Phi_t^{\Delta_0} 
= 
- \int_t^\infty \big(\partial^k \dot c_s^0 \nabla v_s^0\big) (\Phi_s^{\Lv_0}) ds,
\end{equation}
and thus, using \eqref{eq:lv-dcnablav-0-upper-bound} and $\Phi_t^{\Lv_0} \leq \Phi_t^{\GFF_0}$ and Lemma \ref{lem:C-limit},
we have
\begin{align}
\| \partial^k \Phi_t^{\Delta_0} \|_{L^\infty(\Omega)}
&\leq \lambda \sqrt{\beta} 
\int_t^\infty \sup_{x\in \Omega} \| \partial^k \dot c_s^0(x,\cdot) \|_{L^1(\Omega)}
 \wick{ e^{\sqrt{\beta} \max_{\Omega} \Phi_s^{\Lv_0}} }_{L_t} ds
\nnb
& \lesssim_\beta \lambda \int_t^\infty \sup_{x\in \Omega} \| \partial^k \dot c_s^0(x,\cdot) \|_{L^1(\Omega)}
L_s^{\beta/4\pi} e^{\sqrt{\beta} \max_{\Omega} \Phi_s^{\GFF_0}  } ds
\nnb
&\leq \lambda \int_t^\infty L_s^{-k} O(\theta_s)  \expmax_s^0 ds, 
\end{align}
and thus, we have, for $t>0$,
\begin{equation}
\| \partial^k \Phi_t^{\Delta_0} \|_{L^\infty(\Omega)}
\lesssim_\beta \lambda \frac{1}{L_t^k} \int_t^\infty O(\theta_s) \expmax_s^0 ds
\leq \lambda \frac{1}{L_t^k} \int_t^\infty O(\theta_s) \expmax_s^0 ds
\end{equation}
where the last right hand side is a.s.\ finite by Corollary \ref{cor:gff-max-upper-bound-along-all-scales}.
\end{proof}

\subsection{Convergence of the lattice fields: Proof of Theorem \ref{thm:lv-phi-delta-t-limit}}

With the lattice fields and the continuum fields constructed, we can now address the convergence as $\epsilon \to 0$.
In the proof below we apply Theorem \ref{thm:lv-dcnablav-limit} with the random, but fixed, data $\varphi = \Phi_t^{\GFF_0}$ and $\varphi^\epsilon = \Phi_t^{\GFF_\epsilon}$.
Indeed, as shown in \cite[(2.65)]{MR4399156} we have for $t>0$ that a.s.,
\begin{equation}
\sup_{x\in \Omega} |\Phi_t^{\GFF_0}(x) - \Phi_t^{\GFF_\epsilon}(x^\epsilon)| \to 0.
\end{equation}

\begin{proof}[Proof of Theorem~\ref{thm:lv-phi-delta-t-limit}]
We first show \eqref{eq:lv-phi-delta-t-limit}, i.e., that for any fixed $t_0>0$,
\begin{equation}
\label{e:Phi-Delta-limit-bis}
\sup_{t\geq t_0}\norm{\Phi_t^{\Delta_\epsilon} - \Phi_t^{\Delta_0}}_{L^\infty(\Omega_\epsilon)}
\to 0\qquad
\text{in probability as $\epsilon \to 0$.}
\end{equation}
To this end, let $t\geq t_0$ and define $F^0$ as in \eqref{e:F-Picard-def}
and $F^\epsilon$ analogously with $\dot c_s^0\nabla v_s^0$ replaced by $\dot c_s^\epsilon \nabla v_s^\epsilon$.
Let $D = \Phi^\Delta$ be the corresponding fixed points with $E = \Phi^{\GFF}$, i.e.,
\begin{equation}
\Phi^{\Delta_\epsilon} = F^\epsilon(\Phi^{\Delta_\epsilon},\Phi^{\GFF_\epsilon}),
\qquad \Phi^{\Delta_0} = F^0(\Phi^{\Delta_0},\Phi^{\GFF_0}).
\end{equation}
Then, identifying $\Phi_t^{\Delta_0}$ with its restriction to $\Omega_\epsilon$, we have, for $t\geq t_0$,
\begin{align}
\label{e:Phi-limit-twoterms}
\Phi_t^{\Delta_0} -  \Phi_t^{\Delta_\epsilon}
&=F_t^0(\Phi^{\Delta_0}, \Phi^{\GFF_0})-F_t^\epsilon(\Phi^{\Delta_\epsilon}, \Phi^{\GFF_\epsilon})
\nnb
&= [F_t^\epsilon(\Phi^{\Delta_0}, \Phi^{\GFF_0})-F_t^\epsilon(\Phi^{\Delta_\epsilon}, \Phi^{\GFF_\epsilon})]
+ [F_t^0(\Phi^{\Delta_0}, \Phi^{\GFF_0})-F_t^{\epsilon}(\Phi^{\Delta_0},\Phi^{\GFF_0})].
\end{align}
The first term in \eqref{e:Phi-limit-twoterms} is bounded as in \eqref{e:F-Lip} by
\begin{align}
\label{eq:difference-f-eps-gff-0-f-eps-gff-eps}
\| F_t^\epsilon(\Phi^{\Delta_0},\Phi^{\GFF_0}) &-F_t^\epsilon(\Phi^{\Delta_\epsilon}, \Phi^{\GFF_\epsilon}) \|_{L^\infty(\Omega_\epsilon)}
\nnb
&= 
\| F_t^\epsilon(\Phi^{\Delta_0},\Phi^{\GFF_0})-F_t^\epsilon(\Phi^{\Delta_\epsilon}, \Phi^{\GFF_0}) \|_{L^\infty(\Omega_\epsilon)}
\nnb
&\qquad\qquad   + 
\| F_t^\epsilon(\Phi^{\Delta_\epsilon},\Phi^{\GFF_0})-F_t^\epsilon(\Phi^{\Delta_\epsilon}, \Phi^{\GFF_\epsilon}) \|_{L^\infty(\Omega_\epsilon)}
\nnb
&\lesssim_{\beta,\lambda} \int_t^\infty O(\theta_s) e^{\sqrt{\beta} \max_{\Omega_\epsilon} \Phi_s^{\GFF_0}} L_s^{\beta/4\pi} \norm{\Phi^{\Delta_0}_s-\Phi^{\Delta_\epsilon}_s}_{L^\infty(\Omega_\epsilon)} \, ds
\nnb
&\qquad \qquad +
\int_t^\infty O(\theta_s) e^{\sqrt{\beta}  \max_{\Omega_\epsilon}\Phi_s^{\Delta_\epsilon} } L_s^{\beta/4\pi} \norm{\Phi_s^{\GFF_0}-\Phi_s^{\GFF_\epsilon}}_{L^\infty(\Omega_\epsilon)} \, ds
  .
\end{align}
In \cite[Lemma 3.5]{MR4399156} it is shown that
\begin{equation}
\E \Big[\sup_{t\geq t_0}\| \Phi_t^{\GFF_0} -\Phi_t^{\GFF_\epsilon} \|_{L^\infty(\Omega_\epsilon)}^2 \Big] \to 0 
\end{equation}
as $\epsilon \to 0$.
Using this together with $\Phi_s^{\Delta_\epsilon}\leq 0$,
we have that the second term on the right hand side of \eqref{eq:difference-f-eps-gff-0-f-eps-gff-eps} converges to $0$ in probability.

To bound the second term in \eqref{e:Phi-limit-twoterms}, we write
\begin{align}
\label{eq:phi-limit-second-term}
\norm{F_t^0(\Phi^{\Delta_0}, \Phi^{\GFF_0})&-F_t^{\epsilon}(\Phi^{\Delta_0},\Phi^{\GFF_0})}_{L^\infty(\Omega_\epsilon)}
\nnb
&\leq \int_t^\infty \norm{
\dot c_s^0\nabla v_s^0(\Phi_s^{\Delta_0}+\Phi_s^{\GFF_0})
-
\dot c_s^\epsilon\nabla v_s^\epsilon(\Phi_s^{\Delta_0}+\Phi_s^{\GFF_0})}_{L^\infty(\Omega_\epsilon)} \, ds.
\end{align}
By \cite[Lemma 3.5]{MR4399156}, $\Phi_t^{\GFF_0}$ is smooth for all $t \geq t_0>0$ a.s.,
and by \eqref{eq:lv-del-phi-delta-bd},
$\Phi_t^{\Delta_0}$ is also smooth for all $t>0$.
By \eqref{eq:lv-dcnablav-0-upper-bound} the integrand of the right hand side of \eqref{eq:phi-limit-second-term} is a.s.\ bounded and decays exponentially.
Hence, using 
\eqref{eq:lv-dcnablav-limit} and dominated convergence,
we have as $\epsilon \to 0$
\begin{align}
\sup_{t \geq t_0} \norm{F_t^0(\Phi^{\Delta_0}, \Phi^{\GFF_0})-F_t^{\epsilon}(\Phi^{\Delta_0},\Phi^{\GFF_0})}_{L^\infty(\Omega_\epsilon)} \to 0 \qquad \text{a.s.}
\end{align}
In particular,
the same convergence holds true in probability.

In summary, for every $t\geq t_0$, there is a random variable $R_{t_0}^\epsilon$, which is independent of $t\geq t_0$ and converges to $0$ as $\epsilon \to 0$ in probability,
such that
\begin{equation}
\norm{\Phi^{\Delta_0}_t-\Phi^{\Delta_\epsilon}_t}_{L^\infty(\Omega_\epsilon)}
\lesssim_{\beta,\lambda} R_{t_0}^\epsilon
+  \int_t^\infty  O(\theta_s) e^{\sqrt{\beta} \max_{\Omega_\epsilon} \Phi_s^{\GFF_0}} L_s^{\beta/4\pi} \norm{\Phi^{\Delta_0}_s-\Phi^{\Delta_\epsilon}_s}_{L^\infty(\Omega_\epsilon)}  \, ds.
\end{equation}
Now, we estimate
\begin{equation}
\int_t^\infty O(\theta_s) e^{\sqrt{\beta} \max_{\Omega_\epsilon} \Phi_s^{\GFF_0}} L_s^{\beta/4\pi}
ds
\leq
\int_{t_0}^\infty O(\theta_s)
\expmax_s^0 ds , 
\end{equation}
where the right hand side is independent of $\epsilon>0$ and a.s.\ finite by Corollary \ref{cor:gff-max-upper-bound-along-all-scales}.
The same version of Gronwall's inequality as in the proof of 
Theorem~\ref{thm:lv-coupling} thus implies
\begin{equation}
\norm{\Phi^{\Delta_0}_t-\Phi^{\Delta_\epsilon}_t}_{L^\infty(\Omega_\epsilon)}
\lesssim_{\beta,\lambda}
R_{t_0}^\epsilon \int_{t_0}^\infty O(\theta_s) 
\expmax_s^0 ds.
\end{equation}
The last right hand side is unifrom in $t\geq t_0$ and thus, the convergence \eqref{e:Phi-Delta-limit-bis} follows by taking first the supremum over $t\geq t_0$ and then $\epsilon \to 0$.

We now conclude the proof of the convergence $\extepsI \Phi^{\Lv_\epsilon} \to \Phi^{\Lv_0}$ in $H^{-\kappa}(\Omega)$ for $\kappa >0$.
Since $\extepsI \Phi^{\GFF_\epsilon} \to \Phi^{\GFF_0}$ in $H^{-\kappa}(\Omega)$ for any $\kappa>0$ by \cite[Lemma 3.5]{MR4399156}, specifically \cite[(3.45)]{MR4399156},
it suffices to show that the $L^2(\Omega)$-norm of the following
right-hand side converges to $0$ in probability as $\epsilon \to 0$:
\begin{equation}
\label{eq:weak-convergence-phi-lv-0-eps}
(\Phi^{\Lv_0}_0-\extepsI \Phi^{\Lv_\epsilon}_0)
-
(\Phi^{\GFF_0}_0-\extepsI \Phi^{\GFF_\epsilon}_0)
=
(\Phi^{\Delta_0}_0-\Phi^{\Delta_0}_t)
+ \extepsI(\Phi^{\Delta_\epsilon}_t-\Phi^{\Delta_\epsilon}_0)
+
(\Phi^{\Delta_0}_t-\extepsI \Phi^{\Delta_\epsilon}_t).
\end{equation}
The first term on the right-hand side is bounded and independent of $\epsilon>0$,
and by \eqref{eq:lv-phi-delta-0-t-bound} its $L^\infty(\Omega)$-norm converges to $0$ a.s.\ as first $\epsilon\to 0$ and then $t\to 0$.
For the second term, we have
\begin{equation}
\| I_\epsilon (\Phi_t^{\Delta_\epsilon} - \Phi_t^{\Delta_0}) \|_{L^2(\Omega)}
= \| \Phi_t^{\Delta_\epsilon} - \Phi_t^{\Delta_0} \|_{L^2(\Omega_\epsilon)}
\leq \| \Phi_t^{\Delta_\epsilon} - \Phi_t^{\Delta_0} \|_{L^\infty(\Omega_\epsilon)}
\end{equation}
and the last right hand side converges to $0$ a.s.\ as $\epsilon \to 0$ again by \eqref{eq:lv-phi-delta-0-t-bound}.
To see that also the third term vanishes in the said limits, we write
\begin{align}
\| \Phi_t^{\Delta_0} - I_\epsilon \Phi_t^{\Delta_\epsilon} \|_{L^2(\Omega)} 
&\leq
\| \Phi_t^{\Delta_0} - I_\epsilon (\Phi_t^{\Delta_0}|_{\Omega_\epsilon} ) \|_{L^2(\Omega)} 
+ \| I_\epsilon (\Phi_t^{\Delta_0}|_{\Omega_\epsilon}) - I_\epsilon\Phi_t^{\Delta_\epsilon} \|_{L^2(\Omega)}
\nnb
&=
\| \Phi_t^{\Delta_0} - I_\epsilon (\Phi_t^{\Delta_0}|_{\Omega_\epsilon} ) \|_{L^2(\Omega)} 
+ \| \Phi_t^{\Delta_0} - \Phi_t^{\Delta_\epsilon} \|_{L^2(\Omega_\epsilon)}
\nnb
&\leq
\| \Phi_t^{\Delta_0} - I_\epsilon (\Phi_t^{\Delta_0}|_{\Omega_\epsilon} ) \|_{L^2(\Omega)} 
+ \| \Phi_t^{\Delta_0} - \Phi_t^{\Delta_\epsilon} \|_{L^\infty(\Omega_\epsilon)}.
\end{align}
Now, since $\Phi_t^{\Delta_0} \in C^\infty(\Omega)$ a.s., we have that the first term converges to $0$ a.s.\ as $\epsilon \to 0$.
Moreover, by \eqref{eq:lv-phi-delta-t-limit}, the second term converges to $0$ in probability as $\epsilon \to 0$.
This shows that the $L^2(\Omega)$-norm of right hand side of \eqref{eq:weak-convergence-phi-lv-0-eps} converges to $0$ in probability as first $\epsilon \to 0$ and then $t\to 0$,
which then gives
\begin{equation}
\|I_\epsilon \Phi_0^{\Lv_\epsilon} - \Phi_0^\Lv \|_{H^{-\kappa}(\Omega)} \to 0
\end{equation} 
as $\epsilon \to 0$ in probability.
\end{proof}

\section{Convergence of the maximum}

As demonstrated previously in \cite{MR4399156} in the case of the sine-Gordon field,
the most economic way to prove convergence of the centred maximum of $\Phi_0^{\Lv_\epsilon}$ is to first remove the small scales of the difference field $\Phi^{\Delta_\epsilon}$,
and then prove convergence in law for an auxiliary field, below denoted as $\tPhisLv$,
which admits an independent decomposition into a log-correlated Gaussian field and a smooth non-Gaussian field.
More precisely, we write
\begin{equation}
\label{eq:approximate-independent-decomposition}
\Phi_0^{\Lv_\epsilon} = \tPhisLv +  R_s^\epsilon,
\qquad \tPhisLv =  (\Phi_0^{\GFF_\epsilon} - \Phi_s^{\GFF_\epsilon}) + \Phi_s^{\Lv_\epsilon}, \qquad R_s^\epsilon = \Phi_0^{\Delta_\epsilon} - \Phi_s^{\Delta_\epsilon} 
\end{equation}
and note that the maximum of $R_s^\epsilon$ is bounded uniformly in $\epsilon>0$ and converges to $0$ as $s\to 0$ by Theorem \ref{thm:lv-coupling-bounds-lattice}, specifically \eqref{eq:lv-phi-delta-0-t-bound}.


The following result allows to reduce attention to $\tPhisLv$. Its proof follows similar lines as the one of \cite[Lemma 4.1]{MR4399156} the only difference being that the bound on $R_s^\epsilon$ is a.s.\ finite rather than deterministically finite.
It can be easily verified that the statement is also true under these slightly weaker assumptions, for which we omit the proof here.
We emphasise that the proof of this statement uses the tightness of the sequence $\big(\max_{\Omega_\epsilon} \sqrtpi \Phi_0^\Lv - m_\epsilon\big)_{\epsilon>0}$.
This is immediate from Theorem \ref{thm:lv-coupling} and the tightness of the sequence $(\max_{\Omega_\epsilon} \Phi_0^{\GFF_\epsilon} -m_\epsilon)_{\epsilon>0}$,
which is proved in \cite{MR2846636}.

\begin{lemma}[Analog to {\cite[Lemma 4.1]{MR4399156}}]
\label{lem:lv-reduction-to-phi-tilde}
Assume that the limiting law $\tilde \mu_s$ of $\max_{\Omega_\epsilon} \sqrtpi \tPhisLv -m_\epsilon$ as $\epsilon\to 0$ exists for every $s>0$,
and that there are a.s.\ positive random variables $\ZDM_s$ (on the above common probability space) such that
\begin{equation}   
\label{eq:mu-s-gumbel}
\tilde \mu_s((-\infty,x])=\E[e^{-\alpha^* \ZDM_s e^{-\scaling x}}]
\end{equation}
for some constant $\alpha^*>0$.
Then the law of $\max_{\Omega_\epsilon} \sqrtpi \Phi_0^\Lv - m_\epsilon$ converges weakly to
some probability measure $\mu_0$ as $\epsilon \to 0$ and $\tilde \mu_s \rightharpoonup \mu_0$ weakly as $s\to 0$.
Moreover,  there is a positive random variable
$\ZDM^\Lv$ such that
\begin{equation}  
\mu_0((-\infty,x])=\E[e^{-\alpha^* \ZDM^\Lv e^{-\scaling x}}].
\end{equation}
\end{lemma}

\begin{proof}[Proof of Theorem \ref{thm:lv-max-convergence-to-gumbel}]
Following the same steps as in \cite[Section 4]{MR4399156}, it can be shown that, for any $s>0$, $\max_{\Omega_\epsilon} \tPhisLv - m_\epsilon$ converges in distribution as $\epsilon \to 0$ to a randomly shifted Gumbel distribution with a.s.\ positive random shift $\ZDM_s$.
Indeed, the small scales part of $\tPhisLv$ is identical to the field $\tilde \Phi_s^\mathrm{SG}$, and the non-Gaussian smooth part can be dealt with in a similar way thanks to the estimates in Theorem \ref{thm:lv-coupling-bounds-lattice}.

This convergence ensures that the assumptions of Lemma \ref{lem:lv-reduction-to-phi-tilde} are satisfied, thereby completing the proof of Theorem \ref{thm:lv-max-convergence-to-gumbel}.
\end{proof}

\appendix

\section{Covariance estimates}
\label{app:covariance-estimates}

Let $p_t^\epsilon(x)$ be the heat kernel on $\epsilon\Z^2$
and $p_t^0(x)$ the continuous heat kernel on $\R^2$, i.e.,
\begin{equation}
p_t^\epsilon(x) = \epsilon^{-2}\tilde p_{t/\epsilon^2}(x/\epsilon),
\qquad
p_t^0(x) = \frac{e^{-|x|^2/4t}}{4 \pi t}
,
\end{equation}
where $\tilde p_t$ is the heat kernel on the unit lattice $\Z^2$.
Moreover, let $\partial_\epsilon f$ denote the derivative for the vector of lattice gradients
of a function $f\colon \epsilon\Z^2 \to \R$.
Thus, if $\alpha= (\alpha_1,\dots,\alpha_{|\alpha|})$
is a sequence of $|\alpha|$ unit directions in $\Z^2$,
i.e., $\alpha_i \in \{(0, \pm 1), (\pm 1, 0 )\}$,
then 
$\partial_\epsilon^\alpha=\prod_{i=1}^{|\alpha|} \partial_\epsilon^{\alpha_i}$
where $\partial_\epsilon^{\alpha_i} f(x) = \epsilon^{-1} (f(x+  \epsilon \alpha_i)-f(x))$.

\begin{lemma}
\label{lem:pt}
The heat kernel $p_t^\epsilon$ on $\epsilon\Z^2$ satisfies the following upper bounds for $t \geq \epsilon^2$, $x\in\Z^2$,
and all sequences of unit vectors $\alpha$:
\begin{equation}
\label{e:ptbounds}
|\partial_\epsilon^\alpha p_t^\epsilon(x)| \leq O_\alpha(t^{-1 -|\alpha|/2} e^{-c|x|/\sqrt{t}}).
\end{equation}
Moreover, for all $x \in \epsilon\Z^2$ and $k \geq 4$,
\begin{align}
\label{e:ptlimit}
|p_t^\epsilon(x)- p_t^0(x)|
&\leq O(\frac{1}{t}) \times \frac{\epsilon^2}{t}
\qa{
\pa{(\frac{|x|}{\sqrt{t}})^k+1} e^{-|x|^2/t}
+ (\frac{\epsilon^2}{t})^{(k-3)/2}
}.
\end{align}
\end{lemma}

Moreover, we write $p_t^{\epsilon,L}$ for the heat kernel on a torus of side length $L$ and mesh size $\epsilon>0$,
which can be expressed as
\begin{equation}
\label{e:pttorus}
p_t^{\epsilon,L}(x) = \sum_{y\in\Z^d} p_t^\epsilon(x+yL).
\end{equation}
Note that we have $x,y\in \Omega_\epsilon$
\begin{equation}
\dot c_t^\epsilon(x,y) = p_t^{\epsilon,1}(x-y),
\end{equation}
where the difference of $x,y\in \Omega_\epsilon$ is understood coordinatewise and modulo the length of the torus.
The representation \eqref{e:pttorus} allows to transfer estimates for $p_t^{\epsilon}$ to $p_t^{\epsilon,L}$.

\begin{lemma}
\label{lem:pttorus}
The torus heat kernel satisfies, for $t\geq \epsilon^2$,  $|x|_\infty \leq L/2$, 
\begin{equation} 
\label{e:pttorusbounds}
\partial_\epsilon^\alpha p_t^{\epsilon,L}(x)
= \partial_\epsilon^\alpha p_t^\epsilon(x)
+ O_\alpha(t^{-|\alpha|/2}  L^{-d} e^{-cL/\sqrt{t}} ).
\end{equation}
Moreover, for any $t>0$, as $\epsilon \to 0$,
\begin{equation}
\label{e:pttoruslimit}
\sup_{x\in \Omega_\epsilon} |p_t^{\epsilon,L}(x)-p_t^{0,L}(x)|
\lesssim
\p{t^{-1}+L^{-2}}
\frac{\epsilon^2}{t} \log(\frac{\epsilon^2}{t})^2
.
\end{equation}
\end{lemma}

\begin{lemma}[Identical to Lemma 2.3 of \cite{MR4399156}] 
\label{lem:C-limit}
There exists a function $\gamma_t=\gamma_t(m)$ of $(t,m) \in [0,\infty) \times (0,\infty)$
with $\gamma_t = \gamma_0 + O(m^2t)$ as $t\to 0$ such that
\begin{equation}
\label{e:c0-limit}
e^{-\frac12 c_t^{\epsilon}(0,0)}\epsilon^{-1/4\pi} \to  \gamma_t L_t^{-1/4\pi} .
\end{equation}
For all $x \neq y \in \Omega$ and $t>0$, the integral $c_t^0(x,y) = \int_0^t \dot c_s^0(x,y) \, ds$ exists and
uniformly on compact subsets of $x\neq y$,
\begin{equation}
\label{e:cxy-limit}
c_t^\epsilon(x^\epsilon,y^\epsilon) \to c_t^0(x,y).
\end{equation}
For all $t \geq \epsilon^2$, and all $k \in \N$,
\begin{equation}
\label{e:c-limit}
L_t^k \sup_x \norm{\partial_\epsilon^k \dot c_t^\epsilon(x,\cdot)}_{L^1(\Omega_\epsilon)}  
\leq O_k(\theta_t),
\qquad
\sup_{x}
\norm{E_\epsilon \dot c_t^\epsilon(x,\cdot)-\dot c_t^0(x,\cdot)}_{L^1(\Omega)}
\leq O(\frac{\epsilon^2}{t}) \theta_t.
\end{equation}
\end{lemma}

The following estimates are then a consequence of Lemma \ref{lem:pttorus} and Lemma \ref{lem:C-limit}.

\begin{lemma}
\label{lem:dc-difference}
We have for $t\geq \epsilon^2$ and $x,y\in \Omega_\epsilon$
\begin{align}
\label{eq:dc-difference-l1}
\int_{\Omega_\epsilon}
\big|
\dot c_t^\epsilon (x,z) - \dot c_t^\epsilon(y,z) 
\big|
dz
&\lesssim
\big(1\wedge \frac{|x-y|}{L_t} \big) \theta_t,
\\
\label{eq:diffdc-difference-l1}
\int_{\Omega_\epsilon}
\big|
\partial_\epsilon \dot c_t^\epsilon (x,z) - \partial_\epsilon \dot c_t^\epsilon(y,z) 
\big|
dz
&\lesssim
\frac{1}{L_t}\big(1\wedge \frac{|x-y|}{L_t} \big) \theta_t.
\end{align}
Similarly, we have for $t>0$ and $x,y \in \Omega$
\begin{align}
\label{eq:dc-0-difference-l1}
\int_{\Omega}
\big|
\dot c_t^0 (x,z) - \dot c_t^0(y,z) 
\big|
dz
&\lesssim
\big(1\wedge \frac{|x-y|}{L_t} \big) \theta_t,
\\
\label{eq:diffdc-0-difference-l1}
\int_{\Omega}
\big|
\partial \dot c_t^0 (x,z) - \partial \dot c_t^0(y,z) 
\big|
dz
&\lesssim
\frac{1}{L_t}\big(1\wedge \frac{|x-y|}{L_t} \big) \theta_t.
\end{align}
\end{lemma}

\section{Change of the regularised Gaussian fields in $\epsilon>0$}
\label{app:eps-change-of-gff}

In this section we give the remaining proof of Lemma \ref{lem:gaussian-fields-eps-change}, thereby completing the proof of \eqref{eq:max-phi-t-eps-eventually-upper-bounded}.

\begin{proof}[Proof of Lemma \ref{lem:gaussian-fields-eps-change}]
Let $\hat q_s^\epsilon(k)$ for $k\in \Omega_\epsilon^*$ be the Fourier multipliers of $(\dot c_s^\epsilon)^{1/2}$, i.e.,
\begin{equation}
\label{eq:fourier-coefficients-dc-upper}
\hat q_s^{\epsilon} (k) = e^{-\frac{s}{2} (- \hat \Delta^{\epsilon}(k) + m^2 )  } ,
\end{equation}
where we recall that for $k\in \Omega_\epsilon^*$ the Fourier multipliers of $-\Delta^\epsilon$ are given by
\begin{equation}
-\hat \Delta^\epsilon (k) = \sum_{i=1}^2 \epsilon^{-2} \big( 2 - 2\cos(\epsilon k_i) \big) .
\end{equation}
It can be shown that, for $k\in \Omega_\epsilon^*$,
\begin{equation}
0 \leq \hat q_s^{\epsilon}(k) \leq e^{-\frac{1}{2} (\kappa|k|^2 +m^2)} ,
\end{equation}
where $\kappa = 4/\pi^2$.
For a proof of this estimate, we refer to \cite[(2.59)]{MR4399156}.
Then, we have, for $\epsilon_1 \leq \epsilon_2 \leq L_t$ and $x\in \Omega$,
\begin{align}
\label{eq:gff-difference-eps}
\rem_t(x,\epsilon_2) - \rem_t(x,\epsilon_1) 
&= \sum_{k \in \Omega_{\epsilon_2}^*}
\Big(
e^{ikx^{\epsilon_2}} \int_t^\infty \hat q_s^{\epsilon_2} (k)d\hat W_s(k)
- e^{ikx^{\epsilon_1}} \int_t^\infty \hat q_s^{\epsilon_1} (k) d\hat W_s(k) 
\Big)
\nnb
&\qquad -
\sum_{ k\in \Omega_{\epsilon_1}^*\setminus \Omega_{\epsilon_2}^*} e^{ikx^{\epsilon_1}}
\int_t^\infty \hat q_s^{\epsilon_1} (k) d\hat W_s(k)
\nnb
&=\sum_{k \in \Omega_{\epsilon_2}^*}\big( e^{ikx^{\epsilon_2}} - e^{ikx^{\epsilon_1}} \big) \int_t^\infty \hat q_s^{\epsilon_2} (k) d \hat W_s(k)
\nnb
&\qquad + \sum_{k \in \Omega_{\epsilon_2}^*}
e^{ikx^{\epsilon_1}}
\Big(
\int_t^\infty \hat q_s^{\epsilon_2} (k) d \hat W_s(k)
-  \int_t^\infty \hat q_s^{\epsilon_1} (k) d\hat W_s(k) 
\Big)
\nnb
&\qquad -
\sum_{k\in \Omega_{\epsilon_1}^*\setminus \Omega_{\epsilon_2}^*}
e^{ikx^{\epsilon_1}}\int_t^\infty \hat q_s^{\epsilon_1} (k) d \hat W_s(k) .
\end{align}
Thus, using the estimate $|e^{ikx^{\epsilon_2}} - e^{ikx^{\epsilon_1}}|\leq |k|(\epsilon_2-\epsilon_1)$,
we have
\begin{align}
\label{eq:var-gaussian-field-different-eps}
\E \big[  \big( \rem_t(x,\epsilon_2) - \rem_t(x,\epsilon_1) \big)^2 \big]
& \lesssim 
\sum_{k\in \Omega_{\epsilon_2}^*} |k|^2 (\epsilon_2 - \epsilon_1)^2 \int_t^\infty \big| \hat q_s^{\epsilon_2} (k) \big|^2 ds
\nnb
&\qquad + 
\sum_{k\in \Omega_{\epsilon_2}^*} \int_t^\infty \big| \hat q_s^{\epsilon_2} (k) - \hat q_s^{\epsilon_1}(k) \big|^2 ds
\nnb
&\qquad + \sum_{k \in \Omega_{\epsilon_1}^* \setminus \Omega_{\epsilon_2}^*} \int_t^\infty | \hat q_s^{\epsilon_1}(k) |^2 ds 
\end{align}
In what follows, we estimate the three sums separately.
For the first sum, we have
\begin{align}
\label{eq:estimate-first-sum}
\sum_{k\in \Omega_{\epsilon_2}^*} |k|^2 (\epsilon_2 - \epsilon_1)^2 \int_t^\infty \big| \hat q_s^{\epsilon_2} (k) \big|^2 ds
&\leq
(\epsilon_2-\epsilon_1)^2 \sum_{k\in \Omega_{\epsilon_2}^*}
|k|^2 \frac{1}{\kappa |k|^2 + m^2} e^{-t(\kappa |k|^2 + m^2)}
\nnb
&\lesssim (\epsilon_2-\epsilon_1)^2 \sum_{k\in \Omega^*} e^{-t(\kappa |k|^2 +m^2)} 
\lesssim \frac{(\epsilon_2-\epsilon_1)^2}{t} .
\end{align}
For the second sum, we first note that for $\epsilon_1 \leq \epsilon_2 \leq L_t$ and $k\in \Omega_{\epsilon_2}^*$, 
\begin{align}
\label{eq:difference-laplacian-fourier-multipliers-different-eps}
0 \leq -\hat \Delta^{\epsilon_1}(k) + \hat \Delta^{\epsilon_2}(k)
&=
\sum_{i=1}^2\big(
 \epsilon_1^{-2} (2 - \cos(\epsilon_1 k_i))
- \epsilon_2^{-2} (2 - \cos(\epsilon_2 k_i)) 
\big)
\nnb
&= \sum_{i=1}^2 |k_i|^2 (f(\epsilon_1 k_i) - f(\epsilon_2 k_i) \big),
\end{align}
with $f(x) = x^{-2} \big( 2-2\cos(x)\big) )$.
Note that $f$ is symmetric, decreasing in $|x|$ and continuously differentiable on $\R$.
In particular, $f$ is Lipschitz continuous on a the compact set $[-\pi, \pi]$.
Thus, we have for $x_1,x_1 \in [-\pi, \pi]$
\begin{equation}
|f(x_1) - f(x_2) | \lesssim |x_1 - x_2|.
\end{equation}
Applying this to \eqref{eq:difference-laplacian-fourier-multipliers-different-eps} with $x_1 = \epsilon_1 k_i$ and $x_2 = \epsilon_2 k_i$ gives
\begin{align}
0 \leq -\hat \Delta^{\epsilon_1}(k) + \hat \Delta^{\epsilon_2}(k)
&\lesssim \sum_{i=1}^2 |k_i|^2 |\epsilon_2 k_i - \epsilon_1 k_i|
= \sum_{i=1}^2 |k_i|^3 (\epsilon_2  - \epsilon_1)
\lesssim |k|^3 (\epsilon_2 - \epsilon_1).
\end{align}
Thus, we obtain
\begin{align}
0 \leq \hat q_s^{\epsilon_2} (k) - \hat q_s^{\epsilon_1}(k)
&=
e^{-\frac{s}{2} (-\hat \Delta^{\epsilon_2}(k) + m^2)} - e^{-\frac{s}{2} (-\hat \Delta^{\epsilon_1}(k) + m^2)}
\nnb
&\leq e^{-\frac{s}{2}(- \hat \Delta^{\epsilon_2}(k) + m^2 ) }
\big( 1 - e^{-\frac{s}{2} \big(- \hat \Delta^{\epsilon_1} (k)  + \hat \Delta^{\epsilon_2}(k) \big) } \big)
\nnb
&\leq
e^{-\frac{s}{2}(- \hat \Delta^{\epsilon_2}(k) + m^2 ) }
\frac{s}{2}
\big(
-\hat \Delta^{\epsilon_1} (k)  + \hat \Delta^{\epsilon_2}(k) 
\big)
\nnb
&\lesssim 
e^{-\frac{s}{2}(\kappa |k|^2 +m^2) } |k|^3 s (\epsilon_2 - \epsilon_1)
.
\end{align}
It follows that, for $k\in \Omega_{\epsilon_2}^*$,
\begin{align}
\label{eq:int-difference-fourier-coefficients-eps}
\int_t^\infty | \hat q_s^{\epsilon_2}(k) - \hat q_s^{\epsilon_1 } |^2 ds
\lesssim 
(\epsilon_2 - \epsilon_1)^2 |k|^6 \int_t^\infty  s^2 e^{-s(\kappa |k|^2 + m^2 ) } ds .
\end{align}
%
Now, explicit integration gives
\begin{equation}
\int_t^\infty s^2 e^{-s (\kappa |k|^2+m^2)} ds
= 
e^{-t (\kappa |k|^2+m^2)}
\Big(
\frac{t^2}{\kappa |k|^2+m^2} + \frac{2t}{(\kappa |k|^2+m^2)^2} + \frac{2}{(\kappa |k|^2+m^2)^3}
\Big) .
\end{equation}
Thus, summing \eqref{eq:int-difference-fourier-coefficients-eps} over $k\in \Omega_{\epsilon_2}^*$ amounts to three terms,
which are estimated by
\begin{align}
\sum_{k\in \Omega_{\epsilon_2}^*} |k|^6 e^{-t (\kappa |k|^2+m^2)}\frac{2}{(\kappa |k|^2+m^2)^3}
\lesssim
\sum_{k\in \Omega^*} e^{-t (\kappa |k|^2+m^2)}
\lesssim 
\frac{1}{t} 
\end{align}
and
\begin{align}
\sum_{k\in \Omega_{\epsilon_2}^*} |k|^6 e^{-t (\kappa |k|^2+m^2)} \frac{t^2}{\kappa |k|^2+m^2}
&\lesssim
t^2\sum_{k\in \Omega^*} |k|^4 e^{-t (\kappa |k|^2+m^2)}
\lesssim \frac{1}{t}
\end{align}
and similarly
\begin{align}
\sum_{k\in \Omega_{\epsilon_2}^*} |k|^6e^{-t (\kappa |k|^2+m^2)}\frac{2t}{(\kappa |k|^2+m^2)^2}
\lesssim t \sum_{k\in \Omega^*}
|k|^2 e^{-t \kappa |k|^2}
\lesssim
\frac{1}{t} .
\end{align}
Therefore, we get in total for the second sum in \eqref{eq:var-gaussian-field-different-eps}
\begin{equation}
\label{eq:estimate-second-sum}
\sum_{k\in \Omega_{\epsilon_2}^*}
\int_t^\infty | \hat q_s^{\epsilon_2}(k) - \hat q_s^{\epsilon_1 } |^2 ds
\lesssim
\frac{(\epsilon_2-\epsilon_1)^2}{t} .
\end{equation}
To bound the third sum in \eqref{eq:var-gaussian-field-different-eps}, we use \eqref{eq:fourier-coefficients-dc-upper} and note that
\begin{equation}
\int_t^\infty |\hat q_s^{\epsilon_1}(k)|^2 ds
\leq \int_t^\infty e^{-s(\kappa |k|^2 +m^2 )} ds
\lesssim
\frac{1}{|k|^2 + m^2} e^{-t(\kappa |k|^2+m^2)} .
\end{equation}
Therefore, we have
\begin{align}
\label{eq:estimate-third-sum}
\sum_{k \in \Omega_{\epsilon_1}^*\setminus \Omega_{\epsilon_2}^*} 
\int_t^\infty |\hat q_s^{\epsilon_1}(k)|^2 ds
&\lesssim
\sum_{k \in \Omega_{\epsilon_1}^*\setminus \Omega_{\epsilon_2}^*} \frac{1}{ |k|^2 + m^2} e^{-t(\kappa |k|^2+m^2)}
\lesssim \int_{\pi/\epsilon_2}^{\pi/\epsilon_1} \frac{1}{r} e^{-t \kappa r^2} dr
\nnb
&= \int_{\epsilon_1/\sqrt{t\kappa}\pi}^{\epsilon_2/\sqrt{t\kappa}\pi} s^{-1} e^{-\frac{1}{s^2}} ds 
\lesssim \frac{\epsilon_2 - \epsilon_1}{\sqrt{t}}.
\end{align}

Putting \eqref{eq:estimate-first-sum}, \eqref{eq:estimate-second-sum} and \eqref{eq:estimate-third-sum} together,
and using that 
$(\epsilon_2- \epsilon_1)^2/t \leq (\epsilon_2 - \epsilon_1)/L_t$ for $\epsilon_1\leq \epsilon_2\leq L_t$,
we arrive at
\begin{align}
\label{eq:estimate-cov-of-rem}
\E \big[  \big( \rem_t(x,\epsilon_2) - \rem_t(x,\epsilon_1) \big)^2 \big]
\lesssim 
 \frac{(\epsilon_2- \epsilon_1)^2}{t}
+\frac{(\epsilon_2-\epsilon_1)^2}{t}
+\frac{\epsilon_2- \epsilon_1}{\sqrt{t}}
\lesssim \frac{\epsilon_2 - \epsilon_1}{L_t}
\end{align}
as claimed.

It remains to prove the estimate \eqref{eq:relation-eps-gs-to-non-eps-gs}.
To this end, we let $x,y\in \bx{t}{i}$ and note that we can relate the spatial differences of $\bGs{i}_t$ and $\Gs{i}_t$ by 
\begin{align}
\bGs{i,\epsilon_2}_t(x) - \bGs{i,\epsilon_1}_t(y)
&= 
\big( \rem_t(\epsilon_2,x) - \rem_t(\epsilon_1,x) \big)
- \big( \rem_t(\epsilon_2,z_i) - \rem_t(\epsilon_1,z_i)  \big)
\nnb
&\qquad  + \Gs{i,\epsilon_1}_t(x^{\epsilon_1}) - \Gs{i,\epsilon_1}_t(y^{\epsilon_1}).
\end{align}
Thus, using \eqref{eq:estimate-cov-of-rem}, we obtain
\begin{align}
\E[\big( \bGs{i,\epsilon_2}_t(x) - \bGs{i,\epsilon_1}_t(y)\big)^2]
&\lesssim
\E[ \big( \rem_t(\epsilon_2,x) - \rem_t(\epsilon_1,x) \big)^2 ]
+ \E[ \big( \rem_t(\epsilon_2,z_i) - \rem_t(\epsilon_1,z_i)  \big)^2 ]
\nnb
&\qquad +\E[ \big( \Gs{i,\epsilon_1}_t(x^{\epsilon_1}) - \Gs{i,\epsilon_1}_t(y^{\epsilon_1}) \big)^2] 
\nnb
&\lesssim 
\frac{\epsilon_2 - \epsilon_1}{L_t} + 
\E[ \big( \Gs{i,\epsilon_1}_t(x^{\epsilon_1}) - \Gs{i,\epsilon_1}_t(y^{\epsilon_1}) \big)^2]
\end{align}
as claimed.
\end{proof}

\section*{Acknowledgements}

We thank Roland Bauerschmidt for comments on the project in early stages.
We  further thank Christophe Garban, Colin Guillarmou and R\'emi Rhodes for helpful discussions on the Liouville model,
as well as Francesco De Vecchi for additional comments on the results.

This work was supported by Israel Science Foundation grant number 615/24. 
Part of the work on this paper was carried out while both authors were guests at the Hausdorff Institute of Mathematics;
we thank the institute for the hospitality.

\bibliography{lvmax.bbl}

\begin{thebibliography}{10}

\bibitem{MR4124523}
S.~Albeverio, F.~C. De~Vecchi, and M.~Gubinelli.
\newblock Elliptic stochastic quantization.
\newblock {\em Ann. Probab.}, 48(4):1693--1741, 2020.

\bibitem{MR4415393}
S.~Albeverio, F.~C. De~Vecchi, and M.~Gubinelli.
\newblock The elliptic stochastic quantization of some two dimensional
  {E}uclidean {QFT}s.
\newblock {\em Ann. Inst. H. Poincar\'{e} Prob. Statist.}, 57(4):2372--2414,
  2021.

\bibitem{MR395578}
S.~Albeverio and R.~H{\o}egh-Krohn.
\newblock {\em The {W}ightman axioms and the mass gap for strong interactions
  of exponential type in two-dimensional space-time}.
\newblock Mathematics Preprint Series, No. 12 (1973). Universitetet i Oslo,
  Matematisk Institutt, Oslo, 1973.

\bibitem{MR4672110}
N.~Barashkov and M.~Gubinelli.
\newblock On the variational method for {E}uclidean quantum fields in infinite
  volume.
\newblock {\em Probab. Math. Phys.}, 4(4):761--801, 2023.

\bibitem{MR4665719}
N.~Barashkov, T.~S. Gunaratnam, and M.~Hofstetter.
\newblock Multiscale coupling and the maximum of {$\mathcal P(\phi)_2$} models
  on the torus.
\newblock {\em Comm. Math. Phys.}, 404(2):833--882, 2023.

\bibitem{MR4303014}
R.~Bauerschmidt and T.~Bodineau.
\newblock Log-{S}obolev inequality for the continuum sine-{G}ordon model.
\newblock {\em Comm. Pure Appl. Math.}, 74(10):2064--2113, 2021.

\bibitem{MR4798104}
R.~Bauerschmidt, T.~Bodineau, and B.~Dagallier.
\newblock Stochastic dynamics and the {P}olchinski equation: an introduction.
\newblock {\em Probab. Surv.}, 21:200--290, 2024.

\bibitem{MR4399156}
R.~Bauerschmidt and M.~Hofstetter.
\newblock Maximum and coupling of the sine-{G}ordon field.
\newblock {\em Ann. Probab.}, 50(2):455--508, 2022.

\bibitem{MR3652040}
N.~Berestycki.
\newblock An elementary approach to {G}aussian multiplicative chaos.
\newblock {\em Electron. Commun. Probab.}, 22:Paper No. 27, 12 pp., 2017.

\bibitem{MR4043225}
M.~Biskup.
\newblock Extrema of the two-dimensional discrete {G}aussian free field.
\newblock In {\em Random graphs, phase transitions, and the {G}aussian free
  field}, volume 304 of {\em Springer Proc. Math. Stat.}, pages 163--407.
  Springer, Cham, 2020.

\bibitem{BiskupHuang2023DGModel}
M.~Biskup and H.~Huang.
\newblock A limit law for the maximum of subcritical {D}{G}-model on a
  hierarchical lattice.
\newblock 2023.
\newblock Preprint, arxiv: 2309.09389.

\bibitem{MR3509015}
M.~Biskup and O.~Louidor.
\newblock Extreme local extrema of two-dimensional discrete {G}aussian free
  field.
\newblock {\em Comm. Math. Phys.}, 345(1):271--304, 2016.

\bibitem{MR3787554}
M.~Biskup and O.~Louidor.
\newblock Full extremal process, cluster law and freezing for the
  two-dimensional discrete {G}aussian free field.
\newblock {\em Adv. Math}, 330:589--687, 2018.

\bibitem{MR3433630}
M.~Bramson, J.~Ding, and O.~Zeitouni.
\newblock Convergence in law of the maximum of the two-dimensional discrete
  {G}aussian free field.
\newblock {\em Comm. Pure Appl. Math}, 69(1):62--123, 2016.

\bibitem{MR2846636}
M.~Bramson and O.~Zeitouni.
\newblock Tightness of the recentered maximum of the two-dimensional discrete
  {G}aussian free field.
\newblock {\em Comm. Pure Appl. Math.}, 65(1):1--20, 2012.

\bibitem{MR0450480}
H.~J. Brascamp and E.~H. Lieb.
\newblock On extensions of the {B}runn-{M}inkowski and {P}r\'{e}kopa-{L}eindler
  theorems, including inequalities for log concave functions, and with an
  application to the diffusion equation.
\newblock {\em J. Funct. Anal.}, 22(4):366--389, 1976.

\bibitem{MR3729618}
J.~Ding, R.~Roy, and O.~Zeitouni.
\newblock Convergence of the centered maximum of log-correlated {G}aussian
  fields.
\newblock {\em Ann. Probab.}, 45(6A):3886--3928, 2017.

\bibitem{MR4515694}
H.~Duminil-Copin, A.~Rivera, P.-F. Rodriguez, and H.~Vanneuville.
\newblock Existence of an unbounded nodal hypersurface for smooth {G}aussian
  fields in dimension $d\geq 3$.
\newblock {\em Ann. Probab.}, 51(1):228--276, 2023.

\bibitem{MR0266263}
X.~Fernique.
\newblock Int\'egrabilit\'e des vecteurs gaussiens.
\newblock {\em Comptes Rendus de l'Acad\'emie des Sciences},
  (270):A1698--A1699, 1970.

\bibitem{MR4054101}
C.~Garban.
\newblock Dynamical {L}iouville.
\newblock {\em J. Funct. Anal.}, 278(6):108351, 54 pp., 2020.

\bibitem{MR3406823}
M.~Gubinelli, P.~Imkeller, and N.~Perkowski.
\newblock Paracontrolled distributions and singular {PDE}s.
\newblock {\em Forum Math. Pi}, 3:e6, 75, 2015.

\bibitem{MR3274562}
M.~Hairer.
\newblock A theory of regularity structures.
\newblock {\em Invent. Math}, 198(2):269--504, 2014.

\bibitem{MR292433}
R.~H{\o}egh-Krohn.
\newblock A general class of quantum fields without cut-offs in two space-time
  dimensions.
\newblock {\em Comm. Math. Phys.}, 21:244--255, 1971.

\bibitem{Hofstetter2021Extremal}
M.~Hofstetter.
\newblock Extremal process of the sine-{G}ordon field.
\newblock {\em Ann. Inst. H. Poincar\'{e} Prob. Statist.}, (to appear), 2021.
\newblock Preprint, arXiv:2111.04842.

\bibitem{HofstetterZeitouniDecay2024}
M.~Hofstetter and O.~Zeitouni.
\newblock Decay of correlations for the massless {L}iouville model in infinite
  volume.
\newblock 2024.
\newblock Preprint, arXiv:2408.16649.

\bibitem{MR4238209}
M.~Hoshino, H.~Kawabi, and S.~Kusuoka.
\newblock Stochastic quantization associated with the {$\exp(\Phi)_2$}-quantum
  field model driven by space-time white noise on the torus.
\newblock {\em J. Evol. Equ.}, 21(1):339--375, 2021.

\bibitem{MR4528973}
M.~Hoshino, H.~Kawabi, and S.~Kusuoka.
\newblock Stochastic quantization associated with the {$\exp(\Phi)_2$}-quantum
  field model driven by space-time white noise on the torus in the full
  {$L^1$}-regime.
\newblock {\em Probab. Theory Related Fields}, 185(1-2):391--447, 2023.

\bibitem{Hu2009Minimal}
Y.~Hu and Z.~Shi.
\newblock Minimal position and critical martingale convergence in {B}ranching
  {R}andom {Walks}, and directed polymers on disordered trees.
\newblock {\em Ann. Probab.}, 37(2):742--789, 2009.

\bibitem{MR829798}
J.-P. Kahane.
\newblock Sur le chaos multiplicatif.
\newblock {\em Ann. Sci. Math. Qu\'{e}bec}, 9(2):105--150, 1985.

\bibitem{MR4324379}
T.~Oh, T.~Robert, and Y.~Wang.
\newblock On the parabolic and hyperbolic {L}iouville equations.
\newblock {\em Comm. Math. Phys.}, 387(3):1281--1351, 2021.

\bibitem{MR665603}
L.~D. Pitt.
\newblock Positively correlated normal variables are associated.
\newblock {\em Ann. Probab.}, 10(2):496--499, 1982.

\bibitem{BerestyckiPowell2025GFFLv}
E.~Powell and N.~Berestycki.
\newblock Gaussian free field and {L}iouville quantum gravity.
\newblock 2024.

\bibitem{MR3274356}
R.~Rhodes and V.~Vargas.
\newblock Gaussian multiplicative chaos and applications: a review.
\newblock {\em Probab. Surv.}, 11:315--392, 2014.

\bibitem{MR3475456}
A.~Shamov.
\newblock On {G}aussian multiplicative chaos.
\newblock {\em J. Funct. Anal.}, 270(9):3224--3261, 2016.

\bibitem{MR3526836}
O.~Zeitouni.
\newblock Branching random walks and {G}aussian fields.
\newblock In {\em Probability and statistical physics in {S}t. {P}etersburg},
  volume~91 of {\em Proc. Sympos. Pure Math.}, pages 437--471. Amer. Math.
  Soc., Providence, RI, 2016.

\end{thebibliography}

\end{document}